\documentclass[11pt]{amsart}
\usepackage[utf8]{inputenc}
\usepackage{amsmath,amssymb,amsthm,colonequals,mathrsfs,mathtools}
\usepackage{array,enumitem,yfonts}
\usepackage[alphabetic]{amsrefs}
\usepackage[all,cmtip]{xy}
\usepackage[colorlinks,anchorcolor=blue,citecolor=blue,linkcolor=blue,urlcolor=blue,bookmarksdepth=2]{hyperref}
\usepackage{multirow}
\allowdisplaybreaks

\usepackage{comment}
\usepackage{makecell}

\usepackage[margin=1in]{geometry}

\usepackage{tikz}
\usetikzlibrary{cd,arrows}
\tikzset{>=stealth}
\tikzcdset{arrow style=tikz}
\tikzset{link/.style={column sep=1.8cm,row sep=0.16cm}}

\AtBeginDocument{%
	\def\MR#1{}
}
\usepackage{fancyhdr}
\usepackage{amsmath}
\usepackage{mathdots}
\usepackage{amssymb}
\usepackage{amsthm}
\usepackage{here}
\usepackage{amscd} 
\usepackage{mathrsfs}
\usepackage{mathtools}
\usepackage{scalefnt}
\usepackage{url}
\theoremstyle{plain}
\newtheorem{thm}{Theorem}[section]
\newtheorem{lem}[thm]{Lemma}

\newtheorem{cor}[thm]{Corollary}
\newtheorem{prop}[thm]{Proposition}

\theoremstyle{definition}

\newtheorem{rem}[thm]{Remark}
\newtheorem{claim}[thm]{Claim}
\newtheorem{defn}[thm]{Definition}

\newtheorem{ex}[thm]{Example}

\numberwithin{equation}{section}

\def\Q{{\mathbb Q}}
\def\R{{\mathbb R}}

\def\P{{\mathbb P}}

\def\SL{\mathop{\mathrm{SL}}\nolimits}

\def\K{\mathrm{K}}

\def\L{\mathscr{L}}

\newcommand{\G}{\mathbb G}

\def\PGL{\mathop{\mathrm{PGL}}\nolimits}

\allowdisplaybreaks[4]

\usepackage{comment}

\newcommand{\marg}[1]{\normalsize{{\color{red}\footnote{#1}}}{\marginpar[\vskip -.3cm {\color{red}\hfill\tiny{\thefootnote}\normalsize$\implies$}]{\vskip -.3cm{ \color{red}$\impliedby$\tiny\thefootnote}}}}

\newcommand{\Yota}[1]{\marg{\color{blue}(Yota): #1}}

\AtBeginDocument{%
	\def\MR#1{}
}
\usepackage{fancyhdr}
\fancypagestyle{titlepage}
{
	\fancyhf{}

\fancyfoot[l]{
\href{https://mathscinet.ams.org/mathscinet/msc/msc2020.html}
	}
    }

\begin{document}
\title[K-moduli wall crossing for quasimaps to a projective variety]{K-moduli wall crossing for quasimaps to a projective variety}
\author{Masafumi Hattori$^{1}$\and Yota Maeda$^{2,3}$}
\email{hattori.masafumi.8n@kyoto-u.ac.jp,\and y.maeda.math@gmail.com}

\def\l@subsection{\@tocline{2}{0pt}{2.3pc}{5pc}{}}

\maketitle
\thispagestyle{titlepage}
\vspace{-1em}
\begin{center}
  \begin{minipage}{0.9\textwidth}
    \centering
    {\small
    $^{1}$ Department of Mathematics, Kyoto University, Japan\\
    $^{2}$ Fachbereich Mathematik, Technische Universität Darmstadt, Germany\\
    $^{3}$ Mathematical Institute, Tohoku University, Japan\\
    }
  \end{minipage}
\end{center}

\vspace{1em}

\begin{abstract}
We develop a modular wall crossing theory for quasimaps
to a projective variety, allowing independent variation
of the boundary coefficients and the quasimap weight.
Building on the K-stability of quasimaps introduced by
Hashizume and the first author, we construct projective
moduli spaces in the stable, Calabi--Yau, and log Fano
regimes, together with wall crossing morphisms.
A central construction is the moduli theory of boundary
polarized Calabi--Yau quasimaps, which retains an ample
polarization at the numerically trivial locus and allows
comparison with suitable perturbations toward the stable
and log Fano regions.
The stable theory applies in arbitrary genus, while
the comparisons through the Calabi--Yau locus concern
genus zero.
For degree-one boundary divisors, the resulting framework
relates weighted stable maps and quasimaps to Hassett
spaces and GIT quotients of weighted points on
$\mathbb P^1$.
In a companion paper, we apply this framework to give
a modular interpolation between Miranda's GIT
compactification of rational elliptic surfaces and
the Baily--Borel compactification of an eight-dimensional
ball quotient.
\end{abstract}

\section{Introduction} 
\subsection{Moduli of curves, maps, and quasimaps}
Moduli spaces of curves and maps admit several natural compactifications, depending on the stability conditions imposed on degenerations.
For moduli problems associated with curves, basic examples include the Deligne--Mumford compactification, which stabilizes the underlying curve \cite{DM}, and the moduli spaces of Hassett, which retain weighted marked points \cite{Hassett}. For moduli spaces of maps, the Kontsevich compactification retains a morphism by allowing the source curve to acquire additional components \cite{kontsevich1995enumeration,abramovich2002compactifying}. Compactifications using \emph{quasimaps to a projective space}, viewed as maps $q\colon C\rightarrow [\mathbb A^{N+1}/\mathbb G_m]$, provide another possibility: instead of retaining every component carrying positive mapping degree, one allows base points and contracts certain unstable components \cite{givental1998mirror,marian2011moduli,Toda}. 

Varying the stability conditions provides a way to compare these compactifications.
Weighted stable maps admit reduction morphisms as the
marking weights decrease. When the resulting weighted-curve
data are admissible, forgetting the map and stabilizing
the source also gives a morphism to Hassett's moduli spaces \cite{Hassett,Al-Guy,bayer2009stability}. On the quasimap side, variation of the quasimap weight interpolates between stable maps and compactifications allowing base points \cite{Toda}. Meanwhile, for rational curves, variation of weighted pointed curves is closely related to variation of GIT (VGIT) quotients of point configurations on $\P^1$; see, for example, \cite{Hassett,thaddeus1996geometric,DolgachevHu1998VGIT}. Parker \cite{parker2007elementary} also constructed GIT quotients of marked points together with polynomial data and related them to moduli spaces of genus-zero stable maps to projective space.

These comparisons lead to the problem of varying the
boundary coefficients and the quasimap weight while
giving a modular interpretation to the resulting
compactifications and their wall crossings.
We construct such a theory using K-stability of quasimaps.
In genus zero, it relates stable maps and quasimaps to
GIT quotients of weighted points on $\P^1$.

In the companion work \cite{HM2}, this theory gives a
modular interpolation between Miranda's GIT
compactification of rational elliptic surfaces and
the Baily--Borel compactification of an associated
eight-dimensional ball quotient.

\subsection{K-stability of quasimaps and modular interpolation}
K-stability, originally motivated by K\"ahler geometry
\cite{tian1997kahler,donaldson2002scalar}, is an algebraic
stability condition for polarized varieties expressed
through numerical invariants of test configurations,
that is, $\mathbb G_m$-equivariant one-parameter degenerations.
Hashizume and the first author introduced K-stability
for quasimaps in \cite{HH2}, incorporating boundary
and quasimap data into these invariants.
Building on this notion, we develop a moduli theory
with several independently weighted boundary divisors
and compare the resulting spaces across the stable,
Calabi--Yau, and log Fano regions.

Fix a closed immersion $\iota\colon X\hookrightarrow\mathbb P^N$
of a projective variety, and consider a weighted quasimap
$q\colon(C,\vec a\cdot\vec D)\to
[\operatorname{Cone}(X)/\mathbb G_m]$.
Here, $C$ is a proper connected nodal curve, $\operatorname{Cone}(X)$ is the affine cone associated
with $\iota$, $\vec D=(D_1,\ldots,D_r)$ is a tuple of
effective divisors,
and $\vec a=(a_1,\ldots,a_r)$ is the vector of coefficients.
In families, we allow effective relative boundary divisors
of prescribed degrees, not only marked sections.
Let $\mathscr L_q$ be the line bundle associated with $q$.
For a quasimap weight $w$, consider the adjoint divisor
\[
A_{\vec a,w}:=K_C+\vec a\cdot\vec D+w\mathscr L_q.
\]
For a K-semistable quasimap, this divisor is numerically
proportional to the polarization used to define K-stability
(Lemma~\ref{lem:well-balanced}).

In the log Fano region, where $-A_{\vec a,w}$ is ample,
we consider K-semistability with respect to $-A_{\vec a,w}$.
We construct the corresponding moduli stacks and their
projective good moduli spaces, called K-moduli spaces,
together with wall-crossing morphisms as $\vec a$ and $w$ vary.

In the stable region, where $A_{\vec a,w}$ is ample,
weighted stable quasimaps are K-stable with respect to
$A_{\vec a,w}$ (Proposition~\ref{prop:chara--K-stability--qmaps}).
We also construct their projective coarse moduli spaces
and reduction morphisms in arbitrary genus.
When $w>1$, the semi-log-canonical condition excludes base
points; with marked sections as boundary data, this recovers
weighted stable maps.

To compare these two regions in genus zero, we introduce
boundary polarized Calabi--Yau (CY) quasimaps at the locus
$A_{\vec a,w}\equiv0$, following the theory of boundary
polarized CY pairs \cite{babwild,BL24}.
Since the total adjoint divisor is numerically trivial,
we require a specified part of the boundary and quasimap
contributions to give an ample polarization.
These quasimaps are K-semistable with respect to this
polarization (Proposition~\ref{prop:chara--K-stability--qmaps}).
We construct their projective good moduli spaces.
For suitable small perturbations toward the stable and
log Fano regions, the corresponding moduli stacks admit
canonical open immersions into the moduli stacks of boundary polarized
CY quasimaps. These comparisons yield morphisms from the
corresponding stable moduli spaces to the log Fano K-moduli
spaces, in general after seminormalization.

Finally, on the face $w=0$, we forget the quasimap structure
and obtain moduli problems for weighted curves.
For degree-one boundary divisors, this recovers Hassett's
moduli spaces in the stable region and the GIT quotients
of weighted point configurations on $\mathbb P^1$ in the
log Fano region. In the latter case, K-stability agrees
with GIT stability, and varying the boundary coefficients
recovers VGIT (Theorem~\ref{thm:VGIT-K-stability}).

\subsection{Main results} 

We first state the moduli and wall-crossing results in the
stable and log Fano regions, and then compare the two theories
using boundary polarized CY quasimaps.

We take $d,d_i\in\mathbb Z_{>0}$,
$\vec a\in[0,1]^r$, and $w\in\mathbb R_{\geq0}$.
Fix the embedding $\iota\colon X\hookrightarrow\mathbb P^N$,
the quasimap degree $d$, weight $w$, the arithmetic genus $p_a$, and the boundary degrees
$\vec d=(d_1,\ldots,d_r)$ of $\vec{D}$.
As before, the quasimap structure is forgotten when $w=0$,
and a boundary divisor is forgotten when its coefficient is zero.

\begin{thm}[Weighted stable quasimaps and wall crossing] \label{mainthm:stable-region}
Suppose $2p_a-2+\vec a\cdot\vec d+wd>0$.
There exists a proper Deligne--Mumford stack $\mathcal M^{\mathrm{stqmaps}}_{p_a,d,\vec a,\vec d,w,\iota}$ parameterizing weighted stable quasimaps with the fixed numerical data. It admits a projective coarse moduli space  $M^{\mathrm{stqmaps}}_{p_a,d,\vec a,\vec d,w,\iota}$. 
Moreover, if $\vec a\leq\vec b$ and $ w\leq w'$, then there is a canonical reduction morphism $\mathcal M^{\mathrm{stqmaps}}_{p_a,d,\vec b,\vec d,w',\iota} \to \mathcal M^{\mathrm{stqmaps}}_{p_a,d,\vec a,\vec d,w,\iota}$.
\end{thm} 
Theorem \ref{mainthm:stable-region} follows from Theorems \ref{thm:constksbaqmaps}, \ref{thm:projectivity-stqmaps}, and \ref{thm:st-qmaps-wall-crossing}.

\begin{thm}[K-moduli of log Fano quasimaps and wall crossing] \label{mainthm:log-fano-region}
Suppose that $p_a=0$ and $\vec a\cdot\vec d+wd<2$.
Then, there exists an Artin stack $\mathcal M^{\mathrm{Kssqmaps}}_{d,\vec a,\vec d,w,\iota}$ of finite type parameterizing K-semistable log Fano quasimaps with the fixed numerical data. It admits a projective good moduli space  $M^{\mathrm{Kpsqmaps}}_{d,\vec a,\vec d,w,\iota}$. Moreover, the $(\vec a,w)$-parameter space $[0,1]^r\times\R_{\geq0}$ admits a finite rational polyhedral decomposition.
For any face $F$ of this decomposition, let $F^\circ$
denote its relative interior.
If $(\vec a,w)\in F^\circ$ and $(\vec a',w')\in \overline{F}$
satisfy $\vec a\cdot\vec d+wd<2$ and
$\vec a'\cdot\vec d+w'd<2$, then there is a canonical
wall crossing morphism
$\mathcal M^{\mathrm{Kssqmaps}}_{d,\vec a,\vec d,w,\iota}\to \mathcal M^{\mathrm{Kssqmaps}}_{d,\vec a',\vec d,w',\iota}$. If further $\vec{a},\vec{a}'\in (0,1]^r$ and $ww'\ne0$, then the canonical morphism is an open immersion. 
\end{thm}

Theorem \ref{mainthm:log-fano-region} follows from Theorems \ref{thm:K-moduli--weighted--log-Fano-qmaps} and \ref{thm:wallcrossingmorphKss}.

As in the case of log pairs (cf.~\cite{fujiki1990moduli,tian1997kahler}), the CM line bundle for
quasimaps is closely related to K-stability \cite{HH2,CM26}.
For rational boundary coefficients and quasimap weight,
it is defined functorially for families of weighted quasimaps.
We prove that it descends to an ample $\Q$-line bundle
on each of the moduli spaces in Theorems~\ref{mainthm:stable-region} and~\ref{mainthm:log-fano-region}
(Theorems~\ref{thm:K-moduli--weighted--log-Fano-qmaps} and~\ref{thm:projectivity-stqmaps}).

We next construct the moduli spaces used to compare the
stable and log Fano theories.
 When $p_a=0$, to describe the remaining wall $\vec a\cdot\vec d+wd-2=0$, split the boundary and quasimap contributions into two parts $(\vec a_1,\vec{d}_1,\vec D_1,w_1)$ and  $(\vec a_2,\vec{d}_2,\vec D_2,w_2)$ of $(\vec{a},\vec{d},\vec{D},w)$,  and impose \[ K_C +\vec a_1\cdot\vec D_1 +\vec a_2\cdot\vec D_2 +(w_1+w_2)\L_q \equiv0, \] while requiring $\vec a_2\cdot\vec D_2+w_2\L_q$ to be ample. The latter is the polarization retained on the Calabi--Yau wall. 

\begin{thm}[Moduli of boundary polarized CY quasimaps] \label{mainthm:bpcy-region} Fix boundary and quasimap degree data and weights $(\vec a_1,\vec a_2,w_1,w_2)$ satisfying the above Calabi--Yau and polarization conditions. Then, there exists an Artin stack $\mathcal M^{\mathrm{bpcyqmaps}}_ {d,\vec a_1,\vec d_1,\vec a_2,\vec d_2,w_1,w_2,\iota}$ of finite type parameterizing boundary polarized CY quasimaps with fixed numerical data. Furthermore, it admits a projective good moduli space  $M^{\mathrm{bpcyqmaps}}_ {d,\vec a_1,\vec d_1,\vec a_2,\vec d_2,w_1,w_2,\iota}$.  \end{thm} 
Theorem \ref{mainthm:bpcy-region} follows from Lemma \ref{lem:bpcyw>11}, Proposition \ref{prop:bpcy--moduli--qmaps--const}, Theorem \ref{thm:s-complete-bpcyqmaps}, and Corollary \ref{cor:proj:bpcy}.

In the boundary polarized CY case, the corresponding natural line bundle
is the Hodge line bundle, which plays the role of the
CM line bundle.
For rational weights with $w_1+w_2\leq1$, it is defined
functorially for families of boundary polarized CY quasimaps.
We prove that it descends to an ample $\Q$-line bundle
on the good moduli space in Theorem~\ref{mainthm:bpcy-region}
(Lemma~\ref{lem:CM-descent} and Theorem~\ref{thm:proj--CM--bpCY}).

The Calabi--Yau weight locus itself admits a finite rational polyhedral decomposition, and variation of the weights gives canonical reduction morphisms as in Theorem \ref{thm--wall-crossing-bpcy}.

 The three regions are connected by reduction morphisms.  \begin{thm}[Comparison of stable and log Fano K-moduli] \label{mainthm:CY-wall-crossing} Fix boundary polarized CY data $(\vec a_1,\vec a_2,w_1,w_2)$ as above. For every sufficiently small $\epsilon>0$, the perturbation $\bigl( (1-\epsilon)(\vec a_1,\vec a_2), (1-\epsilon)(w_1+w_2) \bigr)$ defines log Fano data and induces a canonical open immersion \[ \mathcal M^{\mathrm{Kssqmaps}}_{d,(1-\epsilon)(\vec a_1,\vec a_2),\vec d,(1-\epsilon)(w_1+w_2),\iota} \hookrightarrow \mathcal M^{\mathrm{bpcyqmaps}}_{d,\vec a_1,\vec d_1,\vec a_2,\vec d_2,w_1,w_2,\iota}, \] 
 which induces an isomorphism of the seminormalizations of good moduli spaces $$(M^{\mathrm{Kpsqmaps}}_{d,(1-\epsilon)(\vec a_1,\vec a_2),\vec d,(1-\epsilon)(w_1+w_2),\iota} )^{\mathrm{sn}}\cong(M^{\mathrm{bpcyqmaps}}_{d,\vec a_1,\vec d_1,\vec a_2,\vec d_2,w_1,w_2,\iota})^{\mathrm{sn}}.$$
On the other hand, there exists a canonical open immersion $$\mathcal{M}^{\mathrm{stqmaps}}_{0,d,(\vec{a}_1,(1+\epsilon)\vec{a}_2),\vec{d},w_1+(1+\epsilon)w_2,\iota}\hookrightarrow\mathcal M^{\mathrm{bpcyqmaps}}_{
  d,\vec a_1,\vec d_1,\vec a_2,\vec d_2,w_1,w_2,\iota}$$ for any sufficiently small $\epsilon>0$. 
Consequently, after seminormalization, there is a canonical
morphism $({M}^{\mathrm{stqmaps}}_{0,d,(\vec{a}_1,(1+\epsilon)\vec{a}_2),\vec{d},w_1+(1+\epsilon)w_2,\iota})^{\mathrm{sn}}\to (M^{\mathrm{Kpsqmaps}}_{d,(1-\epsilon)(\vec a_1,\vec a_2),\vec d,(1-\epsilon)(w_1+w_2),\iota} )^{\mathrm{sn}}$.
  \end{thm}

  Theorem \ref{mainthm:CY-wall-crossing} follows from Theorems \ref{thm:wall-crossingbpcytoK-moduli} and \ref{thm:final-wall-crossing-st-kst}. When $\iota=\operatorname{id}_{\mathbb P^N}$, the comparison
morphism is defined without taking seminormalizations (see Remark \ref{rem:final}).
  
\subsection{Construction of the wall crossing morphisms} 

In the stable region, we construct reduction morphisms by
moving through successive walls and contracting the rational
components on which the adjoint divisor $A_{\vec a,w}$ on the wall has
degree zero. After reducing to rational weights, these
contractions are defined over the base by relative linear
systems. 

When the quasimap weight remains positive, the quasimap
degree on a contracted rational tail is retained as
base-point multiplicity, as in Toda's wall crossing for
stable quotients \cite[Proposition~2.21]{Toda}.
The rational bridges contracted in this case have quasimap
degree zero. To obtain morphisms of moduli stacks, the
resulting boundary divisors and quasimap data must be
constructed in families, compatibly with base change.

For a contraction $h\colon C\to C'$ over $S$ arising in
these reductions, Proposition~\ref{prop:pushfoward} uses Koll\'ar's divisorial
support \cite{kollar-moduli} to define the transforms of the retained
boundary divisors as relative Mumford divisors on $C'$.
Proposition~\ref{prop:push--linear-equivalence} shows that this operation is compatible
with the rational functions defining linear equivalences.
We apply this to the zero divisors of suitable sections
defining the quasimap family and use their ratios to construct
the new line bundle and sections on $C'$.
For the positive quasimap weight $w$, Proposition~\ref{prop:contraction-thm} gives a
locally stable family of quasimaps of the same degree.
The construction is independent of the auxiliary choices
and commutes with base change, including over nonreduced
bases. Together with the contractions above, this gives
the reduction morphisms in Theorem~\ref{thm:st-qmaps-wall-crossing}.

In the log Fano region, no contraction of the source is
needed. For fixed degrees, the numerical K-semistability
criterion involves only finitely many linear inequalities
in the weights (Proposition~\ref{prop:chmber--Kss}). These inequalities allow
us to construct the morphisms of Theorem~\ref{thm:wallcrossingmorphKss} under
specialization to a face by changing the weights and,
when necessary, forgetting data whose weights become zero.

Both constructions are used for wall crossing within the
boundary polarized CY locus (Theorem~\ref{thm--wall-crossing-bpcy}).
When the set of zero boundary coefficients and quasimap
weights is unchanged, the morphism is an open immersion
obtained by changing the weights on the same families.
In general, we may also need to contract components on which
the polarization part
$\vec a'_2\cdot\vec D_2+w'_2L_q$
for the new weights $\vec{a}_1',\vec a'_2,w'_1$ and $w'_2$ has degree zero.
If $w'_1+w'_2>0$, the quasimap on the contracted family
is constructed by the method described above; otherwise,
the quasimap data are forgotten.
The open immersions into the moduli stacks of boundary polarized CY quasimaps in Theorem~\ref{mainthm:CY-wall-crossing}
are obtained by changing the weights without contracting
the source.

  \subsection{Relation with existing moduli theories} 

  Weighted stable maps and their reduction morphisms were
developed by Alexeev--Guy and Bayer--Manin \cite{Al-Guy,bayer2009stability},
building on Hassett's theory of weighted pointed curves
\cite{Hassett}.
Marian--Oprea--Pandharipande introduced stable quotients, maps to quotient stacks
associated with Grassmannian varieties \cite{marian2011moduli}.
Toda generalized their stability condition by introducing
$\epsilon$-stable quotients and studied their wall crossing 
\cite{Toda}.

Givental and Morrison--Plesser introduced compactifications
of spaces of rational maps from $\P^1$ to toric varieties
to study Gromov--Witten invariants \cite{givental1998mirror,morrison1995summing}.
Ciocan-Fontanine--Kim--Maulik generalized this approach,
together with the stable-quotient construction, to a broad
class of GIT quotients, using maps to the associated quotient
stacks to formulate the quasimap moduli problem \cite{ciocan2014stable}.
For such targets, Ciocan-Fontanine--Kim further incorporated
independently weighted markings and constructed the
corresponding moduli stacks of stable quasimaps
\cite[Section~2.2]{ciocan2013big}.
For wall-crossing formulas for virtual classes of these moduli stacks of quasimaps
and their relation to mirror symmetry, see \cite{ciocan2020quasimap}.
Our Theorem \ref{mainthm:stable-region} concerns effective relative
boundary divisors of prescribed degrees, projectivity of the moduli spaces, and the reduction
morphisms obtained by decreasing both their coefficients
and the quasimap weight.

The foundations of KSBA moduli theory were laid in the
pioneering work of Koll\'ar--Shepherd-Barron on moduli of
surfaces of general type and of Alexeev on boundedness
of log surfaces \cite{kollar1988threefolds,alexeev1994boundedness}.
Through subsequent contributions by many authors, this
developed into a theory of higher-dimensional projective
slc pairs with ample log canonical divisor, with projective
coarse moduli spaces for fixed numerical data;
see \cite{kollar-moduli}.
More recently, Ascher--Bejleri--Inchiostro--Patakfalvi
established a wall-crossing theory for higher-dimensional
stable pairs under an assumption that the generic pair is klt, constructing
reduction and wall-crossing morphisms as the boundary
coefficients vary \cite{ascher2023ksba}.
Meng--Zhuang subsequently extended this theory to families
with log canonical general fibers \cite{meng2023mmp}.

Alexeev also extended Kontsevich's moduli theory of stable maps
to maps from surfaces \cite{V}.
Dervan--Ross introduced K-stability for maps between
polarized varieties and constructed projective moduli
spaces of canonically polarized stable maps in arbitrary
dimension, extending Alexeev's surface case \cite{dervan2019stable}.
  
Through contributions by many authors in algebraic and
differential geometry, K-moduli theory, the moduli theory of K-semistable
log Fano pairs has been established, with projective good
moduli spaces for fixed rational boundary coefficients
and numerical data; see \cite{xu2025k}.
Liu--Zhou constructed proper good moduli spaces for K-semistable
log Fano pairs with real boundary coefficients
\cite{liu2024k}.
Ascher--DeVleming--Liu established a wall-crossing framework
for K-moduli of $\Q$-Gorenstein smoothable log Fano pairs whose boundary
is proportional to the anticanonical divisor \cite{ascher2024wall},
and applied it to moduli of plane curves and quartic K3
surfaces \cite{ascher2024wall,AscherDeVlemingLiu2023QuarticK3}.
This wall-crossing theory was subsequently generalized
in the proportional case \cite{zhou2023wall,zhou2026shape} and extended
to the non-proportional case under additional hypotheses
\cite{liu2026nonproportional,zhou2025finiteness}.
  
Zhou \cite{zhou2024log} established a general comparison theorem
between K-moduli and GIT for divisors in $|-mK_X|$ on a
fixed K-polystable Fano variety $X$.
Under a log canonicity assumption, the corresponding
GIT quotient is identified with the K-moduli space for
sufficiently small positive boundary coefficients,
that is, in the first chamber in the K-moduli wall crossing. This theory is also generalized to the non-proportional case \cite{liu2026nonproportional}.

Hashizume--Hattori introduced K-stability for quasimaps
and constructed projective K-moduli spaces in the log
Fano quasimap case \cite{HH2}.
Theorem~1.2 extends this moduli theory to several
independently weighted real boundary divisors and
constructs wall-crossing morphisms as their coefficients
and the quasimap weight vary.
Parker \cite{parker2007elementary} constructed GIT quotients
of weighted point configurations on $\P^1$ together with
polynomial data defining quasimaps to projective space.
For projective-space targets and degree-one boundary
divisors, his GIT (semi)stability criterion agrees with
our K-(semi)stability criterion after matching the weights.
Our construction therefore recovers these GIT quotients
as K-moduli spaces of log Fano quasimaps.

In the setting of plane curve pairs, Ascher--DeVleming--Liu
conjectured a moduli space at the Calabi--Yau wall relating
their K-moduli spaces to Hacking's compactification of moduli of plane curves via KSBA theory \cite{Hac04}
and giving a common ample model of the Hodge line bundle
\cite[Conjecture~1.8]{ascher2024wall}.
Ascher--Bejleri--Blum--DeVleming--Inchiostro--Liu--Wang
introduced boundary polarized CY pairs
\cite{babwild}, established valuative properties of their
moduli stacks, and constructed projective moduli spaces
for degenerations of plane curve pairs.
Blum--Liu constructed projective asymptotically good
moduli spaces for boundary polarized CY surface
pairs \cite{BL24}.
In the present paper, the objects also retain the quasimap
data and its base points. We construct projective good
moduli spaces of boundary polarized CY quasimaps from
curves, study their variation with the weights, and
compare them with the moduli spaces of suitable stable
and log Fano perturbations.

  \subsection{An application to the moduli space of rational elliptic surfaces} A particularly rich application of this framework comes from the moduli space of rational Weierstrass fibrations. A rational Weierstrass fibration can be represented by a Weierstrass form over $\P^1$, and its discriminant is a divisor of degree $12$ on the base. Miranda constructed a projective compactification by applying GIT to Weierstrass fibrations \cite{Miranda1981}. The discriminant also determines a configuration of twelve points on $\P^1$ and hence a point of the corresponding GIT quotient. By the work of Deligne--Mostow and Mostow, this quotient is closely related to a nine-dimensional complex ball quotient \cite{DeligneMostow1986,Mostow1986generalized}.
  The discriminants arising from rational Weierstrass fibrations, however, form a distinguished eight-dimensional locus inside this nine-dimensional moduli space. 
Heckman--Looijenga proved that the normalization of its closure
is the Baily--Borel compactification of an eight-dimensional
ball quotient \cite[Theorem~9.2]{HeckmanLooijenga2002}.
The natural birational correspondence between this
compactification and Miranda's compactification does not
extend to a morphism in either direction.

A rational elliptic surface with reduced discriminant $D$
also determines a degree-$12$ map $J\colon\P^1\to\P^1$.
At a point of the base with smooth fiber, its value is
the normalized $j$-invariant of that fiber.
Its pole divisor is precisely $D$, so that
$D=J^{-1}(\infty)$ consists of twelve distinct points.
Heckman--Looijenga used stable limits of these maps,
with the points of $D$ retained as twelve unordered
markings, to construct a compactification admitting
birational morphisms to both Miranda's and the
Baily--Borel compactifications
\cite[Section~3.2]{HeckmanLooijenga2002}.
They also describe these comparison morphisms
explicitly in terms of the degrees of $J$ on the components
of the source and the distribution of the marked points
\cite[Section~3.2]{HeckmanLooijenga2002}.
Using the K-moduli interpretations of the target
compactifications described below, we recover these
morphisms by composing the reduction morphisms for weighted
stable quasimaps with the comparison to log Fano K-moduli
through the boundary polarized CY locus
(Theorems~\ref{mainthm:stable-region} and \ref{mainthm:CY-wall-crossing}).
Thus, their geometric constructions admit an interpretation
as moduli wall crossing from stable quasimaps to log Fano
quasimaps.


In the companion work \cite{HM2}, we apply K-moduli wall crossing
for log Fano quasimaps to the degree-$12$ quasimaps encoding the
modular $J$-map and the discriminant divisor $D$.
We vary the boundary coefficient and quasimap weight as
$(1-t)/12$ and $t/12$, respectively, for $0\leq t\leq1$.
Let $\mathcal M_t$ denote the normalization of the closure
of the rational elliptic surface locus in the corresponding
K-moduli space.
For $0<t<1$, the moduli problem retains both the
discriminant divisor $D$, with coefficient $(1-t)/12$,
and the degree-$12$ quasimap $q$, with weight $t/12$.
Thus each intermediate model has a modular interpretation
in terms of K-semistable log Fano quasimaps, with stability
varying with $t$.
At end points, $\mathcal M_0$ (resp. $\mathcal M_1$) recovers the Baily--Borel
compactification (resp. Miranda's GIT compactification).
The rational map defined by the discriminant is thereby
described through the following sequence of modular
wall crossings, shown schematically for
$0<\epsilon\ll1$:
\[
\begin{tikzcd}[column sep=3em]
\mathcal M_1
&
\mathcal M_{1-\epsilon}
  \arrow[l, "\text{forget }D"']
  \arrow[r, dashed, "\text{wall crossings}"]
&[2em]
\mathcal M_{\epsilon}
  \arrow[r, "\text{forget }q"]
&
\mathcal M_0 .
\end{tikzcd}
\]
The present paper supplies the general moduli theory and wall crossing morphisms used in this construction.

\subsection{Organization of the paper}
Section~\ref{section:prelimI} collects the necessary background on stacks,
locally stable families, K-stability, boundary polarized
CY pairs, quasimaps, and CM line bundles.
It also establishes the results used to construct boundary
divisors and quasimap data on the families contracting rational tails and bridges.
In Section~\ref{sec:Kssqmapswallcrossing}, we construct projective K-moduli spaces of
log Fano quasimaps with several weighted boundary divisors
and study their chamber decomposition and wall-crossing
morphisms.
Section~\ref{sec:bpcyqmapswallcrossing} develops the moduli theory of boundary polarized
CY quasimaps. We construct their projective good moduli
spaces, prove ampleness of the Hodge line bundle, and
study wall crossing within the CY locus and comparisons
with K-moduli stacks of log Fano quasimaps.
We also compare the Hodge and CM line bundles.
Finally, Section~\ref{sec:KSBAqmapswallcrossing} constructs the moduli spaces of weighted
stable quasimaps in arbitrary genus and proves their
projectivity. We construct reduction
morphisms and, in genus zero, compare the moduli
spaces of stable quasimaps with the moduli spaces of boundary polarized CY quasimaps and K-moduli of log Fano quasimaps.

\subsection*{Acknowledgment} The authors thank Kenta Hashizume and Yuki Koto for valuable discussions.
The authors used ChatGPT (OpenAI) during the preparation of this manuscript for language and expository editing. The final constructions, computations, and mathematical arguments were independently checked by the authors, who take full responsibility for the contents of the paper.
M.H.~was supported by Royal Society International Collaboration Award ICA\textbackslash1\textbackslash231019. Y.M.~is partially supported by the Alexander von Humboldt Foundation through a Humboldt Research Fellowship, by Deutsche Forschungsgemeinschaft (DFG, German Research Foundation) through the Collaborative Research Centre TRR 326 \emph{Geometry and Arithmetic of Uniformized Structures}, project number 444845124, and by Waseda University Grant for Special Research Projects (Project number: 2026R-022).

\section{Preliminaries}\label{section:prelimI}

\subsection*{Notations and Conventions}
We work over an algebraically closed field $\mathbb{K}$ of characteristic zero. 
We collect the following terminologies, which we will use in this paper.

\begin{enumerate}
    \item We say that $V$ is a {\it variety} over $\mathbb{K}$ if $V$ is a reduced and irreducible scheme quasi-projective over $\mathbb{K}$.
    If we say that $C$ is a nodal curve, then $C$ is a reduced and separated one dimensional scheme of finite type over $\mathbb{K}$ with at worst nodal singularities. 
    \item Let $X$ denote a locally Noetherian scheme and $\mathscr{K}_X$ denote the sheaf of total quotient rings (see \cite[II, Section 6]{Ha}).  
    \item Let $K$ be a subset of $\mathbb{R}$.
    For any $\vec{x}=(x_1,\ldots,x_n)\in K^n$ and $\vec{y}=(y_1,\ldots,y_m)\in K^m$, we set $(\vec{x},\vec{y}):=(x_1,\ldots,x_n,y_1,\ldots,y_m)\in K^{n+m}$.
    If $n=m$, then let $\vec{x}\cdot\vec{y}$ denote $\sum_{i=1}^nx_iy_i$.
    We regard $K^0=\emptyset$ and $\vec{x}\in K^0$ means that $\vec{x}=0$.
    \item Let $X$ be a scheme or a stack.
    Let $\Omega$ be an algebraically closed field over $\mathbb{K}$.
    A morphism $\bar{x}\colon\mathrm{Spec}\,\Omega\to X$ is called a {\it geometric point} and we will write it as $\bar{x}\in X$.
    \item Let $f\colon \mathcal{X}\to \mathcal{Y}$ be a morphism of Artin stacks over an Artin stack $S$.
    For any morphism $g\colon T\to S$ from an Artin stack, let $f_T\colon \mathcal{X}_T\to \mathcal{Y}_T$ denote the canonical morphism $\mathcal{X}\times_ST\to \mathcal{Y}\times_ST$.
    If $\mathcal{F}$ is a quasi-coherent sheaf on $\mathcal{X}$, then let $\mathcal{F}_T$ denote the pullback of $\mathcal{F}$ to $\mathcal{X}_T$.
    If $T=\mathrm{Spec}\,A$, then we will write it as $f_A$ for example.
    \item We will write $\mathfrak{X}\times_{\mathbb{K}}\mathfrak{Y}:=\mathfrak{X}\times\mathfrak{Y}$ for simplicity.
    \item Let $S$ be a scheme and $\mathcal{Q}:=\oplus_{d\ge0}\mathcal{Q}_d$ a graded finitely generated $\mathcal{O}_S$-algebra such that $\mathcal{Q}_d$ is coherent.
    Then we can consider $\mathcal{P}:=\mathbf{Proj}_S(\mathcal{Q})$.
    If $\mathcal{Q}_1$ generates $\mathcal{Q}$ and $\mathcal{Q}_0\cong\mathcal{O}_S$, then one can consider $\mathcal{O}_{\mathcal{P}}(1)$. 
    If $X\subset \mathcal{P}$ is a closed subscheme, then we set $\mathcal{O}_X(1):=\mathcal{O}_{\mathcal{P}}(1)|_X$.
    If there is no fear of confusion, we will write it as $\mathcal{O}(1)=\mathcal{O}_X(1)$ for simplicity.
    \item Let $X$ be a normal scheme proper over $S$, where $S$ is a variety over $\mathbb{K}$. We say that a $\mathbb{R}$-line bundle $L$ is {\it semiample} over $S$ if $L=\sum_{i=1}^r\alpha_i L_i$, where $\alpha_i>0$ and $L_i$ is a relatively globally generated line bundle on $X$ over $S$.
    If $L$ is a $\mathbb{Q}$-line bundle, this is nothing but that there exists $m\in\mathbb{Z}_{>0}$ such that $mL$ is relatively globally generated.
    \item Let $X$ be a scheme and $x\in X$.
    Then $\mathcal{O}_{X,x}$ denotes the stalk of $\mathcal{O}_X$ at $x$ and $\kappa(x)$ denotes the residue field.
    \item Let $f\colon X\to S$ be a proper flat morphism with a coherent flat sheaf $\mathcal{F}$ and $A$ a $f$-very ample line bundle.
    We say that $\mathcal{F}$ {\it has the Hilbert polynomial $P$ with respect to} $A$ if for any $s\in S$, $P(m)=\chi(\mathcal{F}_s\otimes A_s^{m})$ for any $m\in\mathbb{Z}$.
    We set $\mathrm{Hilb}_{X/S}$ as the Hilbert scheme parameterizing all flat closed subschemes of $X$ and $\mathrm{Quot}^{\mathcal{F}}_{X/S}$ as the Quot scheme parameterizing flat quotient sheaves of $\mathcal{F}$.
    As discussed in \cite[Sections 5 and 9]{fantechi2005fundamental}, we know the existence of such schemes.
    Furthermore, as \cite[Theorem 5.8]{fantechi2005fundamental}, take $l\in\mathbb{Z}_{>0}$ and consider the functor $\mathfrak{Hom}(\mathcal{O}_X^{\oplus l},\mathcal{F})$ such that for any scheme $T$ over $S$, we attain the set $\mathfrak{Hom}(\mathcal{O}_X^{\oplus l},\mathcal{F})(T):=\mathrm{Hom}_{\mathcal{O}_{X_T}}(\mathcal{O}_{X_T}^{\oplus l},\mathcal{F}_T)$.
    Then there exists a quasi-projective scheme $\mathbf{Hom}_S(\mathcal{O}_X^{\oplus l},\mathcal{F})$ corepresenting $\mathfrak{Hom}(\mathcal{O}_X^{\oplus l},\mathcal{F})$.
\end{enumerate}

\subsection{Stack theory}

We need stack theory (see \cite{olsson2016algebraic,stacksproject}). As in \cite[Remark 2.18]{hashizume2025geometry}, we consider the category of schemes as the category of locally Noetherian schemes over $\mathbb{K}$ but it does not affect our theory.
\begin{defn}
Let $\mathfrak{X}$ be a stack defined over the category of schemes over $\mathbb{K}$ endowed with the \'etale topology (cf.~\cite[Definition 4.6.1]{olsson2016algebraic}).   
We say that
$\mathfrak{X}$ is a {\it Deligne-Mumford stack} if 
\begin{enumerate}
\item the diagonal morphism $\Delta\colon\mathfrak{X}\to\mathfrak{X}\times \mathfrak{X}$ is representable by schemes, and
    \item there exists an \'etale surjective morphism $\pi\colon X\to \mathfrak{X}$ from a scheme over $\mathbb{K}$.
\end{enumerate}
If a Deligne-Mumford stack $\mathfrak{X}$ is an \'etale sheaf, then we say that $\mathfrak{X}$ is an {\it algebraic space}.
We say that $\mathfrak{X}$ is an {\it Artin stack} if 
\begin{enumerate}
\item the diagonal morphism $\Delta\colon\mathfrak{X}\to\mathfrak{X}\times \mathfrak{X}$ is representable by algebraic spaces, and
    \item there exists a smooth surjective morphism $\pi\colon X\to \mathfrak{X}$ from a scheme over $\mathbb{K}$.
\end{enumerate}
\end{defn}
The above definition of a Deligne-Mumford stack coincides with \cite[Definition 8.3.1]{olsson2016algebraic} thanks to \cite[Theorem 7.2.10]{olsson2016algebraic}.

\begin{defn}[{\cite[Definition 4.1]{Alper}}]
    Let $\mathfrak{X}$ be an Artin stack of finite type.
    Then a morphism $\pi\colon \mathfrak{X}\to X$ is called the morphism of the {\it good moduli space} if \begin{enumerate}
        \item $\mathcal{O}_{X}\cong \pi_*\mathcal{O}_{\mathfrak{X}}$, and
        \item $\pi_*\colon \mathrm{QCoh}_{\mathfrak{X}}\to \mathrm{QCoh}_{X}$ is an exact functor.
    \end{enumerate}
    We say that $X$ is the {\it good moduli space}, which is uniquely determined up to isomorphism by \cite[Theorem 6.6]{Alper}.
\end{defn}

From now on, we will recall the definition of the $\mathsf{S}$-completeness and $\Theta$-reductivity.

\begin{defn}[{\cite[Definitions 3.10 and 3.38]{AHLH}}]
    Let $\pi\colon \mathfrak{X}\to \mathfrak{Y}$ be a morphism of Artin stacks.
    For a discrete valuation ring (DVR for short) $R$, we set \[\overline{\mathrm{ST}}(R):=[\mathrm{Spec}(R[s,t]/(st-\varpi))/\mathbb{G}_{m,R}],\] where $\varpi$ is a generator of the maximal ideal of $R$, $\mathbb{G}_{m,R}$ acts on $s$ and $t$ with weight $1$ and $-1$, respectively.
    Set $\overline{\mathrm{ST}}^\circ(R)$ as the complement of the image of the point defined by the ideal $(s,t)$.
    Let $j^{\mathrm{ST}}\colon \overline{\mathrm{ST}}^\circ(R)\hookrightarrow\overline{\mathrm{ST}}(R)$ be the canonical open immersion.
    We set $\Theta_A:=[\mathbb{A}^1_A/\mathbb{G}_{m,A}]$ for any $\mathbb{K}$-algebra $A$.
    For $\Theta_R$, set $\Theta^\circ_R$ as the complement of the image of the zero point of $\mathbb{A}^1_{R/\varpi}$.
    Let $j^{\Theta}\colon \Theta^\circ_R\to\Theta_R$ be the canonical open immersion.
    Then we say that 
    \begin{enumerate}
        \item $\pi$ is {\it $\mathsf{S}$-complete} if for any DVR $R$ and morphism $\mu^\circ\colon \overline{\mathrm{ST}}^\circ(R)\to \mathcal{X}$ with morphism $\xi\colon \overline{\mathrm{ST}}(R)\to \mathcal{Y}$ such that $\xi\circ j^{\mathrm{ST}}=\pi\circ\mu^\circ$, there exists a unique morphism $\mu\colon \overline{\mathrm{ST}}(R)\to \mathcal{X}$ such that $\mu\circ j^{\mathrm{ST}}=\mu^\circ$ and  $\pi\circ\mu=\xi$, and 
        \item $\pi$ is {\it $\Theta$-reductive} if for any DVR $R$ and morphism $\mu^\circ\colon \Theta^\circ_R\to \mathcal{X}$ with morphism $\xi\colon \Theta_R\to \mathcal{Y}$ such that $\xi\circ j^{\Theta}=\pi\circ\mu^\circ$, there exists a unique morphism $\mu\colon \Theta_R\to \mathcal{X}$ such that $\mu\circ j^{\Theta}=\mu^\circ$ and  $\pi\circ\mu=\xi$.
    \end{enumerate}
\end{defn}

For the properties of the good moduli space of Artin stacks of finite type, see \cite{Alper,AHLH}.
For the definition of $\mathsf{S}$-completeness and $\Theta$-reductivity of an Artin stack, see \cite[Definitions 3.10 and 3.38]{AHLH} for more detail.
We say that an Artin stack $\mathfrak{X}$ of finite type over $\mathbb{K}$ with affine diagonal is $\mathsf{S}$-complete and $\Theta$-reductive if $\mathfrak{X}$ is $\mathsf{S}$-complete and $\Theta$-reductive over $\mathbb{K}$.

We say that $\mathfrak{X}$ satisfies the existence part of the valuative criterion if for any DVR $R$ over $\mathbb{K}$ with the fractional field $K$ and morphism $\mathrm{Spec}\,K\to \mathfrak{X}$, there exists an extension $R\subset R'$ of DVR such that the induced morphism $\mathrm{Spec}\,K'\to \mathfrak{X}$ can be extended to $\mathrm{Spec}\,R'\to \mathfrak{X}$, where $K'$ is the fractional field of $R'$.

\begin{thm}[{\cite[Theorem A]{AHLH}}]
  Let $\mathfrak{X}$ be an Artin stack of finite type over $\mathbb{K}$ with affine diagonal.
  Then $\mathfrak{X}$ has a separated good moduli space $X$ if and only if $\mathfrak{X}$ is $\mathsf{S}$-complete and $\Theta$-reductive.
  Furthermore, $\mathfrak{X}$ satisfies the existence part of the valuative criterion if and only if $X$ is proper over $\mathbb{K}$.
\end{thm}

\subsection{Log pairs and locally stable families}

We follow the fundamental notions of log pairs and the minimal model program in \cite{komo}.

Let $X$ be a quasi-projective normal variety with a $\mathbb{R}$-divisor $\Delta$.
If $K_X+\Delta$ is $\mathbb{R}$-Cartier, we say that $(X,\Delta)$ is a {\it sublog pair}.
If $\Delta$ is further effective, we say that $(X,\Delta)$ is a {\it log pair}.

Suppose that $(X,\Delta)$ is a sublog pair.
For any prime divisor $E$ over $X$, we set the following numerical invariant
\[
A_{X,\Delta}(E):=1+\mathrm{Coeff}_E(K_Y-\pi^*(K_X+\Delta)),
\]
where $\pi\colon Y\to X$ is a projective birational morphism from a normal variety such that $E$ is a prime divisor on $Y$. We call this number the {\it log discrepancy} of $E$ with respect to $(X,\Delta)$. Note that $A_{X,\Delta}(E)$ is independent of the choice of $\pi$.
We say that a sublog pair $(X,\Delta)$ is {\it sublc} (resp.~{\it subklt}) if $A_{X,\Delta}(E)\ge0$ (resp.~$>0$) for any prime divisor $E$ over $X$.
If $\Delta$ is further effective, we say that $(X,\Delta)$ is {\it log canonical (lc)} (resp.~{\it Kawamata log terminal (klt)}) if $(X,\Delta)$ is sublc (resp.~subklt).

\begin{defn}[Deminormal pairs {\cite[Section 5]{kollar2013singularities}}]
\begin{enumerate}
    \item 
    Let $X$ be a quasi-projective equidimensional reduced scheme with Serre's $S_2$-condition.
    Suppose that $X$ has at worst nodal singularities in codimension one and there exists a $\mathbb{R}$-divisor $\Delta$ whose support does not contain any singularity of codimension one.
    In this case, we say that $(X,\Delta)$ is {\it a deminormal subpair} if $K_X+\Delta$ is $\mathbb{R}$-Cartier.
    If $\Delta$ is further effective, then $(X,\Delta)$ is {\it a deminormal pair}.

    \item Suppose that $(X,\Delta)$ is deminormal.
    Let $\nu\colon Y\to X$ be the morphism of the normalization.
    Then $\mathcal{H}om_{\mathcal{O}_X}(\nu_*\mathcal{O}_{Y},\mathcal{O}_X)$ can be regarded as an ideal sheaf of both of $\mathcal{O}_X$ and $\mathcal{O}_Y$.
    We say that the subscheme defined by $\mathcal{H}om_{\mathcal{O}_X}(\nu_*\mathcal{O}_{Y},\mathcal{O}_X)$ is the {\it conductor}.
    Let $D$ be the conductor of $Y$ and note that $D$ is a reduced Weil divisor.
    We see that
    \[
    K_Y+D+\nu^{-1}_*\Delta=\nu^*(K_X+\Delta).
    \]
    If $(Y,D+\nu^{-1}_*\Delta)$ is lc, we say that $(X,\Delta)$ is {\it semi log canonical (slc)}.
    \end{enumerate}
\end{defn}

We will consider the family of slc pairs in the following sense.

\begin{defn}
Let $f\colon X\to S$ be a flat projective morphism of finite type and of locally Noetherian schemes. Suppose that any fiber is a deminormal scheme of pure dimension $n$.
Let $D$ be a closed subscheme of $X$.
We say that $D$ is a {\it relative Mumford divisor} if
\begin{enumerate}
    \item For any point $s\in S$, $X_s$ is smooth around the generic point of $D_{s}$,
    \item Let $U\subset X$ be the maximal open subset such that $f|_U$ is smooth, $D|_U$ is Cartier and flat over $S$,
    \item $\mathrm{codim}_{X_s}(D_s\setminus U)\ge2$ for any $s\in S$, and
    \item $D$ is the Zariski closure of $D|_U$.
\end{enumerate}
If the above $D$ is Cartier and flat over $S$, then we say that $D$ is a {\it relative Cartier divisor}.

We say that an $R$-linear combination $\Delta$ of relative Mumford divisors is a {\it relative Mumford $R$-divisor}, where $R=\mathbb{Z}, \mathbb{Q}$ or $\mathbb{R}$.
For any $T\to S$, we set $\Delta_T\subset X_T$ as the closure of the pullback of $\Delta|_U$ to $U_T$.
Then $\Delta_T$ is a relative Mumford $R$-divisor again.
\end{defn}

 Let $f\colon X\to S$ be a flat projective morphism of finite type and of locally Noetherian schemes of relative dimension $n$.
 Let $F$ be a coherent sheaf on $X$.
 Suppose that there exists an open subset $U\subset X$ such that $F|_U$ is flat over $S$ and Cohen-Macaulay and for any $s\in S$, $F_s|_U$ is of pure dimension $n-1$ and $X_s\setminus U$ contains only at most $(n-2)$-dimensional points of $\mathrm{Supp}(F_s)$.
 Suppose further that $f$ is smooth around the generic points of $\mathrm{Supp}(F_s)$ of codimension one in $X_s$.
 In this case, we set $\mathrm{DSupp}(F|_U)=\mathbf{Spec}_U(\mathcal{O}_U/I)$, where $I$ is the $0$-th Fitting ideal of $F|_U$ (see \cite[7.21]{kollar-moduli} for more details). As shown in \cite[Lemma 7.24]{kollar-moduli}, $\mathrm{DSupp}_S(F|_U)$ is a relative Cartier divisor over $S$.
 We also set the Zariski closure of $\mathrm{DSupp}_S(F|_U)$ in $X$ as $\mathrm{DSupp}_S(F)$. 
 Note that for any morphism $g\colon T\to S$, $g_X^*\mathrm{DSupp}_S(F)=\mathrm{DSupp}_T(F_T)$.

 If $S$ is the spectrum of a local ring, we say that an effective relative Mumford divisor $D\subset X$ is {\it K-flat} if for any finite morphism $\rho\colon \mathrm{Supp}(D)\to \mathbb{P}_S^{n}$, $\mathrm{DSupp}_S(\rho_*\mathcal{O}_D)$ is relative Cartier.
 In the general case, we say that $D$ is {\it K-flat} if $D_T$ is K-flat, where $T=\mathrm{Spec}\,\mathcal{O}_{S,s}$ for any $s\in S$.
 Note that if $f$ is a family of nodal curves, $D$ is automatically relative Cartier and K-flat.

Take $d\in\mathbb{Z}_{>0}$.
If $f$ is a family of nodal curves, we can consider a scheme $\mathrm{KDiv}^d_{X/S}$ separated and of finite type over $S$ that can be regarded as the fine moduli scheme parameterizing the family of K-flat divisors of degree $d$ (see \cite[Theorem 7.3]{kollar-moduli}).

\begin{defn}
    Let $f\colon X\to S$ be a flat projective morphism of finite type and of locally Noetherian schemes. Suppose that any fiber is a deminormal scheme of pure dimension $n$.
    Let $\mathscr{F}$ be a coherent sheaf such that there exists an open subset $U\subset X$ such that $\mathscr{F}|_{U}$ is a line bundle and for any $s\in S$, $\mathrm{codim}_{X_s}(X_{s}\setminus U)\ge2$.
    Assume that there exists a coherent sheaf $\mathscr{G}$ on $X$ such that 
    \begin{enumerate}
        \item $\mathscr{G}$ is flat over $S$, and
        \item there exists a morphism $\iota\colon\mathscr{F}\to \mathscr{G}$ such that for any $s\in S$, $\iota$ is isomorphic to the canonical morphism $\mathscr{F}_s\to(\mathscr{F}_s)^{[**]}$.
    \end{enumerate}
    Then we say that $\mathscr{G}$ is the {\it universal hull} of $\mathscr{F}$ and it is easy to see that $\iota$ is unique up to isomorphism.
\end{defn}

\begin{defn}
Let $f\colon X\to S$ be a flat projective morphism of finite type and of locally Noetherian schemes. Suppose that any fiber is a deminormal scheme of pure dimension $n$.
    Let $D_1,\ldots, D_r$ be K-flat relative Mumford divisors on $X$ over $S$ and $a_1,\ldots,a_r\in (0,1]$.
Let $V$ be the smallest rational linear subspace of $\mathbb{R}^{r+1}$  containing $(1,0,\ldots,0)$ and $(0,\ldots,0,a_i,0,\ldots,0)$.
We say that $f\colon (X,\sum_{i=1}^ra_i D_i)\to S$ is a {\it locally stable family} if 
\begin{enumerate}
\item for any $m_0,\ldots,m_r\in\mathbb{Z}_{\ge0}$ such that $(m_0,m_1,\ldots,m_r)\in V$ and $J\subset \{1,\ldots, r\}$ such that $a_j=1$ for any $j\in J$, there exists the universal hull $\omega^{[m_0]}_{X/S}(\sum_{i=1}^rm_iD_i-B)$ of $\omega_{X/S}^{\otimes m_0}\otimes\mathcal{O}_X(\sum_{i=1}^rm_iD_i-B)$, where $B=\sum_{j\in J}D_j$,
\item $K_{X/S}+\sum_{i=1}^ra_i D_i$ is $\mathbb{R}$-Cartier, and
\item $(X_s,\sum_{i=1}^ra_i D_{i,s})$ is slc for any $s\in S$. 
\end{enumerate}
Here, $K_{X/S}+\sum_{i=1}^ra_i D_i$ is defined as a relative Mumford $\mathbb{R}$-divisor by item (1).
\end{defn}

\begin{rem}
    We note that if $f$ is a family of nodal curves, then $f\colon (X,\sum_{i=1}^ra_i D_i)\to S$ is a locally stable family if and only if conditions (2) and (3) above are satisfied and $D_i$ is just a relative Mumford divisor for each $i$ because any relative Mumford divisor is relative Cartier.
    Indeed, if $f$ is a family of nodal curves, we can always define $K_{X/S}$ because nodal curves are Gorenstein and $\omega_{X/S}^{\otimes m_0}\otimes\mathcal{O}_X(\sum_{i=1}^rm_iD_i-B)$ is Cartier. 
    In this case, we do not have to care about condition (1).
    \end{rem}

\begin{defn}
    Let $C$ be a proper connected nodal curve.
    Let $C'\subset C$ be an irreducible component of $C$ such that $C'$ is isomorphic to $\mathbb{P}^1$ and $C'$ intersects with $\overline{(C\setminus C')}$ at one point transversally.
    We say that $C'$ is a {\it rational tail}.
    
    Let $C''\subset C$ be an irreducible component of $C$ such that $C''$ is isomorphic to $\mathbb{P}^1$ and $C''$ intersects with $\overline{(C\setminus C'')}$ at exactly two points transversally.
    We say that $C''$ is a {\it rational bridge}.
    
    Then we can consider a morphism $\varphi\colon C\to \overline{(C\setminus C'')}$ and we call $\varphi$ the {\it contraction of a rational tail $C'$}.
    Furthermore, we can glue $\overline{(C\setminus C')}$ along the two points of $C''\cap \overline{(C\setminus C')}$ together and construct a new proper nodal connected curve $D$.
    Then we can consider a morphism $\psi\colon C\to D$ and we call $\psi$ the {\it contraction of a rational bridge $C''$}.
    We also say that a surjective morphism of nodal proper connected curves given by a sequence of contractions of rational tails or rational bridges successively is a {\it contraction}.
\end{defn}
   We note that for any proper nodal connected but reducible curve $C$ with a boundary $\Delta$ such that $(C,\Delta)$ is lc and
    $K_C+\Delta$ is not nef, then there exists a rational tail $C'$ such that $\mathrm{deg}(-(K_C+\Delta)|_{C'})>0$.
    Furthermore, if $K_C+\Delta$ is not ample, then there exists a rational tail $C'$ such that $\mathrm{deg}(-(K_C+\Delta)|_{C'})\ge0$ or a rational bridge $C''$ such that $\mathrm{Supp}\,\Delta\cap C''=\emptyset$.

\begin{prop}\label{prop:pushfoward}
  Let $f\colon C\to S$ and $g\colon C'\to S$ be flat and projective morphisms with each fiber connected nodal curves.
  Suppose that there exists a surjective morphism $h\colon C\to C'$ over $S$ and for any geometric point $\bar{s}\in S$, $C_{\bar{s}}\to C'_{\bar{s}}$ consists of successive contractions $\varphi_i\colon C_i\to C_{i+1}$ of rational tails and rational bridges.
  If $D$ is an effective relative Mumford divisor on $C$ such that the image of $D$ in $C_i$ is disjoint from the curve contracted by $\varphi_i$ or $\varphi_i$ is a contraction of a rational tail for any $i$, then $\mathrm{DSupp}_S(h_*\mathcal{O}_D)$ is a relative Mumford and Cartier divisor such that for any $\bar{s}\in S$, $\mathrm{DSupp}_S(h_*\mathcal{O}_D)_{\bar{s}}=h_{\bar{s}*}D_{\bar{s}}$.
\end{prop}

\begin{proof}
Note that $D$ is a relative Cartier divisor by \cite[Corollary 7.53]{kollar-moduli}.  
Since $h_*\mathcal{O}_D$ is Cohen-Macaulay and flat, we see that $\mathrm{DSupp}_S(h_*\mathcal{O}_D)$ is relative Cartier by \cite[7.21.5]{kollar-moduli}.
Therefore, we may assume that $S$ is the spectrum of an algebraically closed field.
By induction on the number of contracting curves, we may assume that $h$ is a contraction of a rational tail $F$.
Let $p'\in C'$ be the image of $F$ and note that $p'$ is the smooth point.
The localization $(h_*\mathcal{O}_D)_{p'}$ at $p'$ has the same length as $l(\mathcal{O}_D|_F)$, where $l$ denotes the length.
It is easy to see that the $0$-th Fitting ideal of $(h_*\mathcal{O}_D)_{p'}$ is isomorphic to $\mathfrak{m}_{p'}^{l(\mathcal{O}_D|_F)}$, where $\mathfrak{m}_{p'}$ is the maximal ideal of $\mathcal{O}_{C',p'}$.
This is nothing but the ideal defining the localization of $h_*D$ at $p'$.
We complete the proof.
\end{proof}

\begin{defn}
 In Proposition \ref{prop:pushfoward}, we set $h_*D:=\mathrm{DSupp}_S(h_*\mathcal{O}_D)$.   
 We call this {\it the strict transform} of $D$ under $h$.
 If $\Delta$ is a relative Mumford $\R$-divisor such that $\Delta=\sum a_iD_i$, where $D_i$'s are effective relative Mumford divisors on $C$, then we set $h_*\Delta:=\sum a_ih_*D_i$.
 It is easy to see that $h_*$ is well-defined.
 By the compatibility of the Fitting ideal and the base change (cf.~\cite[Tag 07ZA]{stacksproject}), we see that for any morphism $g\colon T\to S$, $h_{T*}\Delta_T=(h_*\Delta)_T$.
\end{defn}

We show the following further.
See \cite[Tag 02AR]{stacksproject} for the definition of the sheaf of total quotient rings for any Noetherian scheme.

\begin{prop}\label{prop:push--linear-equivalence}

 In Proposition \ref{prop:pushfoward}, suppose further that $S$ is of finite type over $\mathbb{K}$. Let $D'$ be another relative Mumford divisor such that $D\sim D'$ on $C$.
Let $\mathscr{K}$ (resp.~$\mathscr{K}'$) be the sheaf of total quotient rings of $C$ (resp.~$C'$) and $\phi\in \mathscr{K}(C)$ such that $\mathrm{div}(\phi)=D-D'$.
 Then there exists a canonical morphism $h_\sharp\colon h_*\mathscr{K}\to \mathscr{K}'$ and $\mathrm{div}(h_\sharp(\phi))=h_*D-h_*D'$. 
 Furthermore, for any morphism $\mu\colon S'\to S$ from a scheme and $\phi$ as above, $\mu_{C'}^* h_\sharp(\phi)=(h_{S'})_\sharp(\mu_C^*\phi)$ , where $\mu_{C}\colon C_{S'}\to C$ and $\mu_{C'}\colon C'_{S'}\to C'$ are the canonical morphisms if $\phi\in \mathscr{K}(C)$ defines $\mathrm{div}(\phi)$ the summation of relative Cartier divisors.
\end{prop}

\begin{proof}
Note that $\mathscr{K}$ and $\mathscr{K}'$ are quasi-coherent and so is $h_*\mathscr{K}$ since $h$ is a projective morphism.
Let $W$ be an open affine subset such that $g(W)$ is open and for any $s\in g(W)$, the generic points of $g^{-1}(s)$ are contained in $W$, where we note that $g\colon C'\to S$ is the canonical morphism.
Indeed, we can take such $W$ by the flatness and projectivity of $g$.
Take the largest open subset $W'\subset h^{-1}(W)$ such that $h|_{W'}\colon W'\to W$ is an open immersion.
Note that there exists a canonical homomorphism $\mathscr{K}(h^{-1}(W))\to \mathscr{K}(W')$.
By \cite[Theorem 12]{matsumura1980commutative}, we see that $W\setminus h(W')$ contains no embedded point.
Therefore, we see that $\mathscr{K}'(W)\to \mathscr{K}'(h(W'))\cong \mathscr{K}(W')$ is an isomorphism.
Composing the inverse with $\mathscr{K}(h^{-1}(W))\to \mathscr{K}(W')$, we construct $h_\sharp\colon h_*\mathscr{K}(W)\to \mathscr{K}'(W)$.
We can cover $C'$ by such $W$ and the homomorphism can be patched together.
Therefore, we can construct $h_\sharp\colon h_*\mathscr{K}\to \mathscr{K}'$.
Note that for any morphism $\mu\colon S'\to S$ from a scheme, $\mu_{C'}^*h_\sharp(\phi)=(h_{S'})_\sharp(\mu_C^*\phi)$.
Indeed, if $\mathrm{div}(\phi)$ is the summation of relative Cartier divisors, then $h_\sharp(\phi)$ has no zero or pole containing any fiber over $S$ and we can define $\mu_{C}^*\phi$ and $\mu_{C'}^*h_\sharp(\phi)$.

 Let $\phi$ be a global section of the sheaf of total quotient rings $\mathscr{K}$ of $C$ such that $\mathrm{div}(\phi)=D-D'$ as the statement.
 Then we can take $\phi':=h_\sharp \phi\in \mathscr{K}'(C')$.
 By $\mathrm{div}(\phi')$ on $C'$, we see that $\mathrm{div}(\phi'_s)=h_{s*}D_s-h_{s*}D'_s$ for any $s\in S$.
 Therefore, $\mathrm{div}(\phi')=h_*D-h_*D'$ if $S$ is reduced by \cite[Theorem 9.3.7]{fantechi2005fundamental}.
 In the general case, by \cite[Theorem 9.3.7]{fantechi2005fundamental}, we see that there exists a closed subscheme $S'\to S$ that contains the reduced structure of $S$ such that for any morphism $\theta\colon T\to S$ from a scheme of finite type over $\mathbb{K}$, $\theta$ factors through $S'$ if and only if $\mathrm{div}(\phi')_T=(h_*D)_T- (h_*D')_T$. 
 Therefore, we may replace $S$ with its Artinian subscheme because it suffices to show that all Artinian subschemes of $S$ are contained in $S'$.
 Furthermore, we may assume that $S$ is a spectrum of a local ring.
 Note that $h_{\bar{s}}$ is obtained as a sequence of contractions of one rational tail or rational bridge.
 The first contraction of one curve induced by some globally generated line bundle of the form $K_{C_{\bar{s}}}+D''$ for some effective divisor $D''$ and this line bundle is extendable to $C$ by the condition that $H^2(C,\mathcal{O}_C)=0$.  
  Then, using \cite[Proposition 3.3]{Hassett}, we can extend this morphism to entire $C$. Thus, we may assume that $h$ is a contraction of one rational tail. 
  Indeed, if we assume that $h$ is a contraction of a rational bridge, then the supports of $D$ and $D'$ are disjoint from the bridge.
  Let $D'''$ be a sufficiently general relative ample divisor on $C$ such that $H^1(C_{\bar{s}},\mathcal{O}_{C_{\bar{s}}}(D'_{\bar{s}}+D'''_{\bar{s}}))=0$ for any geometric point $\bar{s}\in S$.
  Take a rational number $a<1$ and an effective relative Mumford $\Q$-divisor $\Delta$ such that for any $\epsilon>0$, $(C,(1+\epsilon)\Delta+a(D'+D'''))$ is a stable family but $(C,\Delta+a(D'+D'''))$ is not stable and the reduction functor in \cite[Theorem 4.1]{Hassett} induces $h$ because $S$ is Artinian.
  Here, we can take $\Delta$ such that the support of $\Delta$ is a summation of sections.
  By the argument of \cite[Proposition 3.8, Section 2.1.3]{Hassett}, we see that $(C,\Delta)$ has unobstructed deformation.
  By this and the fact that $H^1(C_{\bar{s}},\mathcal{O}_{C_{\bar{s}}}(D'_{\bar{s}}+D'''_{\bar{s}}))=0$, we may assume that there exists a smooth scheme $T$ of finite type over $\mathbb{K}$ with a morphism $\alpha\colon S\to T$ and  there exist flat and projective morphisms $f_T\colon C_T\to T$ and $g_T\colon C'_T\to T$ with each fiber connected nodal curves satisfying the following:
\begin{itemize}
    \item There exists a projective morphism $h_T\colon C_T\to C'_{T}$ such that for any $t\in T$, $(h_T)_{t}$ is an isomorphism or a contraction of a smooth rational curve $F$ such that $F\cap\mathrm{Sing}((C_T)_t)$ is one point,
    \item the pullbacks of $C_T$, $C'_T$, $h_{T}$, and $D'_T$ by $\alpha$ are $C$, $C'$, $h$, and $D'+D'''$, respectively,
    \item there exists a $\mathbb{Q}$-divisor $D_T$ and $D'_T$ with $\varphi\in \mathscr{K}_{C_T}(C_T)$ such that $\mathrm{div}(\varphi)=D_T-D'_T$ and $\varphi_S$ coincides with $\phi$ if we regard $\varphi\in H^0(C_T,\mathcal{O}_{C_T}(D'_T))$ and $\phi\in H^0(C,\mathcal{O}_C(D'+D'''))$.
\end{itemize}
In this case, we see that $(D_T)_S =D+D'''$.
As we have shown the assertion for the reduced base, we see that the assertion also holds for $S$.
We complete the proof.
\end{proof}

\subsection{K-stability}

We recall the definition of K-stability.
We say that $(X,\Delta;L)$ is a {\it polarized deminormal subpair} if $(X,\Delta)$ is a deminormal subpair, $X$ is connected and $L$ is an ample $\mathbb{R}$-line bundle.
If $\Delta$ is further effective, then we say that $(X,\Delta;L)$ is a {\it polarized deminormal pair}.

\begin{defn}\label{defn:DF-invariant}
Let $(X,\Delta;L)$ be a polarized deminormal subpair.
    We say that $(\mathcal{X},\mathcal{L})$ is a semiample test configuration for $(X,\Delta;L)$ if the following conditions hold:
    \begin{enumerate}
        \item There exists a projective flat morphism $\pi\colon \mathcal{X}\to\mathbb{A}^1$,
        \item $\mathcal{L}$ is a semiample $\mathbb{G}_m$-linearlized $\mathbb{R}$-line bundle on $\mathcal{X}$, that is, $\mathcal{L}$ can be decomposed into a nonnegative $\mathbb{R}$-linear combination of globally generated $\mathbb{G}_m$-linearlized line bundles on $\mathcal{X}$,
        \item $\mathbb{G}_m$ acts on $(\mathcal{X},\mathcal{L})$ equivariantly over $\mathbb{A}^1$, which $\mathbb{G}_m$ acts on $\mathbb{A}^1$ by natural multiplication, and
        \item $(\mathcal{X}_1,\mathcal{L}_1)\cong (X,L)$.
    \end{enumerate}
    If $\mathcal{X}$ is deminormal, we say that $(\mathcal{X},\mathcal{L})$ is a {\it semiample deminormal test configuration}.
    Suppose that $\mathcal{X}$ is deminormal.
    Let $(\overline{\mathcal{X}},\overline{\mathcal{L}})$ be the pair of deminormal schemes and a $\mathbb{R}$-line bundle on $\overline{\mathcal{X}}$ that is constructed by patching $(X\times(\mathbb{P}^1\setminus\{0\}),L\times(\mathbb{P}^1\setminus\{0\}))$ with the trivial $\mathbb{G}_m$-action on the first factor and $(\mathcal{X},\mathcal{L})$ together.
    Set $\Delta_{\mathcal{X}}$ as the closure of $\Delta\times\mathbb{G}_{m}\subset \overline{\mathcal{X}}$.
Then, we set
\[
\mathrm{DF}_{\Delta}(\mathcal{X},\mathcal{L}):=\frac{1}{L^n}\left((K_{\overline{\mathcal{X}}/\mathbb{P}^1}+\Delta_{\mathcal{X}})\cdot \overline{\mathcal{L}}^n-\frac{n(K_X+\Delta)\cdot L^{n-1}}{(n+1)L^n}\overline{\mathcal{L}}^{n+1}\right),
\]
where $\mathrm{dim}\,X=n$.
We call this numerical invariant the {\it Donaldson-Futaki invariant}.
We note that $K_{\overline{\mathcal{X}}/\mathbb{P}^1}+\Delta_{\mathcal{X}}$ might not be $\mathbb{R}$-Cartier in general but we can define the intersection number $(K_{\overline{\mathcal{X}}/\mathbb{P}^1}+\Delta_{\mathcal{X}})\cdot \overline{\mathcal{L}}^n$ since $\overline{\mathcal{L}}$ is a $\mathbb{R}$-line bundle.

On the other hand, take a $\mathbb{G}_{m}$-equivariant resolution of indeterminacy $\mathcal{Y}$ of the rational map from the normalization of $\overline{\mathcal{X}}$ to $X\times\mathbb{P}^1$.
Let $H$ be a semiample $\mathbb{R}$-line bundle on $X$ and set $\rho\colon \mathcal{Y}\to X\times\mathbb{P}^1$ and $\mu\colon \mathcal{Y}\to\overline{\mathcal{X}}$ as the canonical proper morphisms.
Then we set 
\[
\mathcal{J}^{H,\mathrm{NA}}(\mathcal{X},\mathcal{L}):=\frac{1}{L^n}\left(\rho^*(H\times\P^1)\cdot\mu^* \overline{\mathcal{L}}^n-\frac{nH\cdot L^{n-1}}{(n+1)L^n}\overline{\mathcal{L}}^{n+1}\right)
\]
and call this the {\it non-Archimedean} $\mathrm{J}^H$-functional.
Note that this is independent of the choice of $\mathcal{Y}$.
\end{defn}

On the non-Archimedean J-functional, we note the following.

\begin{lem}\label{lem:J-functional--computation}
Let $(X,\Delta;L)$ be a polarized deminormal subpair of dimension $n$ with a semiample $\mathbb{R}$-line bundle $H$.
    In the same situation as Definition \ref{defn:DF-invariant},
\[
\mathcal{J}^{H,\mathrm{NA}}(\mathcal{X},\mathcal{L}):=\sup_{D}\frac{1}{L^n}\left(\rho_*^{-1}(D\times\mathbb{P}^1)\cdot\mu^* \overline{\mathcal{L}}^n-\frac{nH\cdot L^{n-1}}{(n+1)L^n}\overline{\mathcal{L}}^{n+1}\right),
\]
where $D$ runs over all effective $\mathbb{R}$-divisor $\mathbb{R}$-linearly equivalent to $H$.
Moreover, the maximum is attained by any sufficiently general effective $\mathbb{R}$-divisor $D$.
\end{lem}

\begin{proof}
 Since $\rho^*(H\times\P^1)-\rho_*^{-1}(D\times\mathbb{P}^1)$ is effective for any effective $\mathbb{R}$-divisor $\mathbb{R}$-linearly equivalent to $H$, we see that 
 \[
\mathcal{J}^{H,\mathrm{NA}}(\mathcal{X},\mathcal{L})\ge\frac{1}{L^n}\left(\rho_*^{-1}(D\times\mathbb{P}^1)\cdot\mu^* \overline{\mathcal{L}}^n-\frac{nH\cdot L^{n-1}}{(n+1)L^n}\overline{\mathcal{L}}^{n+1}\right).
\]
For any sufficiently general $D$, $\rho^*(H\times\P^1)-\rho_*^{-1}(D\times\mathbb{P}^1)\sim_{\mathbb{R},\mathbb{P}^1}0$ and hence
\[
\mathcal{J}^{H,\mathrm{NA}}(\mathcal{X},\mathcal{L})=\frac{1}{L^n}\left(\rho_*^{-1}(D\times\mathbb{P}^1)\cdot\mu^* \overline{\mathcal{L}}^n-\frac{nH\cdot L^{n-1}}{(n+1)L^n}\overline{\mathcal{L}}^{n+1}\right).
\]
We complete the proof.
\end{proof}

\begin{defn}
    Let $(X,\Delta;L)$ be a polarized deminormal pair.
    We say that $(X,\Delta;L)$ is {\it K-semistable} if $\mathrm{DF}_{\Delta}(\mathcal{X},\mathcal{L})\ge 0$ for any semiample deminormal test configuration $(\mathcal{X},\mathcal{L})$ for $(X,\Delta,L)$.

    We say that $(X,\Delta;L)$ is {\it K-stable} if $(X,\Delta;L)$ is K-semistable and $\mathrm{DF}_{\Delta}(\mathcal{X},\mathcal{L})>0$ for any semiample deminormal test configuration $(\mathcal{X},\mathcal{L})$ for $(X,\Delta;L)$ such that $\mathcal{L}$ is ample and there exists no $\mathbb{G}_m$-equivariant isomorphism $\sigma\colon (\mathcal{X},\Delta_{\mathcal{X}})\cong (X\times\mathbb{A}^1,\Delta\times\mathbb{A}^1)$, where the latter one admits the standard $\mathbb{G}_m$-action.
\end{defn}

\begin{rem}\label{remark:pos-L,NA}
    This definition is equivalent to the original definition of K-stability when $L$ and $\Delta$ are $\Q$-divisors as in \cite{boucksom2017uniform}.
    See \cite[Section 3]{odaka2013generalization} and \cite[Section 3]{fujita2019openness} for more details.

    Furthermore, note that $\mathcal{J}^{L,\mathrm{NA}}(\mathcal{X},\mathcal{L})\ge0$ for any semiample deminormal test configuration for $(X,\Delta;L)$ and if $\mathcal{L}$ is ample, then $\mathcal{J}^{L,\mathrm{NA}}(\mathcal{X},\mathcal{L})=0$ if and only if there exists a $\mathbb{G}_m$-equivariant isomorphism $\sigma\colon (\mathcal{X},\Delta_{\mathcal{X}})\cong (X\times\mathbb{A}^1,\Delta\times\mathbb{A}^1)$, where the latter one admits the standard $\mathbb{G}_m$-action.
    This fact is well-known where $\Delta$ and $L$ are $\Q$-divisors and we explain how to show this for the case where $\Delta$ and $L$ are $\R$-divisors for the reader's convenience. If there exists such $\sigma$, then $\mathcal{J}^{L,\mathrm{NA}}(\mathcal{X},\mathcal{L})=0$ by the proof of \cite[Lemma 2.7 and Corollary 3.4]{hattori2026decomposition}.
    On the other hand, the proof of \cite[Lemma 2.7 and Corollary 3.4]{hattori2026decomposition} and the $S_2$-condition of $\mathcal{X}$ show that we may assume that $X$ and $\mathcal{X}$ are normal.
    Then we can take $\mathbb{G}_m$-equivariant projective birational morphisms $\mu\colon \mathcal{Y}\to\mathcal{X}$ and $\rho\colon \mathcal{Y}\to X\times\mathbb{A}^1$ from a normal variety.
    Note that there exists a $\R$-Cartier divisor $D$ supported in $\mathcal{Y}_0$ such that $\mu^*\mathcal{L}=\rho^*(L\times\mathbb{A}^1)+D$.
    Since $\mathcal{L}$ is ample, if there exists no $\sigma$ as above, $D\not\sim_\R0$.
    It is enough to show that 
    \[
\frac{d}{dt}\bigg|_{t=t_0}\left(\rho^*(L\times\P^1)\cdot(\rho^*L\times\mathbb{P}^1+tD)^n-\frac{n}{(n+1)}(\rho^*L\times\mathbb{P}^1+tD)^{n+1}\right)>0
\]
for any $t_0\in (0,1)$.
The above value coincides with $-nt_0D^2\cdot (\rho^*(L\times\P^1)+t_0D)^{n-1}$ and this is positive by \cite[Lemma 1]{li2014special}.
\end{rem}

We collect the following facts.

\begin{thm}[\cite{odaka2012calabi,odaka2013git,odaka2015testing,boucksom2017uniform}]\label{thm:K-stability--for--log-pairs--fundaments}
Let $(X,\Delta;L)$ be a polarized deminormal scheme.
Then the following hold.
\begin{enumerate}
    \item If $(X,\Delta;L)$ is K-semistable, then $(X,\Delta)$ is slc.
    \item If $L=K_X+\Delta$, then $(X,\Delta;L)$ is K-stable if and only if $(X,\Delta)$ is slc.
    \item If $K_X+\Delta\sim_{\mathbb{R}}0$, then $(X,\Delta;L)$ is K-stable if and only if $(X,\Delta)$ is klt, and $(X,\Delta;L)$ is K-semistable if and only if $(X,\Delta)$ is slc.
    \item If $L=-(K_X+\Delta)$ and $(X,\Delta;L)$ is K-semistable, then $(X,\Delta)$ is klt.
\end{enumerate}
\end{thm}

\begin{proof}
    This is well-known in the case where $L$ and $\Delta$ are $\mathbb{Q}$-divisors.
    We only explain how the same argument in the case where $L$ and $\Delta$ are $\mathbb{Q}$-divisors works for the general case.

    It is not hard to check that \cite[Theorem 4.4]{fujino2023existence} also works for the $\R$-divisor case. Then the argument to show \cite[Theorem 1.2]{odaka2013git} equally works (cf.~the argument to show \cite[Theorem 5.13]{hattori2024fibration}. It also works for slc pairs with non-zero boundaries).
    Thus, we obtain (1).
    Furthermore, we can deduce (2), (3), and (4) as the argument of \cite[Section 9]{boucksom2017uniform} combining what we have stated in Remark \ref{remark:pos-L,NA} when $X$ is normal.
    When $X$ is not normal, we can also deduce (2) by normalize $X$ (cf.~\cite[Remark 3.19]{boucksom2017uniform}).
    Similarly, in (3), we obtain that if $K_X+\Delta\sim_{\mathbb{R}}0$ and $(X,\Delta)$ is slc, then $(X,\Delta;L)$ is K-semistable. 
    For (4), as the proof of \cite[Theorem 5.15]{hattori2024fibration}, we see that some ideal sheaf whose cosupport is contained in the non-normal locus induces a semiample deminormal test configuration with negative Donaldson-Futaki invariant and hence $(X,\Delta;L)$ is never K-semistable if $X$ is not normal.
    Similarly to this argument, we see that if $K_X+\Delta\sim_{\mathbb{R}}0$ and $(X,\Delta)$ is slc but not normal, then some ideal sheaf whose cosupport is contained in the non-normal locus induces a semiample deminormal test configuration $(\mathcal{X},\mathcal{L})$ with $\mathcal{L}$ ample and non-positive Donaldson-Futaki invariant and hence $(X,\Delta;L)$ is never K-stable.
    We complete the proof.
\end{proof}

For the Fano case, we have to introduce the $\delta$-invariant, which was firstly defined by Fujita-Odaka \cite{fujita2018k}.

\begin{defn}
    Let $(X,\Delta)$ be a {\it log Fano pair}, that is, $(X,\Delta)$ is a projective klt pair with $-(K_X+\Delta)$ ample.
Suppose that $\mathrm{dim}\,X=n$.
Then we define $$\delta(X,\Delta):=\inf_{E}\frac{A_{X,\Delta}(E)}{S_{X,\Delta}(E)},$$
where $E$ runs over all prime divisors over $X$. Here, we set 
\[
S_{X,\Delta}(E)=\frac{1}{(-K_X-\Delta)^n}\int_{0}^{+\infty}\mathrm{vol}_Y(-\pi^*(K_X+\Delta)-xE)dx,
\]
where $\pi\colon Y\to X$ is a projective birational morphism from a smooth variety $Y$ such that $E$ is a prime Cartier divisor on $X$.
We call $\delta(X,\Delta)$ the {\it delta invariant}.
\end{defn}

For a log Fano pair $(X,\Delta)$, if we choose $L=-(K_X+\Delta)$, we say that $(X,\Delta)$ is {\it K-(semi)stable} if $(X,\Delta;L)$ is K-(semi)stable for simplicity.

On the delta invariant, the following holds.

\begin{thm}[{\cite{blum2020thresholds,fujita2019valuative,li2017k,liu2022finite}}]\label{thm:delta--log--pairs}
   Let $(X,\Delta)$ be a log Fano pair. Then
    $(X,\Delta)$ is K-stable if and only if $\delta(X,\Delta)>1$ and K-semistable if and only if $\delta(X,\Delta)\ge1$. 
\end{thm}

\begin{proof}
The case where $\Delta$ is a $\mathbb{Q}$-divisor was solved by combining the results of \cite{blum2020thresholds,fujita2019valuative,li2017k,liu2022finite}.
    For the case where $\Delta$ is an $\mathbb{R}$-divisor, if $\delta(X,\Delta)<1$, then $(X,\Delta)$ is not K-semistable by \cite[Theorem 3.2]{liu2024k}.
    If $\delta(X,\Delta)=1$, then we can also check that $(X,\Delta)$ is not K-stable by \cite[Theorem 4.5 and the second proof of Theorem 3.2]{liu2024k}.

    For the converse direction, consider an arbitrary semiample normal test configuration $(\mathcal{X},\mathcal{L})$ for $(X,\Delta;-K_X-\Delta)$.
    Taking the ample model, we may assume that $\mathcal{L}$ is ample.
    We can take $\{t_i\}_{i=1}^l\subset\mathbb{R}_{>0}$, effective $\mathbb{Q}$-divisors $\Delta_i$ on $X$ and ample $\mathbb{G}_m$-linearlized $\mathbb{Q}$-line bundles on $\mathcal{X}$ such that $\sum_{i=1}^l t_i=1$, $(X,\Delta_i)$ is a log Fano pair for each $i$, $(\mathcal{X},\mathcal{L}_i)$ is a test configuration for $(X,\Delta_i;-K_X-\Delta_i)$, and $\mathcal{L}=\sum_{i=1}^lt_i\mathcal{L}_i$.
    Thus, we can take sequences of effective $\mathbb{Q}$-divisors $\{\Delta^{(j)}\}_{j=1}^\infty$ on $X$ and ample $\mathbb{G}_m$-linearlized $\Q$-line bundles $\{\mathcal{L}^{(j)}\}_{j=1}^\infty$ on $\mathcal{X}$ such that $(\mathcal{X},\mathcal{L}^{(j)})$ is a test configuration for each $j\in\mathbb{N}$ and the coefficients of  $\Delta^{(j)}$ and $\mathcal{L}^{(j)}$ converge to those of $\Delta$ and $\mathcal{L}$ in $j\to\infty$.
    By \cite[Proposition 2.10]{liu2024k}, we note that $\lim_{j\to\infty}\delta(X,\Delta^{(j)})=\delta(X,\Delta)$.
    By \cite[Theorem A.6]{hattori2025k}, \cite[Corollary 3.4]{hattori2026decomposition}, and \cite[Lemma 7.25]{boucksom2017uniform} (see also \cite[Definition 2.8]{hattori2025k} for the notation), we note that
    \[
    \mathrm{DF}_{\Delta^{(j)}}(\mathcal{X},\mathcal{L}^{(j)})\ge (\delta(X,\Delta^{(j)})-1)\mathcal{J}^{-K_X-\Delta^{(j)},\mathrm{NA}}(\mathcal{X},\mathcal{L}^{(j)}).
    \]
    By the continuity of the intersection numbers, we see that
    \begin{align*}
        \mathrm{DF}_{\Delta}(\mathcal{X},\mathcal{L})&=\lim_{j\to\infty}\mathrm{DF}_{\Delta^{(j)}}(\mathcal{X},\mathcal{L}^{(j)})\\
        &\ge \lim_{j\to\infty}(\delta(X,\Delta^{(j)})-1)\mathcal{J}^{-K_X-\Delta^{(j)},\mathrm{NA}}(\mathcal{X},\mathcal{L}^{(j)})\\
        &=(\delta(X,\Delta)-1)\mathcal{J}^{-K_X-\Delta,\mathrm{NA}}(\mathcal{X},\mathcal{L}).
    \end{align*}
    By what we mentioned in Remark \ref{remark:pos-L,NA}, we obtain the assertion.
\end{proof}

\begin{ex}\label{ex:one-dim-K-st}
   As mentioned in \cite{li2015remarks}, if $(X,\Delta)$ is log Fano and $\mathrm{dim}\,X=1$, then $X\cong\P^1$ and $(X,\Delta)$ is K-semistable (resp.~K-stable) if and only if $\mathrm{ord}_p(\Delta)\le\frac{1}{2}\mathrm{deg}(\Delta)$ (resp.~$\mathrm{ord}_p(\Delta)<\frac{1}{2}\mathrm{deg}(\Delta)$) for any closed point $p\in X$.
   Indeed, it is easy to check that $\delta(X,\Delta)=\inf_{p\in X}\frac{2(1-\mathrm{ord}_p(\Delta))}{\mathrm{deg}(-(K_{X}+\Delta))}$.
\end{ex}

From now on, we collect the moduli theory.

\begin{defn}
Fix $d\in\mathbb{Z}_{>0}$ and $v,c\in\mathbb{Q}_{>0}$.
Define $\mathcal{M}^{\mathrm{KSBA}}_{d,v,c}$ as the following category fibered in groupoids.
For any scheme $S$, the objects of $\mathcal{M}^{\mathrm{KSBA}}_{d,v,c}(S)$ are $f\colon (X,\Delta)\to S$ such that $f$ is locally stable, $\Delta=cD$ for some K-flat relative Mumford divisor on $X$, and $K_{X/S}+\Delta$ is $f$-ample with $\mathrm{dim}\,X_s=d$ and $\mathrm{vol}_{X_s}(K_{X_s}+\Delta_s)=v$ for any $s\in S$.
For any two objects $f\colon (X,\Delta)\to S$ and $f'\colon (X',\Delta')\to S'$, we define an arrow $\varphi$ from $f$ to $f'$ such that $\varphi\colon X\to X'$ is a morphism that is compatible with the natural morphism $S\to S'$ such that $(X,\Delta)\cong (X'\times_{S'}S,\Delta'_{S})$.
\end{defn}

\begin{thm}\label{thm-KSBA}
   $\mathcal{M}^{\mathrm{KSBA}}_{d,v,c}$ is a Deligne-Mumford stack proper over $\mathbb{K}$ with the projective coarse moduli space $M^{\mathrm{KSBA}}_{d,v,c}$.
\end{thm}

The above theorem is first shown by \cite{kollar1988threefolds,alexeev1994boundedness,kollar-proj} for the surface case with zero boundary.
For the higher-dimensional case, Theorem \ref{thm-KSBA} was shown by dedications of many algebro-geometers. See \cite{kollar-moduli} for more details.
The corresponding wall crossing theory was established by \cite{ascher2023ksba,meng2023mmp} in the higher-dimensional case.
We remark that this moduli can be regarded as one example of K-moduli, that is, the moduli space parameterizing K-(poly)stable polarized varieties.

\begin{defn}
Fix $d\in\mathbb{Z}_{>0}$ and $v,c\in\mathbb{Q}_{>0}$.
Define $\mathcal{M}^{\mathrm{Kss}}_{d,v,c}$ as the following category fibered in groupoids.
For any scheme $S$, the objects of $\mathcal{M}^{\mathrm{Kss}}_{d,v,c}(S)$ are $f\colon (X,\Delta)\to S$ such that $f$ is locally stable, $\Delta=cD$ for some K-flat relative Mumford divisor on $X$, and $(X_s,\Delta_s)$ is a K-semistable log Fano pair with $\mathrm{dim}\,X_s=d$ and $\mathrm{vol}_{X_s}(-(K_{X_s}+\Delta_s))=v$ for any $s\in S$.
For any two objects $f\colon (X,\Delta)\to S$ and $f'\colon (X',\Delta')\to S'$, we define an arrow $\varphi$ from $f$ to $f'$ such that $\varphi\colon X\to X'$ is a morphism that is compatible with the natural morphism $S\to S'$ such that $(X,\Delta)\cong (X'\times_{S'}S,\Delta'_{S})$.
\end{defn}

\begin{thm}\label{thm-kss}
   $\mathcal{M}^{\mathrm{Kss}}_{d,v,c}$ is an Artin stack of finite type over $\mathbb{K}$ with the projective good moduli space $M^{\mathrm{Kps}}_{d,v,c}$.
\end{thm}
$M^{\mathrm{Kps}}_{d,v,c}$ is known as the K-moduli space of log Fano pairs.
For more details on Theorem \ref{thm-kss}, see \cite{xu2025k}.
Theorem \ref{thm-kss} was generalized by \cite{liu2024k} to the case with real coefficients and the corresponding moduli wall crossing theory was established by \cite{liu2026nonproportional,zhou2025finiteness} assuming certain boundedness where the anticanonical volume is infinitely small.  

As \cite[Conjecture 5.2]{odaka2010git}, Odaka predicted that the {\it CM line bundle}, which is observed by \cite{fujiki1990moduli} and named by \cite{tian1997kahler}, would be ample on the K-moduli space if it existed.
For $M^{\mathrm{KSBA}}_{d,v,c}$ and $M^{\mathrm{Kps}}_{d,v,c}$, the ampleness of the CM line bundle was confirmed by \cite{PX} and \cite{codogni2021positivity,posva2022positivity,xu2020positivity}, respectively.
We define the CM line bundle for the moduli of weighted quasimaps in Definition \ref{defn:CM-line-for-qmaps} later.

In the one-dimensional case, we can simplify the definition of locally stable families of K-semistable objects.

\begin{defn}
 Let $f\colon X\to S$ be a flat projective family of connected nodal curves.
   Let $\Delta$ be an effective relative $\mathbb{R}$-divisor on $X$ over $S$ such that $(X,\Delta)\to S$ is locally stable.
\begin{enumerate}
    \item    If $K_{X/S}+\Delta$ is $f$-ample, we say that $f\colon(X,\Delta)\to S$ is a {\it stable family}.
    \item    If $-(K_{X/S}+\Delta)$ is $f$-ample and $(X_{\bar{s}},\Delta_{\bar{s}})$ is K-semistable for any geometric point $\bar{s}\in S$, then we say that $f\colon(X,\Delta)\to S$ is a {\it K-semistable family of log Fano pairs}.
\end{enumerate}
\end{defn}

\subsection{Boundary polarized CY pairs}

For the Calabi-Yau case, we may consider the moduli stack of K-semistable polarized Calabi-Yau pairs naively, but this stack is not of finite type and does not yield any good moduli space in the usual sense in general.
For the lower-dimensional case, \cite{babwild} introduced the following object.

\begin{defn}[{\cite[Definition 2.2]{babwild}}]
Let $(X,\Delta)$ be a projective slc pair such that $-K_X-\Delta$ is ample.
We say that $(X,\Delta+D)$ is a {\it boundary polarized CY pair} if $(X,\Delta+D)$ is slc and $D$ is an effective $\mathbb{R}$-divisor such that $K_X+\Delta+D\sim_{\mathbb{R}}0$. 
We say that $D$ is the {\it polarization part}. If there is no fear of confusion, we will omit the polarization and write $(X,\Delta+D)$ as $(X,\Delta')$ for simplicity, where $\Delta':=\Delta+D$. 
\end{defn}

\cite{babwild,BL24} mainly considered boundary polarized CY pairs of dimension two with rational coefficients.
We mainly deal with the case where $\mathrm{dim}\,X=1$ but do not assume that the coefficients are rational.
In this case, if $(X,\Delta+D)$ is not klt, then $(X,\Delta+D)$ has regularity $0$ in the sense of \cite[Def.~8.5]{babwild}.
Therefore, we see the following.
\begin{lem}[{\cite[Lem.~8.9 and Prop.~8.10]{babwild}}]\label{lem--babwild--chain--two--comps}
Let $(X,\Delta+D)$ be a connected boundary polarized CY pair of dimension one.
Suppose that $(X,\Delta+D)$ is not normal.
Then the normalization $\overline{X}$ of $X$ has at most two irreducible components, where $\overline{G}$ is the conductor and $\overline{\Delta}$ is the strict transform of $\Delta$.
Furthermore, each irreducible component of $(\overline{X},\overline{\Delta}+\overline{G})$ has a unique lc center.
\end{lem}

Note that the same argument as \cite[Lem.~8.9 and Prop.~8.10]{babwild} works for the case where $\Delta$ and $D$ have irrational coefficients by perturbing the coefficients. 

\subsubsection{Klt-trivial fibrations}

Let $f\colon X\to S$ be a morphism of projective normal varieties such that $f_*\mathcal{O}_X\cong\mathcal{O}_S$.
Let $\Delta$ be an effective $\mathbb{Q}$-divisor on $X$.
We say that $f$ is a {\it klt-trivial fibration} if $(F,\Delta|_F)$ is klt, where $F$ is the generic fiber of $f$ and $b(K_X+\Delta)\sim_{f}0$ for some $b\in\mathbb{Z}_{>0}$.
Set $\varphi$ as a nonzero rational function of $X$ such that $K_X+\Delta+\frac{1}{b}\mathrm{div}(\varphi)=f^*L$ for some $\mathbb{Q}$-divisor $L$ on $S$.
For any projective birational morphism $\mu\colon S'\to S$ from a normal variety, we can take $f'\colon X'\to S'$ and $g\colon X'\to X$ from a projective normal variety such that $g$ is birational.
Set $\Delta'$ on $X'$ as a $\mathbb{Q}$-divisor such that $(X',\Delta')$ is log crepant to $(X,\Delta)$.
Set $B_{S'}$ as the $\mathbb{Q}$-divisor on $S'$ such that for any prime divisor $E'$ on $S'$, $\mathrm{coeff}_{E'}(B_{S'})=1-\mathrm{lct}(X'_{\eta_{E'}},\Delta'_{\eta_{E'}};f'^*E')$, where $X'_{\eta_{E'}}$ is a sufficiently small neighborhood of $f'^{-1}(\eta_{E'})$ and $\eta_{E'}$ is the generic point of $E'$. 
Set $M_{S'}$ as the $\mathbb{Q}$-divisor uniquely defined such that
\[
K_{X'}+\Delta'+\frac{1}{b}\mathrm{div}(\varphi)=f'^*(K_{S'}+B_{S'}+M_{S'}).
\]
We say that $B_{S'}$ and $M_{S'}$ are the {\it discriminant} and the {\it moduli} $\mathbb{Q}$-divisors on $S'$, respectively.
We see that $B_{S'}$ and $M_{S'}$ form $\mathbb{Q}$-$b$-divisors and it is known that $\{M_{S'}\}_{S'}$ is a nef $\mathbb{Q}$-$b$-Cartier $b$-divisor by \cite[Theorem 0.2]{ambro2004shokurov}, i.e. there exists a projective birational morphism $g\colon S'\to S$ such that $M_{S'}$ is a nef $\mathbb{Q}$-Cartier $\mathbb{Q}$-divisor and for any projective birational morphism $h\colon S''\to S'$, $M_{S''}=h^*M_{S'}$.

\subsection{Moduli of boundary polarized CY pairs of dimension one with rational coefficients}

We will recall the definitions of the family of one-dimensional boundary polarized CY pairs in \cite{babwild}.
In \cite{babwild}, a {\it family of boundary polarized CY pairs} is defined as $f\colon (X,\Delta+D)\to S$ of dimension $n$ over a locally Noetherian scheme $S$ with index $N\in\mathbb{Z}_{>0}$ and coefficients $(r,s)\in\mathbb{Q}^2_{>0}$ consists of 
    \begin{enumerate}
        \item $f\colon X\to S$ is a flat proper morphism of relative dimension $n$,
        \item relative Cartier divisors $\Delta^{\mathrm{Div}}$ and $D^{\mathrm{Div}}$, where $\Delta=r\Delta^{\mathrm{Div}}$ and $D=sD^{\mathrm{Div}}$
    \end{enumerate}
    with $(X_b,\Delta_b+D_b)$ is a boundary polarized CY pair for any $b\in S$, $\omega_{X/S}^{[N]}(N\Delta+ND)\sim_S\mathcal{O}_X$ and $\omega_{X/S}^{[m]}(m\Delta+nD)$ commutes with arbitrary base change for all $m,n\in\mathbb{Z}_{>0}$ such that $mr,ns\in\mathbb{Z}$.
    We call this a {\it family of boundary polarized CY pairs in the sense of} \cite{babwild}.

The above definition is complicated because if the relative dimension of $f$ is more than one, then there are many subtlety in handling relative Mumford divisors.
We are mainly interested in the case where the relative dimension one and each geometric fiber is a nodal curve and we can modify the definition to the simpler one as follows. 
    Note that for any $N\in\mathbb{Z}_{>0}$ such that $Nr\in\mathbb{Z}$ and $Ns\in\mathbb{Z}$, we see that $\omega_{X/S}^{[N]}(N\Delta+ND)\sim_S\mathcal{O}_X$ by \cite[Corollary 2.65]{kollar-moduli}.
    Therefore, in the above definition, we can remove the condition on $N$.
    Furthermore, $\omega_{X/S}^{[m]}(m\Delta+nD)$ commutes with arbitrary base change for all $m,n\in\mathbb{Z}_{>0}$ such that $m\Delta$ and $nD$ are relative Mumford $\mathbb{Z}$-divisors since we assume the relative dimension of $f$ is one and all fibers are nodal curves.

    Thus, we may define the following.
    \begin{defn}
        A {\it family of one-dimensional boundary polarized CY pairs} $f\colon(X,\Delta+D)\to S$ over a locally Noetherian scheme $S$ consists of the following data:
        \begin{enumerate}
        \item $f\colon X\to S$ is a flat proper morphism of relative dimension one and each geometric fiber is a connected nodal curve,
        \item $\Delta$ and $D$ are relative Mumford $\mathbb{R}$-divisor such that $K_{X/S}+\Delta$ is $\mathbb{R}$-Cartier, $D$ is $f$-ample, $K_{X/S}+\Delta+D\sim_{\mathbb{R},S}0$ and $(X_{\bar{s}},\Delta_{\bar{s}}+D_{\bar{s}})$ is a connected boundary polarized CY pair for any geometric points $\bar{s}\in S$.
    \end{enumerate}
    \end{defn}

On the other hand, we need the following moduli, which will be used to deduce the projectivity of our moduli space of boundary polarized CY quasimaps later.

\begin{defn}\label{defn:usual-ABB+}
    We set $v,r,s$ as positive rational numbers. Then, we set the following categories fibered in groupoids $\mathcal{M}^{\mathrm{bpCY}}_{r,s,v}$:
    \begin{itemize}
        \item[Objects] The objects of $\mathcal{M}^{\mathrm{bpCY}}_{r,s,v}(S)$ is a family of boundary polarized CY pairs $(X,\Delta+D)\to S$ of dimension one over a locally Noetherian scheme $S$ with some index $N_0\in\mathbb{Z}_{>0}$ and coefficients $(r,s)\in\mathbb{Q}^2_{>0}$ such that $\mathrm{deg}_{X_b}(D_b)=v$ for any $b\in S$ in the sense of \cite{babwild},
        \item[Morphisms] A morphism $\varphi$ from an object $(X,\Delta+D)\to B$ to another object $(X',\Delta'+D')\to B'$ is a morphism $\varphi\colon X\to X'$ compatible with the natural morphism $B\to B'$ that induces an isomorphism $(X,\Delta+D)\to (X'\times_{B'}B,\Delta'_{B}+D'_B)$ over $B$. 
    \end{itemize}
    \end{defn}
    Here, we do not need to fix $N_0$. Indeed, if you put $N_0$ as the lowest common multiple of the denominators of $s$ and $r$, then $\omega_{X/S}^{[N_0]}(N_0\Delta+N_0D)\sim_S\mathcal{O}_X$ automatically holds by \cite[Proposition 3.3]{Hassett}.
    Thus, by \cite[Theorem 3.7]{babwild}, $\mathcal{M}^{\mathrm{bpCY}}_{r,s,v}$ is an Artin stack of locally of finite type with affine diagonal.
    Furthermore, the Cartier indices of $D$ for all members of $\mathcal{M}^{\mathrm{bpCY}}_{r,s,v}$ are bounded by the denominator of $s$.
    By \cite[Proposition 3.12]{babwild}, $\mathcal{M}^{\mathrm{bpCY}}_{r,s,v}$ is of finite type.

    On the other hand, for any family of one-dimensional boundary polarized CY pairs $f\colon (X,\Delta+D)\to B$ with index $N$ and coefficients $(r,s)$, we set $\lambda_{\mathrm{Hodge},f,k}:=f_*\omega^{[k]}_{X/B}(k(\Delta+D))$ for any $k\in N\cdot \mathbb{Z}_{>0}$.
    Here, we can take $N$ as the lowest common multiple of the denominators of $r$ and $s$ for an arbitrary object of $\mathcal{M}^{\mathrm{bpCY}}_{r,s,v}$.
    This line bundle is functorial (cf.~\cite[Lemma 14.7]{babwild}) and defines a unique line bundle $\lambda_{\mathrm{Hodge},k}$ on $\mathcal{M}^{\mathrm{bpCY}}_{r,s,v}$ (cf.~\cite[Theorem 10.3]{Alper}).

\begin{thm}\label{thm:bpcY-one-dimensional-moduli}
 $\mathcal{M}^{\mathrm{bpCY}}_{r,s,v}$ admits a projective good moduli space $M^{\mathrm{bpCY}}_{r,s,v}$.   
 Furthermore, $\lambda_{\mathrm{Hodge},k}$ descends to a unique ample line bundle $\Lambda_{\mathrm{Hodge},k}$ on $M^{\mathrm{bpCY}}_{r,s,v}$ for any sufficiently divisible $k\in\mathbb{Z}_{>0}$ up to linear equivalence. 
\end{thm}

\begin{proof}
By \cite[Theorems 5.4 and 6.3]{babwild}, we see that $\mathcal{M}^{\mathrm{bpCY}}_{r,s,v}$ is $\mathsf{S}$-complete and $\Theta$-reductive.
Therefore, $\mathcal{M}^{\mathrm{bpCY}}_{r,s,v}$ admits a separated good moduli space $M^{\mathrm{bpCY}}_{r,s,v}$ of finite type by \cite[Theorem A]{AHLH}.
By \cite[Theorem 7.1]{babwild}, we see that $M^{\mathrm{bpCY}}_{r,s,v}$ is proper.

    By \cite[Proposition 14.9]{babwild}, $\lambda_{\mathrm{Hodge},k}$ descends to a unique line bundle $\Lambda_{\mathrm{Hodge},k}$ on $M^{\mathrm{bpCY}}_{r,s,v}$ for any sufficiently divisible $k\in\mathbb{Z}_{>0}$.
    We see by \cite[Theorem 14.13]{babwild} and \cite[Theorem 3.11]{kollar-proj} that to check the ampleness of $\Lambda_{\mathrm{Hodge},k}$, it suffices to show that the number of points of $M^{\mathrm{bpCY}}_{r,s,v}(\mathbb{K})$ corresponding to strictly lc or slc pairs are finite. 

    We say that a boundary polarized CY pair $(X,\Delta+D)$ of dimension one is {\it polystable} if $(X,\Delta+D)$ corresponds to a closed point of $\mathcal{M}^{\mathrm{bpCY}}_{r,s,v}(\mathbb{K})$.
    Fix a strict lc or slc polystable boundary polarized CY pair $(X,\Delta+D)$ of dimension one.
    Suppose that $X$ is normal.
    By \cite[Proposition 8.17]{babwild}, then $(X,\Delta+D)\cong (\mathbb{P}^1,\{0\}+\{\infty\})$. 
    Since $s$ and $r$ are fixed, there are only finitely many choices of $\Delta$ and $D$. This means that there are only finitely many isomorphism classes of normal polystable boundary polarized CY pairs that correspond to $\mathcal{M}^{\mathrm{bpCY}}_{r,s,v}(\mathbb{K})$.
    On the other hand, suppose that $X$ is not normal.
    By Lemma \ref{lem--babwild--chain--two--comps}, we see that $X$ is isomorphic to the union of two smooth rational curves intersecting transversally at one point.
    We note that $\mathbb{G}_{m}$ acts on $X$ such that on the normalization $\mathbb{P}^1\sqcup \mathbb{P}^1$ with the conductor supported on $\{0\}$, the action lifts as the natural action $\mathbb{G}_{m}$ on $\mathbb{P}^1$ by multiplication.
    This $\mathbb{G}_m$ induces a weakly special test configuration of $(X,\Delta+D)$.
    In this case, since $(X,\Delta+D)$ is polystable, we see that the normalization $(\overline{X},\overline{\Delta}+\overline{D}+G)$ of $(X,\Delta+D)$, where $G$ is the conductor divisor, is isomorphic to the copy of $(\mathbb{P}^1,\{0\}+\{\infty\})$. 
    There are only finitely many choices of $\Delta$ and $D$.
    Thus, the number of points of $M^{\mathrm{bpCY}}_{r,s,v}(\mathbb{K})$ corresponding to strictly lc or slc pairs are finite.
    By \cite[Theorem 14.13]{babwild} and \cite[Theorem 3.11]{kollar-proj}, $\Lambda_{\mathrm{Hodge},k}$ is ample and hence $M^{\mathrm{bpCY}}_{r,s,v}$ is projective (cf.~\cite[Tag 0D32]{stacksproject}).
\end{proof}

We call $\Lambda_{\mathrm{Hodge}}:=\frac{1}{k}\Lambda_{\mathrm{Hodge},k}$ the {\it Hodge line bundle} on $M^{\mathrm{bpCY}}_{r,s,v}$.
This $\Q$-line bundle is also ample and functorial.

\begin{defn}[Weighted log pairs]\label{defn:weightedpairs-isom}
Let $f\colon X\to S$ be a flat projective family of nodal connected curves.
    Let $D_1,\ldots,D_r$ be effective relative Mumford $\mathbb{R}$-divisors.
We denote it by $\vec{D}:=(D_1,\ldots,D_r)$ for $r\in\mathbb{Z}_{\ge0}$.
Consider $\vec{a}:=(a_1,\ldots,a_r)$, where $a_i\in\mathbb{R}_{\ge0}$.
We say that $f\colon (X,\vec{\Delta})\to S$ is a {\it weighted log family of pairs}, where $\vec{\Delta}:=(a_iD_i)_{i=1}^r$.
If $(X,\sum_{i=1}^ra_iD_i)$ is locally stable over $S$, then we say that $(X,\vec{\Delta})$ is {\it locally stable} over $S$, where $\vec{a}\cdot \vec{D}:=\sum_{i=1}^ra_iD_i$.
Throughout this paper, we consider $r=0$ to be the case where $\vec{\Delta}=0$.
We say that a morphism $g\colon X\to X'$ defines an {\it isomorphism} of $f\colon (X,\vec{\Delta})\to S$ and $f'\colon (X',\vec{\Delta}')\to S$ if $g$ is an isomorphism of $S$-schemes such that $a_i=a'_i$ and if $a_i\ne0$, then $g_*D_i=D'_i$ for any $1\le i\le r$, where $\vec{\Delta}':=(a'_iD'_i)_{i=1}^r$.

\begin{enumerate}
    \item 
If $(X,\sum_{i=1}^ra_iD_i)$ is locally stable over $S$ and $K_{X/S}+\sum_{i=1}^ra_iD_i$ is $f$-ample, then we say that $f\colon (X,\vec{\Delta})\to S$ is a {\it stable family of (weighted) log pairs}.

\item If $(X,\sum_{i=1}^ra_iD_i)$ is a family of K-semistable log Fano pairs, then we say that $f\colon (X,\vec{\Delta})\to S$ is a {\it family of K-semistable (weighted) log Fano pairs}.

\item Take $r'\in [0,r-1]\cap\mathbb{Z}$ and let $\vec{a}':=(a_1,\ldots,a_{r'})$, $\vec{a}'':=(a_{r'+1},\ldots,a_{r})$, $\vec{D}':=(D_1,\ldots,D_{r'})$ and $\vec{D}'':=(D_{r'+1},\ldots,D_{r})$. In this notion, if $r'=0$, we understand $\vec{a}'=0$ and $\vec{D}'=\emptyset$. 
If $(X,\sum_{i=1}^{r'}a_iD_i+\sum_{i=r'+1}^{r}a_iD_i)$ is a family of one-dimensional boundary polarized CY pairs, where $\mathcal{O}_{X}(\sum_{i=r'+1}^{r}a_iD_i)$ is $f$-ample, then we say that $f\colon (X,\vec{\Delta}_1+\vec{\Delta}_2)\to S$ is a {\it family of one-dimensional boundary polarized CY pairs, (where $\vec{a}''\cdot\vec{D}'')$ is $f$-ample)}. Here, $\vec{\Delta}_1:=(a_iD_i)_{i=1}^{r'}$ and $\vec{\Delta}_2:=(a_iD_i)_{i=r'+1}^{r}$.
\end{enumerate}
\end{defn}

\subsection{Log quasimaps}
We recall the definition of the quasimap.

\begin{defn}[Quasimaps]\label{def:qmaps}
Fix a closed immersion $X\subset \mathbb{P}^N$ of a projective variety and let $\mathrm{Cone}(X)\subset\mathbb{A}^{N+1}$ be the affine cone of $X$.
Let $C$ be a nodal proper curve.
We say that $q\colon C\to [\mathrm{Cone}(X)/\mathbb{G}_m]$ is a {\it quasimap} if $q$ maps any generic point of $C$ into $X\cong [(\mathrm{Cone}(X)\setminus\{0\})/\mathbb{G}_m]$.  
\end{defn}
If $X=\mathbb{P}^N$, then the information of $q$ corresponds to a homomorphism
\[
q_{\mathrm{hom}}\colon \mathcal{O}_{C}^{\oplus N+1}\to \mathscr{L}
\]
for some line bundle $\mathscr{L}$ such that $q_{\mathrm{hom}}$ is not a zero map around any generic point of any irreducible component of $C$ by the argument of \cite[Definition 3.3]{HH2}.
Note that $\mathscr{L}$ and $q_{\mathrm{hom}}$ as above are unique up to isomorphism.
In the general case, considering the composition of $q$ and the canonical $\mathbb{G}_m$-equivariant closed immersion $\mathrm{Cone}(X)\hookrightarrow \mathbb{A}^{N+1}$, we can yield $q_{\mathrm{hom}}$ and $\mathscr{L}$ as above.
Let $\mathrm{Bs}(q):=\sum_{p\in C}l(\mathrm{Coker}(q_{\mathrm{hom}})\otimes\mathcal{O}_{C,p})[p]$ and we say that this is the fixed part of $q$.
We set $\mathrm{deg}\,q:=\mathrm{deg}_C\mathscr{L}$.

\begin{defn}[Log quasimaps]\label{def:log_qmaps}
In the situation of Definition \ref{def:qmaps}, we say that $q$ is  {\it deminormal} if the support of $\mathrm{Bs}(q)$ is disjoint from the singular locus of $C$.
Moreover, let $\Delta$ be a $\mathbb{R}$-divisor on $C$ whose support is disjoint from the singular locus of $C$.
Then, we say that $q\colon (C,\Delta)\to [\mathrm{Cone}(X)/\mathbb{G}_m]$ is a {\it log quasimap}.
If $C$ has no nodal singularities, we say that $q$ is {\it normal}.
In this case, we also have a homomorphism
\[
q_{\mathrm{hom}}\colon \mathcal{O}_{C}^{\oplus N+1}\to \mathscr{L}
\]
that corresponds to $q$ as above.
We say $\mathscr{L}$ is the {\it line bundle associated with} $q$.

Fix $w\in\mathbb{R}_{>0}$. 
We say that a log quasimap $q$ is {\it slc (resp.~klt) with respect to the weight $w$} if $\mathrm{Bs}(q)$ is not supported at the singular locus of $C$ and $(C,w\mathrm{Bs}(q)+\Delta)$ is slc (resp.~klt). 
We say that $q$ is {\it lc with respect to the weight $w$} if $q$ is slc and normal.
For any slc quasimap $q$ with respect to $w$, we say that $q$ is 
\begin{enumerate}
    \item a {\it stable log quasimap} with respect to the weight $w$ if $K_C+\Delta+w\mathscr{L}$ is ample, where $\mathscr{L}$ is the line bundle associated with $q$;
    \item a {\it log Fano quasimap} with respect to the weight $w$ if $-(K_C+\Delta+w\mathscr{L})$ is ample and $q$ is klt, where $\mathscr{L}$ is the line bundle associated with $q$; 
    \item a {\it boundary polarized Calabi--Yau (CY) quasimap} with respect to the weight $(w'',w')$ if $(K_C+\Delta+(w''+w')\mathscr{L})\sim_{\R}0$ and $w'\mathscr{L}+D$ is ample for some effective $\mathbb{R}$-divisor $D\le \Delta$ with $w',w''\in\mathbb{R}_{\ge0}$ such that $w'+w''=w$; 
\end{enumerate}
where $\mathscr{L}$ is the line bundle associated with $q$. Here, we say that $w'\mathscr{L}+D$ is the {\it polarization part} in (3).

We will omit the weight if there is no fear of confusion.
Furthermore, we set $\mathrm{Sec}_q(\mathscr{L})$ as the set of sections of $\Gamma(\mathscr{L})$.
\end{defn}

By the definition, we see that $(C,\Delta+\frac{w}{l}\sum_{i=1}^l\mathrm{div}(f_i))$ is a boundary polarized CY pair if $q$ is a boundary polarized CY quasimap as (3), $f_i$'s are general sections corresponding to $e_i\in \mathbb{K}^{N+1}$, and $l:=\lceil w+w'\rceil$.

Note that there exists an open locus $U\subset C$ such that $U$ contains all the singular points of $C$ and $q|_U$ factors through $X$ if $q$ is slc.
By the valuative criterion for the properness, we can extend $q|_U$ to the morphism $\overline{q|_U}\colon C\to X$. If $\overline{q|_U}$ is a constant morphism, then we say that $q$ is {\it constant}. 
We remark that $q_{\mathrm{hom}}$ is surjective if and only if $q$ factors through the open substack $X\subset [\mathrm{Cone}(X)/\mathbb{G}_m]$. In this case, we say that $q$ is a {\it map} to $X$.

\begin{rem}
\begin{enumerate}
    \item If $w>1$ and $q$ is stable, then $q$ is a {\it stable log map}, which and whose moduli theory and wall crossing theory are discussed in \cite{bayer2009stability,Al-Guy}.
    \item The concept of quasimaps is first introduced by Givental \cite{givental1998mirror} and also studied in \cite{morrison1995summing}. As in \cite{marian2011moduli,ciocan2014stable,Toda}, the concept of quasimaps is enlarged to a map to some Artin stacks.
\end{enumerate}
\end{rem}

Let $S$ be a scheme. Let $\pi\colon \mathcal{C}\to S$ be a flat proper family of nodal curves with a morphism $q\colon \mathcal{C}\to [\mathrm{Cone}(X)_S/\mathbb{G}_{m,S}]$.  
Let $\Delta$ be a K-flat relative Mumford $\mathbb{R}$-divisor on $\mathcal{C}$.
We say that $q\colon(\mathcal{C},\Delta)\to [\mathrm{Cone}(X)_S/\mathbb{G}_{m,S}]$ is a {\it family of log quasimaps} if $q_{\bar{s}}$ is a deminormal log quasimap for any geometric point $\bar{s}\in S$.
Then, the information of $q$ corresponds to a homomorphism 
\[
q_{\mathrm{hom}}\colon \mathcal{O}_{\mathcal{C}}^{\oplus N+1}\to \mathscr{L}
\]
for some line bundle $\mathscr{L}$ on $\mathcal{C}$ uniquely up to isomorphism as above.
We call $q_{\mathrm{hom}}$ the {\it induced homomorphism} and $\mathscr{L}$ the {\it line bundle associated with} $q$.

For any two families of log quasimaps $q^{(1)}\colon(\mathcal{C}^{(1)},\Delta^{(1)})\to [\mathrm{Cone}(X)_S/\mathbb{G}_{m,S}]$ and $q^{(2)}\colon(\mathcal{C}^{(2)},\Delta^{(2)})\to [\mathrm{Cone}(X)_S/\mathbb{G}_{m,S}]$, we say that $\varphi\colon q^{(1)}\to q^{(2)}$ is an {\it isomorphism of families of log quasimaps} if $\varphi\colon \mathcal{C}^{(1)}\to\mathcal{C}^{(2)}$ is an isomorphism over $S$ such that $\varphi_*\Delta^{(1)}=\Delta^{(2)}$ and $q^{(2)}\circ\varphi$ and $q^{(1)}$ are isomorphic.

Furthermore, if $\vec{\Delta}=(a_i D_i)_{i=1}^r$, where each $D_i$ is an effective relative Mumford divisor and $a_i\in(0,1]$, then we say that $q\colon (\mathcal{C},\vec{\Delta})\to [\mathrm{Cone}(X)_S/\mathbb{G}_{m,S}]$ is a {\it family of weighted log quasimaps} over $S$ if $q_{\bar{s}}\colon (\mathcal{C}_{\bar{s}},\Delta_{\bar{s}})\to[\mathrm{Cone}(X)_{\overline{\kappa(s)}}/\mathbb{G}_{m,\overline{\kappa(s)}}]$ is deminormal for any geometric point $\bar{s}\in S$ given by $\overline{\kappa(s)}$. Here, we set $\vec{a}:=(a_1,\ldots,a_r)$, $\vec{D}:=(D_1,\ldots,D_r)$ and $\vec{a}\cdot\vec{D}:=\Delta$.
To simplify notations, let $q\colon (\mathcal{C},\vec{a}\cdot\vec{D})\to [\mathrm{Cone}(X)_S/\mathbb{G}_{m,S}]$ denote $q\colon (\mathcal{C},\vec{\Delta})\to [\mathrm{Cone}(X)_S/\mathbb{G}_{m,S}]$ abusively if there is no fear of confusion.

For any two families of weighted quasimaps $q^{(1)}\colon(\mathcal{C}^{(1)},\vec{a}^{(1)}\cdot\vec{\mathcal{D}}^{(1)})\to [\mathrm{Cone}(X)_S/\mathbb{G}_{m,S}]$ and $q^{(2)}\colon(\mathcal{C}^{(2)},\vec{a}^{(2)}\cdot\vec{\mathcal{D}}^{(2)})\to [\mathrm{Cone}(X)_S/\mathbb{G}_{m,S}]$, we say that $\varphi\colon q^{(1)}\to q^{(2)}$ is an {\it isomorphism of families of weighted quasimaps} if $\varphi$ is an isomorphism of $q^{(1)}$ and $q^{(2)}$ as families of log quasimaps such that $\varphi_*D_i^{(1)}=D^{(2)}_i$ and $a^{(1)}_i=a^{(2)}_i$. 

Fix $w\in\mathbb{R}_{>0}$.
We say that the above $q$ is {\it locally stable with respect to} $w$ or {\it a locally stable family of quasimaps} if $q_{\bar{s}}\colon (\mathcal{C}_{\bar{s}},\vec{a}\cdot\vec{D}_{\bar{s}})\to[\mathrm{Cone}(X)_{\overline{\kappa(s)}}/\mathbb{G}_{m,\overline{\kappa(s)}}]$ is slc with respect to the weight $w$ for any geometric point $\bar{s}\in S$.
If $w>1$, it is easy to see that $q$ factors through $X_S$.


If there is no fear of confusion, we will omit the numerical data including the weight.

\begin{defn}[K-stability of quasimaps]
Let $q\colon (C,\Delta)\to [\mathrm{Cone}(X)/\mathbb{G}_m]$ be a log quasimap with weight $w$ of degree $d$.    
Let $L$ be an ample $\R$-line bundle.
Consider $(\mathcal{C},\mathcal{L})$ an arbitrary semiample deminormal test configuration for $(C,\Delta;L)$.
Then we set the {\it Donaldson-Futaki invariant} with respect to $q$ as $$\mathrm{DF}_q(\mathcal{C},\mathcal{L}):=\mathrm{DF}_{\Delta+w\mathrm{div}(f)}(\mathcal{C},\mathcal{L}),$$
where $f\in \mathrm{Sec}_q(\mathscr{L})$ is a section corresponding to a general element of $\mathbb{K}^{N+1}$ and $\mathscr{L}$ is the associated line bundle.  
It is easy to check that $\mathrm{DF}_q(\mathcal{C},\mathcal{L})$ is independent of the choice of $f$.

We say that $q$ is {\it K-semistable with respect to $L$} if $\mathrm{DF}_q(\mathcal{C},\mathcal{L})\ge0$ for any semiample deminormal test configuration $(\mathcal{C},\mathcal{L})$ for $(C,\Delta;L)$.
If $q$ is K-semistable and $\mathrm{DF}_q(\mathcal{C},\mathcal{L})>0$ for any semiample deminormal test configuration $(\mathcal{C},\mathcal{L})$ for $(C,\Delta;L)$ such that $\mathcal{L}$ is ample and there exists no $\mathbb{G}_m$-equivariant isomorphism $\sigma\colon(\mathcal{C},\overline{\Delta\times\mathbb{G}_m})\cong(C\times\mathbb{A}^1,\Delta\times\mathbb{A}^1)$ over $\mathbb{A}^1$, then we say that $q$ is {\it K-stable with respect to $L$}.
\end{defn}

Before stating the properties of K-stability of quasimaps, we recall the following notion, which is a generalization of the twisted K-stability of Dervan \cite{dervan2016uniform}.

\begin{defn}[Log Twisted K-stability]
 Let $(X,\Delta;L)$ be a polarized deminormal subpair with a semiample $\mathbb{R}$-line bundle $H$.
 Then we call $(X,\Delta;H;L)$ a {\it log-twisted subpair} and if $\Delta$ is further effective, then we say that $(X,\Delta;H;L)$ is a {\it log-twisted pair}.
 We say that $(X,\Delta;H;L)$ is {\it log-twisted K-semistable} if $\mathrm{DF}_{\Delta}(\mathcal{X},\mathcal{L})+\mathcal{J}^{H,{\mathrm{NA}}}(\mathcal{X},\mathcal{L})\ge 0$ for any semiample deminormal test configuration $(\mathcal{X},\mathcal{L})$ for $(X,\Delta,L)$.

    We say that $(X,\Delta;H;L)$ is {\it log-twisted K-stable} if $(X,\Delta;H;L)$ is log-twisted K-semistable and $\mathrm{DF}_{\Delta}(\mathcal{X},\mathcal{L})+\mathcal{J}^{H,{\mathrm{NA}}}(\mathcal{X},\mathcal{L})> 0$ for any semiample deminormal test configuration $(\mathcal{X},\mathcal{L})$ for $(X,\Delta;L)$ such that $\mathcal{L}$ is ample and there exists no $\mathbb{G}_m$-equivariant isomorphism $\sigma\colon (\mathcal{X},\Delta_{\mathcal{X}})\cong (X\times\mathbb{A}^1,\Delta\times\mathbb{A}^1)$.
\end{defn}

The log twisted K-stability is also conjectured to be equivalent to the existence of special K\"ahler currents as the Yau-Tian-Donaldson conjecture and an important notion (cf.~\cite{ross2021twisted}).
Lemma \ref{lem;K-stqmaps and log-twisted} below shows that the K-stability of quasimaps is related to the log twisted K-stability in the following sense.

\begin{lem}\label{lem;K-stqmaps and log-twisted}
    Let $q\colon (C,\Delta)\to [\mathrm{Cone}(X)/\mathbb{G}_m]$ be a log quasimap with weight $w$ of degree $d$.
    Let $L$ be an ample $\mathbb{R}$-line bundle on $C$ and $\mathscr{L}$ the line bundle associated to $q$.
    Let $\mathrm{Bs}(q)$ be the fixed part of $q$ and $H:=\mathscr{L}-\mathrm{Bs}(q)$.
    Then $q$ is K-(semi)stable with respect to $L$ if and only if $(C,\Delta+w\mathrm{Bs}(q);wH;L)$ is log-twisted K-(semi)stable.
\end{lem}

\begin{proof}
Let $(\mathcal{C},\mathcal{L})$ be a semiample deminormal test configuration for $(C,\Delta+w\mathrm{Bs}(q);L)$.
Then it suffices to show that $\mathrm{DF}_{\Delta+w\mathrm{Bs}(q)}(\mathcal{C},\mathcal{L})+\mathcal{J}^{wH,\mathrm{NA}}(\mathcal{C},\mathcal{L})\ge0$ (resp.~$>0$) if and only if $\mathrm{DF}_{q}(\mathcal{C},\mathcal{L})\ge0$ (resp.~$>0$).
Let $\widehat{\mathcal{C}}$ be the normalization of $\mathcal{C}$.
To show this, we may replace $C$ with its normalization and $\mathcal{C}$ with a blow up of $\widehat{\mathcal{C}}$ and assume that $\rho\colon\widehat{\mathcal{C}}\to C\times\mathbb{A}^1$ by \cite[Remark 3.19]{boucksom2017uniform}, \cite[Proposition 3.2]{hattori2026decomposition} and \cite[Proposition 3.3]{fujita2019openness}.
Note that the linear system generated by sections of $\mathscr{L}$ corresponding to $\mathbb{K}^{N+1}$ is a globally generated linear system in $|H|$.
Thus, we can conclude the assertion by what we noted in Definition  \ref{defn:DF-invariant}.
\end{proof}

\begin{lem}\label{lem:well-balanced}
The situation as Lemma \ref{lem;K-stqmaps and log-twisted}.
If $q$ is K-semistable with respect to $L$, then $q$ is slc and 
\[
\frac{\mathrm{deg}_C(K_C+w\mathscr{L}+\Delta)}{\mathrm{deg}_{C}(L)}=\frac{\mathrm{deg}_{C'}((K_C+w\mathscr{L}+\Delta)|_{C'})}{\mathrm{deg}_{C'}(L|_{C'})}
\]
for any irreducible component $C'\subset C$, i.e.~there exists $\alpha\in\mathbb{R}$ such that $\alpha L$ is numerically equivalent to $K_C+w\mathscr{L}+\Delta$.    
\end{lem}

\begin{proof}
Suppose that $q$ is K-semistable with respect to $L$.
First, assume that $q$ is not slc at a closed point $p\in C$.
If $q$ is not lc but normal, then \cite[Proposition 2.24]{hattori2025k} and Lemma \ref{lem;K-stqmaps and log-twisted} show that $q$ is not K-semistable.
If $q$ is not normal, then consider a semiample deminormal test configuration $(\mathrm{Bl}_{(p,0)}(C\times\mathbb{A}^1),\mu^*(L\times\mathbb{A}^1)-\epsilon E)$ for any sufficiently small $\epsilon>0$, where $\mu\colon \mathrm{Bl}_{(p,0)}(C\times\mathbb{A}^1)\to C\times\mathbb{A}^1$ is the one point blow-up and $E$ is the exceptional divisor.
Note that since $q$ is deminormal, $p$ is not a nodal point.
In this case, it is easy to see that $\mathrm{DF}_{\Delta+w\mathrm{Bs}(q)}(\mathrm{Bl}_{(p,0)}(C\times\mathbb{A}^1),\mu^*(L\times\mathbb{A}^1)-\epsilon E)+\mathcal{J}^{wH,\mathrm{NA}}(\mathrm{Bl}_{(p,0)}(C\times\mathbb{A}^1),\mu^*(L\times\mathbb{A}^1)-\epsilon E)<0$ for any sufficiently small $\epsilon>0$ by the same argument as \cite[Propositions 4.11 and 9.12]{boucksom2017uniform}.
Therefore, Lemma \ref{lem;K-stqmaps and log-twisted} shows that $q$ is not K-semistable again.
The latter statement follows from Lemma \ref{lem;K-stqmaps and log-twisted} and \cite[Theorem 3.6]{hattori2026decomposition} (which was shown only for polarized deminormal pairs but the same argument works in the same way for the log-twisted case).
\end{proof}

The above lemma says that for the K-stability of quasimaps, it suffices to consider only the cases where $q$ is log general type, log Calabi-Yau, or the log Fano numerically.

\begin{prop}\label{prop:chara--K-stability--qmaps}
Let $q\colon (C,\Delta)\to [\mathrm{Cone}(X)/\mathbb{G}_m]$ be a slc quasimap with weight $w$ of degree $d$.
    Let $L$ be an ample $\mathbb{R}$-line bundle on $C$ and $\mathscr{L}$ the line bundle associated to $q$.
    Let $\mathrm{Bs}(q)$ be the fixed part of $q$ and $H:=\mathscr{L}-\mathrm{Bs}(q)$.
    Suppose that there exists $\alpha\in\R$ such that $\alpha L$ is numerically equivalent to $K_C+w\mathscr{L}+\Delta$.
    Then the following hold:
    \begin{enumerate}
        \item If $\alpha>0$, then $q$ is K-stable with respect to $L$,
        \item if $\alpha=0$  then $q$ is K-semistable with respect to $L$. Furthermore, $q$ is K-stable with respect to $L$ if and only if $q$ is klt, and 
        \item if $\alpha<0$ and $q$ is K-semistable with respect to $L$, then $q$ is klt. Furthermore, when $q$ is klt, $q$ is K-semistable (resp.~K-stable)  with respect to $L$ if and only if $\inf_{p\in C}\frac{2(1-\mathrm{ord}_p(\Delta+w\mathrm{Bs}(q)))}{\mathrm{deg}(-(K_C+w\mathscr{L}+\Delta))}\ge1$ (resp.~$>1$).
    \end{enumerate}
\end{prop}

\begin{proof}
Note that $q$ is klt (resp.~lc, slc) if and only if $(C,\Delta+w\mathrm{Bs}(q))$ is klt (resp.~lc, slc).
For the first assertion, we may replace $C$ with its normalization and we can conclude it by \cite[Proposition 2.24]{hattori2025k} and Lemmas \ref{lem;K-stqmaps and log-twisted} and \ref{lem:well-balanced}. 
For the second assertion, by \cite[Proposition 2.24]{hattori2025k}, it suffices to show that if $q$ is not klt, then $q$ is not K-stable with respect to $L$.
Suppose that $q$ is not klt around $p\in C$.
If $p$ is not nodal, then $\mathrm{DF}_{\Delta+w\mathrm{Bs}(q)}(\mathrm{Bl}_{(p,0)}(C\times\mathbb{A}^1),\mu^*(L\times\mathbb{A}^1)-\epsilon E)+\mathcal{J}^{wH,\mathrm{NA}}(\mathrm{Bl}_{(p,0)}(C\times\mathbb{A}^1),\mu^*(L\times\mathbb{A}^1)-\epsilon E)=0$ for any sufficiently small $\epsilon>0$ by the same argument as \cite[Propositions 4.11 and 9.12 and Lemma 7.25]{boucksom2017uniform}.
If $p$ is nodal, we can take the partial normalization $\nu\colon \tilde{C}\to C$ at $p$ and let $D$ be the conductor in $\tilde{C}$.
Note that $D=\{p_1,p_2\}$ and using $p_1$, we can construct a semiample deminormal test configuration $(\mathrm{Bl}_{(p_1,0)}(\tilde{C}\times\mathbb{A}^1),\mu^*(\nu^*L\times\mathbb{A}^1)-\epsilon E)$ for any sufficiently small $\epsilon>0$, where $\mu\colon \mathrm{Bl}_{(p_1,0)}(\tilde{C}\times\mathbb{A}^1)\to \tilde{C}\times\mathbb{A}^1$ is the canonical morphism and $E$ the exceptional curve.
Using the semiampleness of $H$, regard $H$ as an effective general $\mathbb{Q}$-divisor such that $\mathrm{ord}_p(H)=0$ for any $p\in \mathrm{Supp}\,\mathrm{Bs}(q)$. 
Suppose that $L=\sum_{i=1}^k \alpha_i L_i$ for some $\alpha_i\in\mathbb{R}_{>0}$ such that $\sum\alpha_i=1$ and ample $\Q$-line bundles $L_i$.
Since $\mu^*(\nu^*L_i\times\mathbb{A}^1)-\epsilon E$ is relatively ample over $\mathbb{A}^1$ for any sufficiently small $\epsilon>0$ and each $i$, we can take an effective general $\mathbb{Q}$-Cartier $\mathbb{Q}$-divisor $A_i\sim_\Q\mu^*(\nu^*L_i\times\mathbb{A}^1)-\epsilon E$ such that $$(\mathrm{Bl}_{(p_1,0)}(\tilde{C}\times\mathbb{A}^1),\overline{(D+\nu^*(wH+w\mathrm{Bs}(q)+\Delta))\times\mathbb{G}_m}+\eta\sum \alpha_iA_i+\mathrm{Bl}_{(p_1,0)}(\tilde{C}\times\mathbb{A}^1)_0)$$ is slc around $\mathrm{Bl}_{(p_1,0)}(\tilde{C}\times\mathbb{A}^1)_0$ for any sufficiently small $\eta>0$.
Then we can apply \cite[Theorem 11.38]{kollar-moduli} locally around $\mathrm{Bl}_{(p_1,0)}(\tilde{C}\times\mathbb{A}^1)_0$ to glue $\mathrm{Bl}_{(p_1,0)}(\tilde{C}\times\mathbb{A}^1)$ by the canonical involution of $\overline{(D\times\mathbb{G}_m)}$ and obtain a slc pair
\[
(\mathcal{C}_U,\Delta_{\eta}+\mathcal{C}_0)
\]
flat and projective over some open neighborhood $U\subset \mathbb{A}^1$ of $0$ that corresponds to 
$$(\mathrm{Bl}_{(p_1,0)}(\tilde{C}\times\mathbb{A}^1)|_U,(\overline{(D+\nu^*(wH+w\mathrm{Bs}(q)+\Delta))\times\mathbb{G}_m}+\eta\sum \alpha_iA_i+\mathrm{Bl}_{(p_1,0)}(\tilde{C}\times\mathbb{A}^1)_0)|_{\mathrm{Bl}_{(p_1,0)}(\tilde{C}\times\mathbb{A}^1)|_U}).$$
Note that 
\begin{align*}
    K_{\mathrm{Bl}_{(p_1,0)}(\tilde{C}\times\mathbb{A}^1)}+\overline{(D+\nu^*(wH+w\mathrm{Bs}(q)+\Delta))\times\mathbb{G}_m}+\eta\sum \alpha_iA_i\sim_{\mathbb{R}}\eta(\mu^*(\nu^*L\times\mathbb{A}^1)-\epsilon E).
\end{align*}
Therefore, patching $(C\times\mathbb{G}_m,L\times\mathbb{G}_m)$ and $(\mathcal{C}_U,\eta^{-1}(K_{\mathcal{C}_U}+\Delta_{\eta}))$ together, we obtain a semiample deminormal test configuration $(\mathcal{C},\mathcal{L})$ such that $\mathcal{L}$ is ample and $\mathcal{C}\not\cong C\times\mathbb{A}^1$.
We see by \cite[Remark 3.19]{boucksom2017uniform} and the definition of the non-Archimedean $J$-functional that
\begin{align*}
\mathrm{DF}_{\Delta+w\mathrm{Bs}(q)}(\mathcal{C},\mathcal{L})+\mathcal{J}^{wH,\mathrm{NA}}(\mathcal{C},\mathcal{L})&=\mathrm{DF}_{\nu^{-1}_*\Delta+D+w\nu^*\mathrm{Bs}(q)}(\mathrm{Bl}_{(p_1,0)}(\tilde{C}\times\mathbb{A}^1),\mu^*(\nu^*L\times\mathbb{A}^1)-\epsilon E) \\
&+\mathcal{J}^{w\nu^*H,\mathrm{NA}}(\mathrm{Bl}_{(p_1,0)}(\tilde{C}\times\mathbb{A}^1),\mu^*(\nu^*L\times\mathbb{A}^1)-\epsilon E)=0.
\end{align*}
Thus, we obtain (2).

Finally, we deal with (3). If $\alpha<0$ and $q$ is K-semistable with respect to $L$, then $q$ is klt by the same argument as the proof of \cite[Theorem 2.23]{hattori2025k}.
Therefore, $C\cong\mathbb{P}^1$.
The last assertion follows from \cite[Theorem A.16]{hattori2025k} when $\Delta$ has rational coefficients and $w\in\mathbb{Q}$.
For the general case, one can see this by hand since we are dealing with only the case $C=\mathbb{P}^1$.
Indeed, if there exists $p\in C$ such that $\frac{2(1-\mathrm{ord}_p(\Delta+w\mathrm{Bs}(q)))}{\mathrm{deg}(-(K_C+w\mathscr{L}+\Delta))}<1$ (resp.~$\le1$), then we see that a normal ample test configuration $\mathcal{C}:=\mathrm{Proj}_{\mathbb{A}^1}(\oplus_{m,\lambda\in\mathbb{Z}_{>0}}\mathscr{F}^\lambda_{\mathrm{ord}_p}H^0(\mathbb{P}^1,\mathcal{O}(m)))$ with its $\mathcal{O}(1)$ satisfies 
\[
\mathrm{DF}_{\Delta+w\mathrm{Bs}(q)}(\mathcal{C},\mathcal{O}(1))+\mathcal{J}^{wH,\mathrm{NA}}(\mathcal{C},\mathcal{O}(1))<0,\qquad (\text{resp. }\le0)
\]
as the calculation of the proof of \cite[Theorem A.15]{hattori2025k}, where
\[
\mathscr{F}^\lambda_{\mathrm{ord}_p}H^0(\mathbb{P}^1,\mathcal{O}(m)):=\{f\in H^0(\mathbb{P}^1,\mathcal{O}(m))|\mathrm{ord}_p(f)\ge\lambda\}.
\]
Conversely, set $\delta:=\inf_{p\in C}\frac{2(1-\mathrm{ord}_p(\Delta+w\mathrm{Bs}(q)))}{\mathrm{deg}(-(K_C+w\mathscr{L}+\Delta))}$.
Normalize $L$ and assume that $L=\mathcal{O}_{\mathbb{P}^1}(1)$ without loss of generality.
Then one can find a sequence of effective  $\mathbb{Q}$-divisors $\{\Delta_i\}_{i=1}^\infty$ with the same supports and $\{w_i\}_{i=1}^\infty\subset\mathbb{Q}_{>0}$ such that $\Delta_i$'s coefficients converge to $\Delta$'s coefficients in $i\to\infty$, $\lim_{i\to\infty}w_i=w$ and set $\delta_i:=\inf_{p\in C}\frac{2(1-\mathrm{ord}_p(\Delta_i+w_i\mathrm{Bs}(q)))}{\mathrm{deg}(-(K_C+w_i\mathscr{L}+\Delta_i))}$. Note that $\lim_{i\to\infty}\delta_i=\delta$.
For any deminormal and semiample test configuration $(\mathcal{C},\mathcal{L})$ for $(\mathbb{P}^1,\mathcal{O}(1))$, we see that
\begin{align*}
    &\mathrm{DF}_{\Delta+w\mathrm{Bs}(q)}(\mathcal{C},\mathcal{L})+\mathcal{J}^{wH,\mathrm{NA}}(\mathcal{C},\mathcal{L})\\
    &=\lim_{i\to\infty}(\mathrm{DF}_{\Delta_i+w_i\mathrm{Bs}(q)}(\mathcal{C},\mathcal{L})+\mathcal{J}^{w_iH,\mathrm{NA}}(\mathcal{C},\mathcal{L})).
\end{align*}
By \cite[Theorem A.6]{hattori2025k} (see also \cite[Definition 2.8]{hattori2025k} for the terminologies), we see that
\[
\mathrm{DF}_{\Delta_i+w_i\mathrm{Bs}(q)}(\mathcal{C},\mathcal{L})+\mathcal{J}^{w_iH,\mathrm{NA}}(\mathcal{C},\mathcal{L})\ge(\delta_i-1)\mathrm{deg}(-(K_C+w_i\mathscr{L}+\Delta_i))\mathcal{J}^{L,\mathrm{NA}}(\mathcal{C},\mathcal{L}).
\]
Thus,
\[
\mathrm{DF}_{\Delta+w\mathrm{Bs}(q)}(\mathcal{C},\mathcal{L})+\mathcal{J}^{wH,\mathrm{NA}}(\mathcal{C},\mathcal{L})\ge(\delta-1)\mathrm{deg}(-(K_C+w\mathscr{L}+\Delta))\mathcal{J}^{L,\mathrm{NA}}(\mathcal{C},\mathcal{L})
\]
and we obtain the assertion by the positivity of $\mathcal{J}^{L,\mathrm{NA}}(\mathcal{C},\mathcal{L})$ by \cite[Lemmas 2.7 and 2.12 and Corollary 3.4]{hattori2026decomposition}.
\end{proof}

By Proposition \ref{prop:chara--K-stability--qmaps}, we note that for any $\mathbb{K}\subset \mathbb{K}'$ field extension of algebraically closed fields, $q$ is K-semistable with respect to $L$ if and only if $q_{\mathbb{K}'}$ is K-semistable with respect to $L_{\mathbb{K}'}$.
Proposition \ref{prop:chara--K-stability--qmaps} shows that
stable quasimaps are K-stable with respect to their adjoint
polarization, whereas boundary polarized CY quasimaps are
K-semistable with respect to their polarization part.
In the Calabi--Yau case, K-stability holds precisely
when the quasimap is klt.
In this respect, we can regard Proposition \ref{prop:chara--K-stability--qmaps} as a quasimap analog of Theorem \ref{thm:K-stability--for--log-pairs--fundaments}.
Therefore, the moduli space of Theorem \ref{mainthm:stable-region} can be regarded as an example of K-moduli of quasimaps.
Furthermore, Proposition \ref{prop:chara--K-stability--qmaps} (3) gives a numerical criterion for the K-stability of the log Fano quasimaps as in Theorem \ref{thm:delta--log--pairs}.

\begin{defn}
Let $q\colon (C,\Delta)\to [\mathrm{Cone}(X)/\mathbb{G}_m]$ be a log Fano quasimap with weight $w$ of degree $d$, that is, $q$ is klt, $C\cong\mathbb{P}^1$ and $L:=-(K_{\P^1}+\Delta+w\mathcal{O}_{\mathbb{P}^1}(d))$ is ample.
We say that $q$ is {\it K-(semi)stable} if $q$ is K-(semi)stable with respect to $L$.

    Set $\vec{\Delta}=(a_i D_i)_{i=1}^r$, where each $D_i$ is an effective relative Mumford divisor and $a_i\in(0,1]$, then we say that a locally stable family $q\colon (\mathcal{C},\vec{\Delta})\to [\mathrm{Cone}(X)_S/\mathbb{G}_{m,S}]$ with weight $w$ is a {\it family of K-(semi)stable log Fano quasimaps} over $S$ with respect to the weight $w$ if $q_{\bar{s}}\colon (\mathcal{C}_{\bar{s}},\vec{a}\cdot\vec{D}_{\bar{s}})\to[\mathrm{Cone}(X)_{\overline{\kappa(s)}}/\mathbb{G}_{m,\overline{\kappa(s)}}]$ with weight $w$ is a K-(semi)stable log Fano quasimap for any geometric point $\bar{s}\in S$ given by $\overline{\kappa(s)}$.
\end{defn}

We put the following $\mathbb{R}$-coefficient version of \cite[Lemma 3.12]{HH2}.

\begin{lem}\label{lem:HH2-modified}
    Let $q\colon (\mathbb{P}^1,\Delta)\to[\mathrm{Cone}(X)/\mathbb{G}_m]$ be a log Fano quasimap of degree $d$ with weight $w>0$.
    Let $\mathscr{L}$ be the associated line bundle.
    Put $v:=\mathrm{deg}\,(-(K_{\P^1}+\Delta))-wd$.
    For any $l\in\mathbb{Z}_{>0}$ such that $\frac{w}{l}<1-\frac{v}{2}$, the following are equivalent.
    \begin{enumerate}
    \item $q$ is K-semistable.
        \item There exist elements $e_1,\ldots,e_l\in \mathbb{K}^{N+1}$ such that $(\mathbb{P}^1,\Delta+\frac{w}{l}\sum_{i=1}^l\mathrm{div}(f_i))$ is K-semistable, where $f_i\in\mathrm{Sec}_q(\mathscr{L})$ is the corresponding section to $e_i$.
        \item For any general elements $e_1,\ldots,e_l\in \mathbb{K}^{N+1}$, $(\mathbb{P}^1,\Delta+\frac{w}{l}\sum_{i=1}^l\mathrm{div}(f_i))$ is K-semistable, where $f_i\in\mathrm{Sec}_q(\mathscr{L})$ is the corresponding section to $e_i$.
    \end{enumerate}
\end{lem}

\begin{proof}
 In the situation of $(3)$, by the choice of $f_i$, we see that the support of $\frac{w}{l}\sum_{i=1}^l\mathrm{div}(f_i)-w\mathrm{Bs}(q)$ is disjoint from $\mathrm{Bs}(q)+\Delta$ and  $\mathrm{ord}_{p'}(\frac{w}{l}\sum_{i=1}^l\mathrm{div}(f_i)))\le\max\{\max_{p\in\mathbb{P}^1} \mathrm{ord}_p(\Delta+w\mathrm{Bs}(q)),\frac{w}{l}\}$ for any closed point $p'\in\P^1$.    
 Therefore, we see that $(1)\Rightarrow (3)$ by Example \ref{ex:one-dim-K-st} and Proposition \ref{prop:chara--K-stability--qmaps}.
 $(3)\Rightarrow (2)$ is trivial. In the situation of $(2)$, take an arbitrary semiample deminormal test configuration $(\mathcal{X},\mathcal{L})$ for $(\mathbb{P}^1,-(K_{\P^1}+\Delta+w\mathscr{L}))$.
 By the K-semistability of $(\mathbb{P}^1,\Delta+\frac{w}{l}\sum_{i=1}^l\mathrm{div}(f_i))$ and Lemma \ref{lem:J-functional--computation}, we see that 
 \[
 \mathrm{DF}_{q}(\mathcal{X},\mathcal{L})=\mathrm{DF}_{\Delta+w\mathrm{Bs}(q)}(\mathcal{X},\mathcal{L})+\mathcal{J}^{w(\mathscr{L}-\mathrm{Bs}(q)),\mathrm{NA}}(\mathcal{X},\mathcal{L})\ge \mathrm{DF}_{\Delta+\frac{w}{l}\sum_{i=1}^l\mathrm{div}(f_i)}(\mathcal{X},\mathcal{L})\ge0.
 \]
 Thus, we obtain $(2)\Rightarrow(1)$.
\end{proof}

\subsection{Deformation theory of log quasimaps}

We put the following useful proposition to detect the smoothness of the moduli stack parameterizing log quasimaps with arithmetic genus zero.

\begin{prop}\label{prop:smoothness-of-stacks}
Let $A$ be an artinian local ring finite over $\mathbb{K}$ with the maximal ideal $\mathfrak{m}$ and an ideal $I\subset A$ such that $I\cdot\mathfrak{m}=0$.
Let $q\colon (\mathcal{C},\Delta)\to [\mathbb{A}_{A/I}^{N+1}/\mathbb{G}_{m,A/I}]$ be a locally stable family of quasimaps over $\mathrm{Spec}\,A/I$.
Let $s$ be the closed point of $\mathrm{Spec}\,A/I$.
Suppose that $p_a(\mathcal{C}_{s})=0$.
Then there exists a locally stable family of quasimaps $\bar{q}\colon (\overline{\mathcal{C}},\overline{\Delta})\to [\mathbb{A}_{A}^{N+1}/\mathbb{G}_{m,A}]$ over $\mathrm{Spec}\,A$ that is an extension of $q$. 
\end{prop}

\begin{proof}
Before dealing with each case, we explain fundamental results.
By \cite[Lemma 1.3]{DM}, there exists a flat family of nodal curves $\overline{\mathcal{C}}\to \mathrm{Spec}\,A$ that is an extension of $\mathcal{C}$.
Moreover, suppose that $\Delta=\sum a_i D_i$ is the irreducible decomposition.
Each $D_i$ is a relative Cartier divisor and hence $(D_i)_s$ is nef.
By Lemma \ref{lem--vanish} below, we see that $H^1(\mathcal{C}_s,\mathcal{O}_{\mathcal{C}_s}((D_i)_s))=0$ for each $i$.
Using $H^2(\mathcal{C}_s,\mathcal{O}_{\mathcal{C}_s})=0$, we can take an extension $\overline{D}_i$ of $D_i$ on $\overline{\mathcal{C}}$. 
Therefore, it suffices to extend the structure of a family of quasimaps.
Let $\mathscr{L}$ be the line bundle associated with $q$.
Since $H^2(\mathcal{C}_s,\mathcal{O}_{\mathcal{C}_s})=0$, we can take a line bundle $\overline{\mathscr{L}}$ on $\overline{\mathcal{C}}$ that is an extension of $\mathscr{L}$.
 We assert that $H^1(\mathcal{C}_{s},\mathscr{L}_s)=0$.
 Indeed, note that $\mathscr{L}_s$ is nef since $q_{s}$ is a slc quasimap and hence we can apply Lemma \ref{lem--vanish} below.

 Let $f_0,\ldots,f_{N}$ be the global sections of $\mathscr{L}$ that correspond to a general basis $\{e_0,\ldots,e_N\}$ of $\mathbb{A}^{N+1}_{\mathbb{K}}$.
 Since $H^1(\mathcal{C}_{s},\mathscr{L}_s)=0$, $f_i$ can be extended to global sections $\overline{f}_{i}$ of $\overline{\mathscr{L}}$.
 Using $\overline{f}_{i}$, we can obtain a family of quasimaps $\overline{q}\colon \overline{\mathcal{C}}\to [\mathbb{A}^{N+1}_{A}/\mathbb{G}_{m,A}]$.
\end{proof}

\begin{lem}\label{lem--vanish}
    Let $C$ be a proper connected nodal curve with $p_a(C)=0$.
    Let $L$ be a nef line bundle and $p\in C$ an arbitrary smooth closed point.
    Then $H^1(C,L(-p))=0$ and $H^1(C,L)=0$.
\end{lem}

\begin{proof}
 The former assertion implies the latter assertion because we have a surjective morphism $H^1(C,L(-p))\to H^1(C,L)$.
 We show the former assertion by induction on the number of irreducible components of $C$.
 First, suppose that $C$ is irreducible. Then $\mathrm{deg}_C(L(-p))\ge-1$ and we obtain the assertion in this case.
 Suppose that $C$ is reducible. 
 Since $p_a(C)=0$, $C$ is a tree of smooth rational curves.
 Let $E$ be the irreducible component that contains $p$.
 Let $C'=\overline{C\setminus E}$.
 If $C'\cap E=\{q_1,\ldots,q_l\}$, then $C'$ consists of $l$ connected components and each connected component intersects with $E$ at exactly one point.
 Consider the following exact sequence
 \[
 H^1\left(C',L|_{C'}\left(-\sum_{k=1}^lq_k\right)\right)\to H^1(C,L(-p))\to H^1(E,L|_E(-p))\to0.
 \]
 By the induction hypothesis, we have $h^1(C',L|_{C'}(-\sum_{k=1}^lq_k))=h^1(E,L|_E(-p))=0$.
 Thus, we obtain $H^1(C,L(-p))=0$.
\end{proof}

Furthermore, to achieve the moduli wall crossing for stable families of log quasimaps, we put the following proposition.

\begin{prop}\label{prop:contraction-thm}
 In Proposition \ref{prop:push--linear-equivalence}, suppose that $-(K_{C/S}+\Delta+w\mathscr{L})$ is $h$-nef. Let $q\colon (C,\Delta)\to [\mathrm{Cone}(X)_S/\mathbb{G}_{m,S}]$ be a family of locally stable quasimaps of degree $d$ with weight $w>0$.    
 Then there exists a locally stable family $q'\colon (C',h_*\Delta)\to [\mathrm{Cone}(X)_S/\mathbb{G}_{m,S}]$ of log quasimaps of degree $d$ with weight $w$ uniquely up to isomorphisms that satisfies the following:
 \begin{itemize}
     \item[]  Let $U\subset S$ be a non-empty open subset and $\varphi_0,\ldots,\varphi_N \in \mathrm{Sec}_{q_U}(\mathscr{L}_U)$ the sections that correspond to a basis $e_0,\ldots,e_{N}\in\mathbb{A}^{N+1}_{\mathbb{K}}$ and $\varphi\in \mathrm{Sec}_{q_U}(\mathscr{L}_U)$ the section that corresponds to a general element $e\in\mathbb{A}^{N+1}_{\mathbb{K}}$ such that $\mathrm{div}(\varphi)$ and $\mathrm{div}(\varphi_i)$ are relative Mumford divisors of $C_U$ over $U$, where $\mathscr{L}$ is the line bundle associated with $q$.
 Then the line bundle $\mathscr{L}'$ associated with $q'$ satisfies that $\mathscr{L}'_U\sim \mathcal{O}_{C'_U}((h_{U})_*\mathrm{div}(\varphi|_{C_U}))$ and $q'_U$ is defined by $\varphi'_i:=((h_U)_\sharp(\frac{\varphi_i}{\varphi}))\cdot s$, where $s$ is the section of $\Gamma(C'_U,\mathscr{L}'_U)$ corresponding to $h_*\mathrm{div}(\varphi)|_{C'_U}$ and $(h_U)_\sharp$ is as defined in Proposition \ref{prop:push--linear-equivalence}. 
 \end{itemize}
\end{prop}

\begin{proof}
First, we observe that for any point $p\in S$, there exists an open neighborhood $V$ such that for any general $e'\in \mathbb{K}^{N+1}$, the corresponding section $f\in \mathrm{Sec}_{q_V}(\mathscr{L}_V)$ such that $(h_{V})_*\mathrm{div}(f)$ is a relative Mumford divisor. To show this, we may assume that $S=\mathrm{Spec}\,\mathbb{K}$. 
Let $p\in C'$ be a closed point.
For any $f\in\mathrm{Sec}_q(\mathscr{L})$ corresponding to a general element $e\in \mathbb{K}^{N+1}$, note that $\mathrm{ord}_{p}(h_*\Delta+wh_*\mathrm{div}(f))\le1$ by the condition that $-(K_{C/S}+\Delta+w\mathscr{L})$ is $h$-nef.
Furthermore, it is easy to observe that if $p\in \mathrm{Sing}\,C'$, then $\mathrm{ord}_{p}(h_*\Delta+wh_*\mathrm{div}(f))=0$ for any $f\in\mathrm{Sec}_q(\mathscr{L})$ corresponding to a general element $e\in \mathbb{K}_{N+1}$.
Indeed, we can decompose $h$ to contractions of rational tails and bridges and we may assume that $h$ is a contraction of one rational tail or one rational bridge $D$.
In this case, since $-(K_{C/S}+\Delta+w\mathscr{L})$ is $h$-nef, $\mathrm{deg}_D(\Delta|_D+w\mathscr{L}|_D)\le1$ if $D$ is a rational tail and $\mathrm{deg}_D(\Delta|_D+w\mathscr{L}|_D)=0$ if $D$ is a rational bridge. 
This means that if you choose general $f$ for any $p\in C'$, $(C',h_*\Delta+wh_*\mathrm{div}(f))$ is slc around $p$.

Next, for general $S$, we show that $\varphi'_i\in \Gamma(C'_U,\mathscr{L}'_U)$, where $U$ is as in the statement.
By the first paragraph of this proof, we may assume that $h_*\mathrm{div}(\varphi)$ is a relative Mumford divisor.
By Proposition \ref{prop:push--linear-equivalence}, we see that $\mathrm{div}(\varphi'_i)=(h_*\mathrm{div}(\varphi_i)-h_*\mathrm{div}(\varphi))+h_*\mathrm{div}(\varphi)\ge0$ and this is linearly equivalent to $h_*\mathrm{div}(\varphi)$.
Note that $\mathrm{div}(\varphi'_i)$ is independent of the choice of $\varphi$.
Note that $\varphi_i$'s define a homomorphism $\mathcal{O}_{C'_U}^{\oplus N+1}\to \mathscr{L}'$.
This homomorphism is independent of the choice of $\varphi$ up to multiple of local sections of $\mathcal{O}_{U}^\times$.
Therefore, this homomorphism defines a structure of a family of quasimaps $q'_U\colon C'_U\to [\mathrm{Cone}(X)_U/\mathbb{G}_{m,U}]$ as the assertion locally over $U$ whose structure is independent of the choice of $\varphi$ because $\varphi'_i/\varphi'_j=(h_U)_\sharp(\varphi_i/\varphi_j)$.
Furthermore, by a similar discussion, it is not hard to check that the above construction of $q'_U$ is independent of the choice of $\mathrm{div}(\varphi_i)$.
Therefore, we can patch $q_U$ together over any open subset $U$ of $S$ and we can construct a family of quasimaps $q'\colon C'\to [\mathrm{Cone}(X)_S/\mathbb{G}_{m,S}]$ as the assertion. 
It is not hard to check that our $q'$ is the desired one.
\end{proof}

\subsection{CM line bundle}

We will collect the fundamental notions of the CM line bundles (cf.~\cite{HH2,CM26}).

\begin{defn}\label{defn:CM-line-for-qmaps}
 Let $q\colon (\mathcal{C},\Delta)\to [\mathrm{Cone}(X)_S/\mathbb{G}_{m,S}]$ be the family of quasimaps over $S$ of degree $d$.   Let $\mathscr{L}$ be the line bundle associated with $q$.
 Let $\mathcal{A}$ be a relatively very ample line bundle on $\mathcal{C}$ over $S$.
 Set $\pi\colon \mathcal{C}\to S$ as the canonical morphism.
Note that $\Delta=\sum a_i D_i$ for some $a_i\in(0,1]$, where $D_i$ is a relative Cartier divisor.
Since $\mathcal{C}$ and $D_i$ are flat over $S$, we can apply \cite[Theorem 4]{knudsen1976projectivity} and obtain the following: there exist line bundles $\mathscr{M}_0$, $\mathscr{M}_1$, and $\mathscr{M}_2$ over $S$ unique up to linear equivalence such that
\begin{align*}
    \mathrm{det}(\pi_*\mathcal{A}^{\otimes m})&\sim\bigotimes^{2}_{j=0}\mathscr{M}^{\binom{m}{j}}_j
\end{align*}
for any sufficiently large $m\in\mathbb{Z}_{>0}$.
Furthermore, the above decomposition is functorial: for any morphism $g\colon T\to S$ from a scheme, 
\begin{align*}
    \mathrm{det}(\pi_{T*}\mathcal{A}_T^{\otimes m})&\sim\bigotimes^{2}_{j=0}g^*\mathscr{M}^{\binom{m}{j}}_j.
\end{align*}

Let $w\in\mathbb{Q}_{>0}.$
If there exists $l\in\mathbb{Z}_{>0}$ such that $l(K_{\mathcal{C}/S}+\Delta+w\mathscr{L})=\mathcal{A}$ (resp.~$-l(K_{\mathcal{C}/S}+\Delta+w\mathscr{L})=\mathcal{A}$), then we set the CM line bundle  
\[
\lambda_{\mathrm{CM},q,\Delta,w}:=\frac{1}{l^2}\mathscr{M}_2. \quad \left(\text{resp.~}\frac{-1}{l^2}\mathscr{M}_2\right).
\]
This definition is independent of the choice of $l$.
By the functoriality of the Knudsen-Mumford expansion, we see that the above definition is also functorial, that is, for any morphism $g\colon T\to S$ from a scheme, we see that $g^*\lambda_{\mathrm{CM},q,\Delta,w}=\lambda_{\mathrm{CM},q_T,\Delta_T,w}$.

\end{defn}

\section{K-moduli wall crossing for log Fano quasimaps}\label{sec:Kssqmapswallcrossing}

We fix $\iota\colon X\hookrightarrow\mathbb{P}^N$ a closed immersion of a projective variety and $\mathrm{Cone}(X)\subset\mathbb{A}^{N+1}$.

\begin{defn}
Let $d\in\mathbb{Z}_{>0}$ and $w\in\R_{>0}$.
Let $\vec{a}\in (0,1]^{r}$ and $\vec{d}\in \mathbb{Z}_{>0}^{r}$.
    Let $\mathcal{M}^{\mathrm{Kssqmaps}}_{d,\vec{a},\vec{d},w,\iota}$ be a category fibered in groupoids such that the objects of  $\mathcal{M}^{\mathrm{Kssqmaps}}_{d,\vec{a},\vec{d},w,\iota}(S)$ for any scheme $S$ are families of weighted K-semistable log Fano quasimaps $q\colon (\mathcal{C},\vec{a}\cdot\vec{\mathcal{D}})\to[\mathrm{Cone}(X)_S/\mathbb{G}_{m,S}]$ over $S$ with weight $w$ of degree $d\in\mathbb{Z}_{>0}$ such that
    \begin{itemize}
        \item $\mathrm{deg}_{\mathcal{C}_{s}}((\mathscr{L}_q)_{s})=d$ and $\mathcal{C}_s$ is geometrically connected for any $s\in S$, where $\mathscr{L}_q$ is the associated line bundle with $q$,
        \item $\vec{\mathcal{D}}$ is a vector of $(D^{(1)},\ldots,D^{(r)})$ and $D^{(j)}$ is a relative Mumford divisor of degree $d^{(j)}$, where $\vec{d}=(d^{(1)},\ldots,d^{(r)})$.
        \end{itemize}
        Isomorphisms of $\mathcal{M}^{\mathrm{Kssqmaps}}_{d,\vec{a},\vec{d},w,\iota}(S)$ are defined as isomorphisms of families of weighted quasimaps.
        When $w=0$, we set $\mathcal{M}^{\mathrm{Kssqmaps}}_{d,\vec{a},\vec{d},w,\iota}$ as the K-moduli stack of log Fano pairs of dimension one with relative Mumford divisors of degree $d^{(1)}\ldots,d^{(r)}$ with coefficients $\vec{a}$, i.e.~forget the structure of quasimaps.
        In this case, we can omit $d$, $w$ and $\iota$ and will write it as $\mathcal{M}^{\mathrm{Kss}}_{\vec{a},\vec{d}}$ for simplicity.
        If $\vec{a}\in [0,1]^r$ and there exists $J\subset \{1,\ldots,r\}$ such that $i\in J$ if and only if $a_i=0$, then let $\vec{b},\vec{d}'\in \mathrm{Map}(\{1,\ldots,r\}\setminus J,(0,1])$ such that $b_i=a_i$ and $d_i=d'_i$ for $i\not\in J$ and regard $\mathcal{M}^{\mathrm{Kssqmaps}}_{d,\vec{a},\vec{d},w,\iota}$ as $\mathcal{M}^{\mathrm{Kssqmaps}}_{d,\vec{b},\vec{d}',w,\iota}$.  
        It is easy to check that $\mathcal{M}^{\mathrm{Kssqmaps}}_{d,\vec{a},\vec{d},w,\iota}$ is a stack.
\end{defn}

Hashizume and the first author dealt with $\mathcal{M}^{\mathrm{Kssqmaps}}_{d,\vec{a},\vec{d},w,\iota}$ and the existence of the projective good moduli space in \cite{HH2} in the case where $r=1$ and $w$ and $\vec{a}$ are rational.

First, in the case where $w=0$, note that $\mathcal{M}^{\mathrm{Kss}}_{\vec{a},\vec{d}}$ is an Artin stack of finite type with a proper good moduli space by \cite{zhou2026shape} for any $\vec{a}$ and $\vec{d}$.
We relate the K-moduli space to the VGIT (variation of GIT) moduli of $n$-points of $\mathbb{P}^1$ for $n\in\mathbb{Z}_{>0}$.
This is well-known to experts but we will explain it for the reader's convenience.
We can naturally consider the GIT stability of $(\mathbb{P}^1)^n$ with respect to the diagonal action of $\SL(2)$ and the polarization $\mathcal{O}(t_1,t_2,\ldots,t_n)$ for any $\vec{t}:=(t_1,t_2,\ldots,t_n)\in \mathbb{Z}^n_{>0}$ (cf.~\cite[Section 6]{thaddeus1996geometric} and \cite{GIT}).
By the general theory of VGIT \cite[Theorems 2.3 and 2.4]{thaddeus1996geometric} (see also \cite{DolgachevHu1998VGIT}), there is the decomposition of $\mathbb{R}_{\ge0}^n$ into finitely many polyhedral cones and these cones determine the GIT (semi)stable locus of $(\mathbb{P}^1)^n$ with respect to the diagonal action of $\SL(2)$ and the polarization $\mathcal{O}(t_1,t_2,\ldots,t_n)$.
It is well-known that the GIT (semi)stability of $(x_1,\ldots,x_n)\in (\mathbb{P}^1)^n$ with respect to $\mathcal{O}(t_1,t_2,\ldots,t_n)$ coincides with the condition that
\[
\mathrm{ord}_p\left(\sum_{j=1}^n t_j[x_j]\right)\le\frac{\sum_{j=1}^n t_j}{2} \qquad(\text{resp.}<)
\]
for any closed point $p\in \mathbb{P}^1$.
By \cite[Theorem 3]{li2015remarks} (cf.~Theorem \ref{thm:delta--log--pairs}), we see that the above condition is nothing but the K-(semi)stability of $(\P^1,\mu\sum_{j=1}^n t_j[x_j])$, where $\mu\in\mathbb{R}_{>0}$ is a number such that $\mu\sum_{j=1}^n t_j<2$.
This means that the K-moduli stack is nothing but the VGIT moduli stack.
To be precise, we obtain the following.

\begin{thm}\label{thm:VGIT-K-stability}
Let $\vec{t}:=(t_1,t_2,\ldots,t_n)\in\mathbb{Z}_{>0}^n$ and let $\mu\in\mathbb{R}_{>0}$ be a number such that $\mu\sum_{j=1}^n t_j<2$.
Then the GIT (semi/poly)stability of $(x_1,\ldots,x_n)\in (\mathbb{P}^1)^n(\mathbb{K})$ with respect to $\mathcal{O}(t_1,t_2,\ldots,t_n)$ coincides with the K-(semi/poly)stability of the log pair $(\P^1,\mu\sum_{j=1}^n t_j[x_j])$.
Moreover, letting $((\mathbb{P}^1)^n)^{\mathrm{ss}}$ be the corresponding GIT semistable locus, we obtain 
\[
\mathcal{M}^{\mathrm{Kss}}_{\mu\vec{t},\vec{\mathbf{1}}}\cong[((\mathbb{P}^1)^n)^{\mathrm{ss}}/\PGL(2)]
\]
as stacks, where $\vec{\mathbf{1}}=(1,1,\ldots,1)$.
\end{thm}

\begin{proof}
By the above discussion, we note that K-(semi)stability of $(\P^1,\mu\sum_{j=1}^n t_j[x_j])$ coincides with the GIT (semi)stability of $(x_1,\ldots,x_n)\in (\mathbb{P}^1)^n$ with respect to $\mathcal{O}(t_1,t_2,\ldots,t_n)$.  
Furthermore, if $(X,\Delta)$ is an object of $ \mathcal{M}^{\mathrm{Kss}}_{\mu\vec{t},\vec{\mathbf{1}}}(\mathbb{K})$, then $(X,\Delta)$ is log Fano and hence $X\cong\mathbb{P}^1$.
Thus, it is not hard to check that there exists a canonical isomorphism
\[
\phi\colon\mathcal{M}^{\mathrm{Kss}}_{\mu\vec{t},\vec{\mathbf{1}}}\to[((\mathbb{P}^1)^n)^{\mathrm{ss}}/\PGL(2)].
\]
Under $\phi$, the closed points of two stacks correspond to each other.
Thus, the GIT polystability of $(x_1,\ldots,x_n)\in (\mathbb{P}^1)^n$ with respect to $\mathcal{O}(t_1,t_2,\ldots,t_n)$ coincides with the K-polystability of the log pair $(\P^1,\mu\sum_{j=1}^n t_j[x_j])$
\end{proof}

By \cite[e.g.~Theorem 1.6]{zhou2026shape}, there is a wall crossing phenomenon for $\mathcal{M}^{\mathrm{Kss}}_{\mu\vec{t},\vec{\mathbf{1}}}$ but we will describe it in a more general form (Theorem \ref{thm:wallcrossingmorphKss}).
Theorem \ref{thm:VGIT-K-stability} can be deduced from \cite[Theorem 1.1]{zhou2024log} when $\mu$ is sufficiently small.

Before considering $\mathcal{M}^{\mathrm{Kssqmaps}}_{d,\vec{a},\vec{d},w,\iota}$, we observe the variation of the stacks in $\vec{a}$ and $w$.

\begin{prop}\label{prop:chmber--Kss}
 Fix 
 $d\in\mathbb{Z}_{>0}$ and $\vec{d}\in \mathbb{Z}_{>0}^{r}$. 
Then there are finitely many locally closed disjoint rational polyhedra $P_k\subset [0,1]^{r}\times\R_{\ge0}$ for $k\in\{1,\ldots,m\}$ that satisfy the following:
\begin{itemize}
    \item[] if there exists a K-semistable log Fano quasimap $q\colon (\mathbb{P}^1,\vec{a}\cdot\vec{D})\to[\mathrm{Cone}(X)/\mathbb{G}_{m}]$ with weight $w$ of degree $d$, where $\vec{D}=(D_i)$ such that $D_i$ is an effective Cartier divisor of degree $d_i$, then $(\vec{a},w)\in P_{k}$ for some $k$ and for any $(\vec{c},w')\in \overline{P_k}$, $q\colon (\mathbb{P}^1,\vec{c}\cdot\vec{D})\to[\mathrm{Cone}(X)/\mathbb{G}_{m}]$ with weight $w'$ is K-semistable if $q\colon (\mathbb{P}^1,\vec{c}\cdot\vec{D})\to[\mathrm{Cone}(X)/\mathbb{G}_{m}]$ is log Fano.
\end{itemize}
Here, we regard $q\colon (\mathbb{P}^1,\vec{c}\cdot\vec{D})\to[\mathrm{Cone}(X)/\mathbb{G}_{m}]$ as a log pair $(\mathbb{P}^1,\vec{c}\cdot\vec{D})$ if $w'=0$.
\end{prop}

\begin{proof}
    If $q\colon (\mathbb{P}^1,\vec{a}\cdot\vec{D})\to[\mathrm{Cone}(X)/\mathbb{G}_{m}]$ is K-semistable with the base locus $\mathrm{Bs}(q)$, then for any $p\in \mathbb{P}^1$, we obtain 
    \begin{align}
    wd+\sum_{i}a_id_i&<2 \label{eq:klt-domain}, \text{ and}\\
    2\left(\sum_ia_i\mathrm{mult}_p(D_i)+w\mathrm{mult}_p(\mathrm{Bs}(q))\right)&\le wd+\sum_{i}a_id_i.\label{eq:K-polyst-domain} 
    \end{align}
    If further $q$ is K-stable, then inequality \eqref{eq:K-polyst-domain} is strict (see \cite[Lemma 3.12]{HH2} and Proposition \ref{prop:chara--K-stability--qmaps}).
    Here, $\mathrm{mult}_p(D_i)$ and $\mathrm{mult}_p(\mathrm{Bs}(q))$ possibly take any number of $\{0,1,\ldots,d_i\}$ and $\{0,1,\ldots,d\}$, respectively.
    There are only finitely many inequalities \eqref{eq:klt-domain} and \eqref{eq:K-polyst-domain}.
    Therefore, there is the desired decomposition of $[0,1]^r\times\R_{\ge0}$ into finitely many rational polyhedra $P_k\subset [0,1]^{r}\times \R_{\ge0}$ and we see that such $P_k$ satisfy the assertion. 
\end{proof}

\begin{lem}\label{lem-wall-same}
In Proposition \ref{prop:chmber--Kss}, fix a face $F\subset P_k$ and let $F^\circ$ be the relative interior of $F$.
    For every two points $(\vec{a}_1,w_1),(\vec{a}_2,w_2)\in F^\circ$, $\mathcal{M}^{\mathrm{Kssqmaps}}_{d,\vec{a}_1,\vec{d},w_1,\iota}\cong \mathcal{M}^{\mathrm{Kssqmaps}}_{d,\vec{a}_2,\vec{d},w_2,\iota}$ as stacks.
\end{lem}
\begin{proof}
    Straightforward by Proposition \ref{prop:chmber--Kss} but we will give a proof here for the reader's convenience.
    By the condition of $(\vec{a}_1,w_1),(\vec{a}_2,w_2)\in F^\circ$, note that $w_1=0$ if and only if $w_2=0$ and ${a}_1^{(i)}=0$ if and only if ${a}_2^{(i)}=0$ for any $1\le i\le r$, where $\vec{a}_j:=({a}_j^{(i)})_{i=1}^r$ for $j=1,2$.
    Take an arbitrary object $q\in \mathcal{M}^{\mathrm{Kssqmaps}}_{d,\vec{a}_1,\vec{d},w_1,\iota}(S)$ for any scheme $S$, then changing the coefficients and the weight appropriately, we may regard $q$ as an object of $\mathcal{M}^{\mathrm{Kssqmaps}}_{d,\vec{a}_2,\vec{d},w_2,\iota}(S)$.
    It is easy to check that this correspondence induces a canonical morphism $\mathcal{M}^{\mathrm{Kssqmaps}}_{d,\vec{a}_1,\vec{d},w_1,\iota}\to \mathcal{M}^{\mathrm{Kssqmaps}}_{d,\vec{a}_2,\vec{d},w_2,\iota}$.
    By the same argument, we can construct a quasi-inverse morphism to the above and complete the proof.
\end{proof}

\begin{rem}
 In Lemma \ref{lem-wall-same}, note that $\alpha(\vec{a},w)\in F^\circ$ for any $(\vec{a},w)\in F^\circ$ and $\alpha\in\mathbb{R}_{>0}$ such that $\alpha(wd+\sum_{i}a_id_i)<2$ by the proof of Proposition \ref{prop:chmber--Kss}.
 Thus, $\mathcal{M}^{\mathrm{Kssqmaps}}_{d,\vec{a},\vec{d},w,\iota}\cong \mathcal{M}^{\mathrm{Kssqmaps}}_{d,\alpha\vec{a},\vec{d},\alpha w,\iota}$ as stacks by Lemma \ref{lem-wall-same} in this case.
\end{rem}

\begin{thm}\label{thm:K-moduli--weighted--log-Fano-qmaps}
    $\mathcal{M}^{\mathrm{Kssqmaps}}_{d,\vec{a},\vec{d},w,\iota}$ is an Artin stack of finite type over $\mathbb{K}$ with a projective good moduli space $M^{\mathrm{Kpsqmaps}}_{d,\vec{a},\vec{d},w,\iota}$.

    Furthermore, if $\vec{a}\in\mathbb{Q}^{r}$ and $w\in\mathbb{Q}$, then $\lambda_{\mathrm{CM},\vec{a},w}$ descends to an ample $\mathbb{Q}$-line bundle $\Lambda_{\mathrm{CM},\vec{a},w}$ on $M^{\mathrm{Kpsqmaps}}_{d,\vec{a},\vec{d},w,\iota}$.
\end{thm}

\begin{proof}
In the case where $w=0$, all assertions follow from \cite[Theorem 1.3]{liu2022finite}, Proposition \ref{prop:chmber--Kss} and Lemma \ref{lem-wall-same}, and we may assume that $w>0$.
   First, we deal with the case where $\vec{a}\in\mathbb{Q}^r$ and $w\in\Q$.
   As the argument of the proof of \cite[Theorem 3.45]{HH2}, we see that $\mathcal{M}^{\mathrm{Kssqmaps}}_{d,\vec{a},\vec{d},w,\iota}$ is an Artin stack of finite type over $\mathbb{K}$ with affine diagonal in this case. 
   Take $c\in\mathbb{Q}_{>0}$ and $d'^{(j)}\in\mathbb{Z}_{>0}$ such that $cd'^{(j)}=a^{(j)}d^{(j)}$ and $d^{(j)}|d'^{(j)}$ for any $j$. Set $d':=\sum_{j=1}^r d'^{(j)}$.
   Then there exists a canonical morphism 
   $$\varpi\colon\mathcal{M}^{\mathrm{Kssqmaps}}_{d,\vec{a},\vec{d},w,\iota}\to \mathcal{M}^{\mathrm{Kssqmaps}}_{d,c,d',w,\iota}$$
   such that $\varpi$ maps an object $q\colon (\mathcal{C},\vec{a}\cdot\vec{\mathcal{D}})\to[\mathrm{Cone}(X)_S/\mathbb{G}_{m,S}]$ over $S$ to $q\colon (\mathcal{C},\sum_{j=1}^r a^{(j)}D_j)\to[\mathrm{Cone}(X)_S/\mathbb{G}_{m,S}]$.
   It is easy to see that $\varpi$ is quasi-finite and proper as the proof of \cite[Proposition 2.33]{BL24}.
   Furthermore, $\varpi$ induces an injective homomorphism of each stabilizer groups of geometric points of $\mathcal{M}^{\mathrm{Kssqmaps}}_{d,\vec{a},\vec{d},w,\iota}$ and $\mathcal{M}^{\mathrm{Kssqmaps}}_{d,c,d',w,\iota}$.
   Thus, $\varpi$ is representable by \cite[Lemma 4.4.3]{abramovich2002compactifying}.
   Therefore, we see that $\varpi$ is a finite morphism.
   By \cite[Theorem 3.45]{HH2}, $\mathcal{M}^{\mathrm{Kssqmaps}}_{d,c,d',w,\iota}$ has a projective good moduli space $M^{\mathrm{Kssqmaps}}_{d,c,d',w,\iota}$.
   By \cite[Propositions 3.21 and 3.44, and Theorem A]{AHLH}, we see that $\mathcal{M}^{\mathrm{Kssqmaps}}_{d,\vec{a},\vec{d},w,\iota}$ also has a good moduli space $M^{\mathrm{Kssqmaps}}_{d,\vec{a},\vec{d},w,\iota}$ finite over $M^{\mathrm{Kssqmaps}}_{d,c,d',w,\iota}$. Thus, $M^{\mathrm{Kssqmaps}}_{d,\vec{a},\vec{d},w,\iota}$ is a projective scheme by \cite[Theorem 7.2.10]{olsson2016algebraic}.

   Furthermore, we see that $\lambda_{\mathrm{CM},\vec{a}}$ descends to a $\mathbb{Q}$-line bundle $\Lambda_{\mathrm{CM},\vec{a}}$ on $M^{\mathrm{Kpsqmaps}}_{d,\vec{a},\vec{d},w,\iota}$ by the same argument as the proof of \cite[Lemma 3.36]{HH2} induced by \cite[Theorem 6.6]{Alper}.
   Let $\Lambda_{\mathrm{CM},c}$ be the CM line bundle on $M^{\mathrm{Kpsqmaps}}_{d,c,d',w,\iota}$, which is known to be ample by \cite[Theorem 3.45]{HH2}.
   Set $\pi\colon M^{\mathrm{Kpsqmaps}}_{d,\vec{a},\vec{d},w,\iota}\to M^{\mathrm{Kpsqmaps}}_{d,c,d',w,\iota}$ as the canonical morphism.
   By the functoriality of the CM line bundle, $\pi^*\Lambda_{\mathrm{CM},c}\sim_{\mathbb{Q}}\Lambda_{\mathrm{CM},\vec{a}}$. Since $\pi$ is finite, we see that $\Lambda_{\mathrm{CM},\vec{a}}$ is ample.

   Next, we deal with the general case.
   By Proposition \ref{prop:chmber--Kss} and Lemma \ref{lem-wall-same}, there are $\vec{a}'\in\mathbb{Q}^r$ and $w'\in\mathbb{Q}$ very close to $\vec{a}$ and $w$, respectively, such that $\mathcal{M}^{\mathrm{Kssqmaps}}_{d,\vec{a},\vec{d},w,\iota}\cong\mathcal{M}^{\mathrm{Kssqmaps}}_{d,\vec{a}',\vec{d},w',\iota}$ as stacks.
   We complete the proof by what we have shown in the above argument.
\end{proof}

\begin{thm}[Wall crossing for K-semistable log Fano quasimaps]\label{thm:wallcrossingmorphKss}
    Fix a face $F\subset P_k$ and let $F^\circ$ be the interior of $F$.
    Let $(w,\vec{a})\in F^\circ$ and $(w',\vec{a}')\in \overline{F}$. Suppose that $dw+\vec{a}\cdot \vec{d}< 2$ and $dw'+\vec{a}'\cdot \vec{d}< 2$.
    Then there exists a canonical morphism $\mathcal{M}^{\mathrm{Kssqmaps}}_{d,\vec{a},\vec{d},w,\iota}\to \mathcal{M}^{\mathrm{Kssqmaps}}_{d,\vec{a}',\vec{d},w',\iota}$.
    Furthermore, if $ww'\ne0$ and $\vec{a},\vec{a}'\in (0,1]^r$, then the canonical morphism is an open immersion. 
\end{thm}

\begin{proof}
    By Lemma \ref{lem-wall-same}, we may assume that $w,w'\in\mathbb{Q}$ and $\vec{a}$ and $\vec{a}'$ are rational.
    In the case where $ww'\ne0$ and $\vec{a},\vec{a}'\in (0,1]^r$, the existence of the canonical morphism $\varphi\colon \mathcal{M}^{\mathrm{Kssqmaps}}_{d,\vec{a},\vec{d},w,\iota}\to \mathcal{M}^{\mathrm{Kssqmaps}}_{d,\vec{a}',\vec{d},w',\iota}$ follows from the proof of Theorem \ref{thm:K-moduli--weighted--log-Fano-qmaps} and Proposition \ref{prop:chmber--Kss}.
    We show that $\varphi$ is an open immersion.
    In fact, \cite[Lemma 4.4.3]{abramovich2002compactifying} shows that $\varphi$ is representable.
    Note that, for any object $q\colon (\mathcal{C},\vec{a}'\cdot\vec{\mathcal{D}})\to[\mathrm{Cone}(X)_S/\mathbb{G}_{m,S}]\in \mathcal{M}^{\mathrm{Kssqmaps}}_{d,\vec{a}',\vec{d},w',\iota}(S)$ for any scheme $S$ of finite type over $\mathbb{K}$, there exists an open subset $U\subset S$ such that if we regard $q$ as the family of log quasimaps $q\colon (\mathcal{C},\vec{a}\cdot\vec{\mathcal{D}})\to[\mathrm{Cone}(X)_S/\mathbb{G}_{m,S}]$ with weight $w$, then for any geometric point $\bar{s}\in S$, $q_{\bar{s}}$ is K-semistable if and only if $\bar{s}\in U$.
Indeed, the openness of $U$ follows from that if $s\in U$, then we can take an open neighborhood $V$ of $s$ and two sections $f_1,f_2\in \mathrm{Sec}_{q_V}(\mathscr{L})$ corresponding to elements of $\mathbb{K}^{N+1}$, where $\mathscr{L}$ is the line bundle associated to $q_V$, such that $\mathrm{ord}_{\bar{p}}(w\mathrm{div}(f_i)_s+\vec{a}\cdot\vec{\mathcal{D}}_s)\le\frac{dw+\vec{a}\cdot\vec{d}}{2}$ for any geometric point $\bar{p}\in \mathcal{C}_s$ and for some $i=1,2$.
If you shrink $V$, then for any $s'\in V$ we see that $\mathrm{ord}_{\bar{p}}(w\mathrm{div}(f_i)_{s'}+\vec{a}\cdot\vec{\mathcal{D}}_{s'})\le\frac{dw+\vec{a}\cdot\vec{d}}{2}$ for any geometric point $\bar{p}\in \mathcal{C}_{s'}$ and for some $i=1,2$ and this shows that $q_{\bar{s'}}$ is K-semistable. Thus, $V\subset U$ and hence $U$ is open.
This means that $\varphi$ is an open morphism.
Note that $\varphi(\mathcal{M}^{\mathrm{Kssqmaps}}_{d,\vec{a},\vec{d},w,\iota})$ is clearly isomorphic to $\mathcal{M}^{\mathrm{Kssqmaps}}_{d,\vec{a},\vec{d},w,\iota}$ by definition and hence $\varphi$ is an open immersion.

    In the general case, we can show the assertion using Proposition \ref{prop:chmber--Kss} because any object $q\colon (\mathcal{C},\vec{a}\cdot\vec{\mathcal{D}})\to[\mathrm{Cone}(X)_S/\mathbb{G}_{m,S}]\in \mathcal{M}^{\mathrm{Kssqmaps}}_{d,\vec{a},\vec{d},w,\iota}(S)$ for any scheme $S$ of finite type over $\mathbb{K}$ can be regarded as an object $q\colon (\mathcal{C},\vec{a}'\cdot\vec{\mathcal{D}})\to[\mathrm{Cone}(X)_S/\mathbb{G}_{m,S}]\in \mathcal{M}^{\mathrm{Kssqmaps}}_{d,\vec{a}',\vec{d},w',\iota}(S)$.
\end{proof}

\begin{rem}
    In the above case, if $w\ne0$ but $w'=0$, then the canonical morphism is obtained by forgetting the quasimap structure.
    This defines a morphism from the K-moduli space of weighted log Fano quasimaps to a finite quotient of the VGIT moduli space of weighted points in $\P^1$ (cf.~Theorem \ref{thm:VGIT-K-stability}).
\end{rem}

Finally, as well-known in the log Fano case, we note that for any object $q\in\mathcal{M}^{\mathrm{Kssqmaps}}_{d,\vec{a},\vec{d},w,\iota}(\mathbb{K})$ defining a $\mathbb{K}$-valued point $x$ of $\mathcal{M}^{\mathrm{Kssqmaps}}_{d,\vec{a},\vec{d},w,\iota}$, we see that $q$ is K-stable if and only if $x$ is closed and has a finite stabilizer. This can be checked as Proposition \ref{prop:bpcy-stack--points--characterizations} below.
Furthermore, we can define the following.

\begin{defn}
  For any object $q\in\mathcal{M}^{\mathrm{Kssqmaps}}_{d,\vec{a},\vec{d},w,\iota}(\mathbb{K})$ defining a $\mathbb{K}$-valued point $x$ of $\mathcal{M}^{\mathrm{Kssqmaps}}_{d,\vec{a},\vec{d},w,\iota}$, we say that $q$ is {\it K-polystable} if and only if $x$ is closed.  
\end{defn}

It follows from \cite[Lemma 3.24]{AHLH} that this definition coincides with \cite[Definition 3.21]{HH2}.

\section{Moduli of boundary polarized Calabi-Yau quasimaps}\label{sec:bpcyqmapswallcrossing}

As in the previous section, we fix $\iota\colon X\hookrightarrow \mathbb{P}^N$.

\begin{defn}
Let $d\in\mathbb{Z}_{>0}$ and $w_1,w_2\in\R_{\ge0}$.
Let $\vec{a_1}\in [0,1]^{r_1}$, $\vec{a_2}\in [0,1]^{r_2}$, $\vec{d_1}\in \mathbb{Z}_{>0}^{r_1}$, and $\vec{d_2}\in \mathbb{Z}_{>0}^{r_2}$ for $r_1,r_2\in\mathbb{Z}_{\ge0}$.
Suppose that $w_2\ne0$ or $\vec{a}_2\ne0$.
First, consider when $w_1+w_2>0$ and $\vec{a_1}\in (0,1]^{r_1}$, $\vec{a_2}\in (0,1]^{r_2}$.
    Let $\mathcal{M}^{\mathrm{bpcyqmaps}}_{d,\vec{a_1},\vec{d_1},\vec{a_2},\vec{d_2},w_1,w_2,\iota}$ be a fibered category in groupoids such that the objects of  $\mathcal{M}^{\mathrm{bpcyqmaps}}_{d,\vec{a_1},\vec{d_1},\vec{a_2},\vec{d_2},w_1,w_2,\iota}(S)$ for any scheme $S$ are families of weighted boundary polarized CY quasimaps $q\colon (\mathcal{C},\vec{a_1}\cdot\vec{\mathcal{D}_1}+\vec{a_2}\cdot\vec{\mathcal{D}_2})\to[\mathrm{Cone}(X)_S/\mathbb{G}_{m,S}]$ over $S$ with weight $(w_1,w_2)$ such that \begin{itemize}
        \item if $\mathscr{L}_q$ is the associated line bundle with $q$, then $\vec{a_2}\cdot\vec{\mathcal{D}_2}+w_2\mathscr{L}_q$ is relatively ample over $S$ and $(\vec{a_2}\cdot\vec{\mathcal{D}_2}+w_2\mathscr{L}_q)_{\bar{s}}$ is the polarization part for any geometric point $\bar{s}\in S$,
        \item $q_{\bar{s}}$ is a boundary polarized CY quasimap of weight $(w_1,w_2)$, $\mathcal{C}_{\bar{s}}$ is connected, and $d=\mathrm{deg}(q_{\bar{s}})$ for any geometric point $\bar{s}\in S$, and
        \item $\vec{\mathcal{D}_1}$ and $\vec{\mathcal{D}_2}$ are vectors of $(D_1^{(1)},\ldots,D_1^{(r_1)})$ and $(D_2^{(1)},\ldots,D_2^{(r_2)})$, respectively and $D_{j}^{(i)}$ is a relative Mumford divisor of degree $d_{j}^{(i)}$, where $\vec{d_1}=(d_1^{(1)},\ldots,d_1^{(r_1)})$ and $\vec{d_2}=(d_2^{(1)},\ldots,d_2^{(r_2)})$.
        \end{itemize}
        In the above notation, the {\it source of $q$} means $(\mathcal{C},\vec{a_1}\cdot\vec{\mathcal{D}_1}+\vec{a_2}\cdot\vec{\mathcal{D}_2})$.
        Furthermore, isomorphisms of $\mathcal{M}^{\mathrm{bpcyqmaps}}_{d,\vec{a_1},\vec{d_1},\vec{a_2},\vec{d_2},w_1,w_2,\iota}(S)$ are defined as isomorphisms of families of weighted quasimaps.

        If $\vec{a}_1\in [0,1]^{r_1}$, $\vec{a}_2\in [0,1]^{r_2}$,  and there exist $J_1\subset \{1,\ldots,r_1\}$ and $J_2\subset \{1,\ldots,r_2\}$ such that $i\in J_k$ if and only if $a_k^{(i)}=0$ for $k=1,2$, where $\vec{a_1}=(a_1^{(1)},\ldots,a_1^{(r_1)})$ and $\vec{a_2}=(a_2^{(1)},\ldots,a_2^{(r_2)})$,
        then let $\vec{b}_k=(b_k^{(i)}),\vec{d}'_k=((d'_k)^{(i)})\in \mathrm{Map}(\{1,\ldots,r_k\}\setminus J_k,(0,1])$ such that $b_k^{(i)}=a_k~{(i)}$ and $d^{(i)}_k=(d'_k)^{(i)}$ for $i\not\in J_k$ and regard $\mathcal{M}^{\mathrm{bpcyqmaps}}_{d,\vec{a_1},\vec{d_1},\vec{a_2},\vec{d_2},w_1,w_2,\iota}$ as $\mathcal{M}^{\mathrm{bpcyqmaps}}_{d,\vec{b_1},\vec{d'_1},\vec{b_2},\vec{d'_2},w_1,w_2,\iota}$.  
        Moreover, if $w_1+w_2=0$, then we set $\mathcal{M}^{\mathrm{bpcyqmaps}}_{d,\vec{a_1},\vec{d_1},\vec{a_2},\vec{d_2},w_1,w_2,\iota}$ as the category fibered in groupoids parameterizing boundary polarized CY pairs, that is, the objects of $\mathcal{M}^{\mathrm{bpcyqmaps}}_{d,\vec{a_1},\vec{d_1},\vec{a_2},\vec{d_2},w_1,w_2,\iota}(S)$ are families of weighted one-dimensional boundary polarized CY pairs with relative Mumford divisors with fixed degrees and fixed coefficients as above and the isomorphisms are defined as Definition \ref{defn:weightedpairs-isom}.
        This means that we will forget the quasimap structure in this case.
        In the case where $w_2=0$ and $\vec{a}_2=\vec0$, we set $\mathcal{M}^{\mathrm{bpcyqmaps}}_{d,\vec{a_1},\vec{d_1},\vec{a_2},\vec{d_2},w_1,w_2,\iota}=\emptyset$.
\end{defn}

To avoid introducing separate notation for the case $w_1=w_2=0$, which would unnecessarily complicate the notation, we treat both cases simultaneously throughout this section.
        It is also easy to check that $\mathcal{M}^{\mathrm{bpcyqmaps}}_{d,\vec{a_1},\vec{d_1},\vec{a_2},\vec{d_2},w_1,w_2,\iota}$ is a stack defined over $\mathbb{K}$.

We discuss the properties of $\mathcal{M}^{\mathrm{bpcyqmaps}}_{d,\vec{a_1},\vec{d_1},\vec{a_2},\vec{d_2},w_1,w_2,\iota}$ in the same strategy as Section \ref{sec:Kssqmapswallcrossing}.

\begin{lem}\label{lem:bpcyw>11}
Suppose that $\vec{a}_1\cdot\vec{d}_1+\vec{a}_2\cdot\vec{d}_2+d(w_1+w_2)=2$.
    If $w_1+w_2>1$, and $w_2\ne0$ or $\vec{a}_2\ne\vec{0}$, then $$\mathcal{M}^{\mathrm{bpcyqmaps}}_{d,\vec{a_1},\vec{d_1},\vec{a_2},\vec{d_2},w_1,w_2,\iota}\cong\mathcal{M}^{\mathrm{Kssqmaps}}_{d,((1-\epsilon)\vec{a_1},(1-\epsilon)\vec{a_2}),(\vec{d_1},\vec{d_2}),(1-\epsilon) (w_1+w_2),\iota}$$ for any sufficiently small $\epsilon>0$.
\end{lem}

\begin{proof}
First, note that the stack on the right hand side is independent of the choice of sufficiently small $\epsilon$ by Lemma \ref{lem-wall-same} and we claim that there exists a canonical morphism $$\varphi\colon\mathcal{M}^{\mathrm{bpcyqmaps}}_{d,\vec{a_1},\vec{d_1},\vec{a_2},\vec{d_2},w_1,w_2,\iota}\to\mathcal{M}^{\mathrm{Kssqmaps}}_{d,((1-\epsilon)\vec{a_1},(1-\epsilon)\vec{a_2}),(\vec{d_1},\vec{d_2}),(1-\epsilon) (w_1+w_2),\iota}.$$
  To show this, it suffices to check that for any object $q\colon (C,\vec{a_1}\cdot\vec{D_1}+\vec{a_2}\cdot\vec{D_2})\to[\mathrm{Cone}(X)/\mathbb{G}_{m}]\in \mathcal{M}^{\mathrm{bpcyqmaps}}_{d,\vec{a_1},\vec{d_1},\vec{a_2},\vec{d_2},w_1,w_2,\iota}(\mathbb{K})$, $q\colon (C,(1-\epsilon)(\vec{a_1}\cdot\vec{D_1}+\vec{a_2}\cdot\vec{D_2}))\to[\mathrm{Cone}(X)/\mathbb{G}_{m}]$ with weight $(1-\epsilon) (w_1+w_2)$ is K-semistable. 
  First, we note that $C\cong \mathbb{P}^1$. If not, since $p_a(C)=0$, there is an irreducible component $T\subset C$ such that $\mathrm{deg}_T(\mathscr{L}|_T)\ge1$, where $\mathscr{L}$ is the associated line bundle.
  Note that $\mathrm{deg}_T(\omega_C|_T)\ge-1$ and hence it obstructs the boundary polarized CY condition $\mathrm{deg}_T(K_C+(w_1+w_2)\mathscr{L}+\vec{a_1}\cdot\vec{D_1}+\vec{a_2}\cdot\vec{D_2}|_T)=0$.
  This is a contradiction.
  Therefore, $C\cong\mathbb{P}^1$.
  Then it is easy to check that $q\colon (C,(1-\epsilon)(\vec{a_1}\cdot\vec{D_1}+\vec{a_2}\cdot\vec{D_2}))\to[\mathrm{Cone}(X)/\mathbb{G}_{m}]$ with weight $(1-\epsilon) (w_1+w_2)$ is K-semistable by the boundary polarized CY condition (cf.~\cite[Propositions 2.19 and 2.25]{ascher2024wall}). 
  Thus, there exists a canonical morphism $\varphi$.
  It is easy to construct a quasi-inverse of $\varphi$ for any sufficiently small $\epsilon>0$.
  Indeed, by the finiteness of the chamber decomposition (Proposition \ref{prop:chmber--Kss}), for any sufficiently small $\epsilon>0$, we can regard an arbitrary object $q\in \mathcal{M}^{\mathrm{Kssqmaps}}_{d,((1-\epsilon)\vec{a_1},(1-\epsilon)\vec{a_2}),(\vec{d_1},\vec{d_2}),(1-\epsilon) (w_1+w_2),\iota}(S)$ as an object of $\mathcal{M}^{\mathrm{bpcyqmaps}}_{d,\vec{a_1},\vec{d_1},\vec{a_2},\vec{d_2},w_1,w_2,\iota}(S)$ changing the coefficients and the weights appropriately, where $S$ is a scheme.
\end{proof}

\begin{prop}\label{prop:chmber--bdcy}
 Fix 
 $d\in\mathbb{Z}_{>0}$, $\vec{d}_1\in \mathbb{Z}_{>0}^{r_1}$, and $\vec{d}_2\in \mathbb{Z}_{>0}^{r_2}$. Set $\vec{d}:=(d_1^{(1)},\ldots,d_1^{(r_1)},d_2^{(1)},\ldots,d_2^{(r_2)} )$ and $r:=r_1+r_2$.
Consider $$H:=\left\{(a_1,\ldots,a_r,w_1,w_2)\in[0,1]^{r}\times\R_{\ge0}^2\middle|w_1d+w_2d+\sum_{i=1}^ra_id_i=2\right\}.$$
Then there exists a subdivision $Q_l$ of $H$ by finitely many locally closed disjoint rational polyhedra, where $l\in \{1,\ldots,n\}$, that satisfies the following.

For any $(a_1,\ldots,a_r,w_1,w_2)\in H$, let $\vec{a}_1:=(a_1,\ldots,a_{r_1})$ and $\vec{a}_2:=(a_{r_1+1},\ldots,a_{r_1+r_2})$.
If there exists a boundary polarized CY quasimap $q\colon (C,\vec{a}_1\cdot\vec{D}_1+\vec{a}_2\cdot \vec{D}_2)\to[\mathrm{Cone}(X)/\mathbb{G}_{m}]$ with weight $(w_1,w_2)$ of degree $d$, where $\vec{D}_1=(D_1^{(i)})_{1\le i\le r_1}$ and $\vec{D}_2=(D_2^{(i)})_{1\le i\le r_2}$ such that $D^{(i)}_j$ is an effective Cartier divisor of degree $d_j^{(i)}$, then $(\vec{a}_1,\vec{a}_2,w_1,w_2)\in Q_{l}$ for some $l$, and for any $(\vec{c}_1,\vec{c}_2,w'_1,w'_2)\in H$ such that $(\vec{c}_1,\vec{c}_2,w'_1,w'_2)\in \overline{Q_l}$, $a_{r_1+i}\ne0$ implies $c^{(i)}_2\ne0$, $w_2\ne0$ implies $w'_2\ne0$, and $w_1\ne0$ implies $w'_1\ne0$, $q\colon (C,\vec{c}_1\cdot\vec{D}_1+\vec{c}_2\cdot\vec{D}_2)\to[\mathrm{Cone}(X)/\mathbb{G}_{m}]$ with weight $(w'_1,w'_2)$ is boundary polarized CY.
Here, $\vec{c}_{j}:=(c^{(i)}_j)$ for $j=1,2$.
\end{prop}

\begin{proof}
Consider the following inequality 
\begin{equation}
    \left(\sum_{1\le i\le r_1}a_i\mathrm{mult}_p(D_1^{(i)})+\sum_{1\le i\le r_2}a_{i+r_1}\mathrm{mult}_p(D_2^{(i)})+(w_1+w_2)\mathrm{mult}_p(\mathrm{Bs}(q))\right)\le 1,\label{eq:bpcydomain} 
\end{equation}
which comes from the slc condition.
We note that $\mathrm{mult}_p(D_1^{(i)})\in\{0,1,\ldots,d_1^{(i)}\}$, $\mathrm{mult}_p(D_2^{(i)})\in\{0,1,\ldots,d_2^{(i)}\}$ and $\mathrm{mult}_p(\mathrm{Bs}(q))\in \{0,1,\ldots,d\}$ for any $p\in C$.
First, we deal with the case where $C$ is irreducible.
In this case, we note that $C=\mathbb{P}^1$.
    By the proof of Proposition \ref{prop:chmber--Kss}, we see that \eqref{eq:bpcydomain} makes the polyhedral decomposition satisfies the assertion of this proposition for this $C$.

    Next, we deal with the case where $C$ is reducible.
    By the proof of Lemma \ref{lem:bpcyw>11}, it only occurs when $w_1+w_2\le1$.
    In this case, $C$ has at most two irreducible components that intersect transversally at one point with each other by Lemma \ref{lem--babwild--chain--two--comps}.
    We consider the following linear forms
    \[
    F_{b_1,b_2,e_1^{(j)},e_2^{(j)}}(\vec{a},w_1,w_2)=w_1b_1+w_2b_2+\sum_{1\le j\le r_1}e_1^{(j)}a_j+\sum_{1\le j\le r_2}e_2^{(j)}a_{j+r_1}
    \]
    for any $b_1,b_2\in\{0,\ldots,d\}$, $0\le e^{(j)}_1 \le d^{(j)}_1$ and $0\le e^{(j)}_2 \le d^{(j)}_2$.
    $w_2b_2+\sum_{j=1}^{r_2}e_2^{(j)}a_{r_1+j}>0$ comes from the ampleness condition of $w_2\mathscr{L}+\vec{a}_2\cdot\vec{D}_2$ and $F_{b_1,b_1,e_1^{(j)},e_2^{(j)}}(\vec{a},w_1,w_2)=1$ comes from the CY condition on each irreducible component of $C$.

    Using the above linear forms, we can take a subdivision $Q_l$ for $H$ by finitely many locally closed disjoint rational polyhedra.
    Because if $C$ is reducible, then $C$ consists of two smooth rational curves intersecting at one point transversally, we see that such $Q_l$ satisfy the assertion.
\end{proof}

\begin{lem}\label{lem-wall-same:bdcy}
In Proposition \ref{prop:chmber--bdcy}, fix a face $F\subset Q_l$ and let $F^\circ$ be the relative interior of $F$.
    For every two points $(\vec{a}_1^{(1)}, \vec{a}_2^{(1)},w^{(1)}_1,w_2^{(1)}),(\vec{a}_1^{(2)}, \vec{a}_2^{(2)},w^{(2)}_1,w^{(2)}_2)\in F^\circ$, $\mathcal{M}^{\mathrm{bpcyqmaps}}_{d,\vec{a_1}^{(1)},\vec{d_1},\vec{a_2}^{(1)},\vec{d_2},w^{(1)}_1,w_2^{(1)},\iota}\cong \mathcal{M}^{\mathrm{bpcyqmaps}}_{d,\vec{a_1}^{(2)},\vec{d_1},\vec{a_2}^{(2)},\vec{d_2},w^{(2)}_1,w_2^{(2)},\iota}$ as stacks
\end{lem}

\begin{proof}
We note that $a_{r_1+1}=\ldots=a_{r_1+r_2}=w_2=0$ defines one chamber in the decomposition of $H$ in Proposition \ref{prop:chmber--bdcy} and on this chamber, we set $\mathcal{M}^{\mathrm{bpcyqmaps}}_{d,\vec{a_1},\vec{d_1},\vec{a_2},\vec{d_2},w_1,w_2,\iota}=\emptyset$.
Therefore, it is straightforward by Proposition \ref{prop:chmber--bdcy} and can be shown as Lemma \ref{lem-wall-same}.
\end{proof}

\begin{prop}\label{prop:bpcy--moduli--qmaps--const}
$\mathcal{M}^{\mathrm{bpcyqmaps}}_{d,\vec{a_1},\vec{d_1},\vec{a_2},\vec{d_2},w,w',\iota}$ is an Artin stack of finite type over $\mathbb{K}$ with affine diagonal.    
\end{prop}

\begin{proof}
First, we may assume that $w+w'\in\mathbb{Q}\cap [0,1]$, $w\in\mathbb{Q}$ and $\vec{a_1}$ and $\vec{a_2}$ are rational vectors by Lemmas \ref{lem-wall-same:bdcy} and \ref{lem:bpcyw>11}. 
We may also assume that $w'\ne0$ or $\vec{a}_2\ne\vec0$ by our assumption.
We mainly deal with the case where $w+w'>0$ because the $w=w'=0$ case follows from a similar argument (cf.~Theorem \ref{thm:bpcY-one-dimensional-moduli}). 
We know the boundedness of the boundary polarized CY pairs $(C,\vec{a_1}\cdot\vec{D_1}+\vec{a_2}\cdot\vec{D_2}+(w+w')\mathrm{div}(f))$, where $q\colon (C,\vec{a_1}\cdot\vec{D_1}+\vec{a_2}\cdot\vec{D_2})\to [\mathrm{Cone}(X)/\mathbb{G}_m]\in\mathcal{M}^{\mathrm{bpcyqmaps}}_{d,\vec{a_1},\vec{d_1},\vec{a_2},\vec{d_2},w,\iota}(\mathbb{K})$ and $f\in\mathrm{Sec}_q(\mathscr{L}_q)$ is general since $\mathcal{M}^{\mathrm{bpCY}}_{r,s,v}$ is of finite type if you choose $v:=w'd+\vec{a}_2\cdot\vec{d}_2$ and $r$ and $s$ as the lowest common multiples of all entries of $\vec{a}_i$ for $i=1,2$, $w$ and $w'$.
There exists $M\in\mathbb{Z}_{>0}$ that only depends on $d,\vec{a_1},\vec{d_1},\vec{a_2},\vec{d_2},w,w'$ such that $M(\vec{a_2}\cdot\vec{D_2}+w'\mathscr{L})$ is a very ample Cartier divisor such that $H^1(C,\mathcal{O}_{C}(kM(\vec{a_2}\cdot\vec{D_2}+w'\mathscr{L})))=0$ for any $k\in\mathbb{Z}_{>0}$ by \cite[Proposition 3.3]{Hassett}.
By the Riemann-Roch theorem, we can take $P(m):=mMv+1$ as the Hilbert polynomial $\mathbb{Z}\ni m\mapsto\chi(C,\mathcal{O}_{C}(mM(\vec{a_2}\cdot\vec{D_2}+w'\mathscr{L})))$.
Indeed, it is easy to check that $p_a(C)=0$ by Lemma \ref{lem--babwild--chain--two--comps}.
Moreover, since we bound the information of $\mathrm{div}(f)$ as above, we can take $l_1\in\mathbb{Z}_{>0}$ so that $\mathscr{L}$ is $l_1$-regular with respect to $\mathcal{O}_{C}(M(\vec{a_2}\cdot\vec{D_2}+w'\mathscr{L}))$.
Here, $l_1$-regular means the Castelnuovo--Mumford regularity (see \cite[Section 5.2]{fantechi2005fundamental}).

Consider the Hilbert scheme $H$ parameterizing closed subschemes of $\mathbb{P}^{P(1)-1}$ with Hilbert polynomial $P$.
By the standard argument, there exists the largest open subset $U_1\subset H$ that parameterizes only connected nodal curves with the arithmetic genus zero and at most two irreducible components.
Let $\pi\colon\mathcal{U}\to U_1$ be the universal family and $\mathcal{O}(k)$ denote the pullback of $\mathcal{O}_{\mathbb{P}^{P(1)-1}_{U_1}}(k)$ for any $k\in\mathbb{Z}$.
Consider $U_2:=\bigtimes_{U_1}^{1\le k\le2,1\le j\le r_j}\mathrm{KDiv}^{d^{(j)}_k}_{\mathcal{U}/U_1}$.
Moreover, consider $\mathrm{Quot}^{\mathcal{O}(-l_1)^{\oplus 1+d+l_1M(\sum \vec{a}_2\cdot \vec{d}_2+w'd)}}_{\mathcal{U}\times_{U_1}U_2/U_2}$ and let $\mathcal{F}$ be the universal quotient sheaf of $\mathcal{O}(-l_1)^{\oplus 1+d+l_1M(\sum \vec{a}_2\cdot \vec{d}_2+w'd)}$.
Take an open subset $U_3\subset \mathrm{Quot}_{\mathcal{U}\times_{U_1}U_2/U_2}$ that is the locus over which $\mathscr{F}_s$ is a line bundle of degree $d$ that is $l_1$-regular with respect to $\mathcal{O}(1)$ if and only if $s\in U_3$. Note that $U_3$ is quasi-projective.
By the same argument in \cite[Section 3.2]{HH2}, we take a locally closed subscheme $U_4\subset U_3$ such that for any morphism $g\colon T\to U_3$ from a scheme, $g$ factors through $U_4$ if and only if $g$ satisfies the following:
\begin{itemize}
    \item $\omega_{\mathcal{U}_T/T}^{\otimes N}(N(\vec{a}_1\cdot \vec{\mathcal{D}}_{1,T}+\vec{a}_2\cdot \vec{\mathcal{D}}_{2,T}))\otimes \mathscr{F}^{\otimes N(w+w')}_T\sim_{T}\mathcal{O}_{\mathcal{U}_T}$, where $\vec{\mathcal{D}}_1=(D_1^{(1)},\ldots, D_1^{(r_1)})$ (resp.~$\vec{\mathcal{D}}_2=(D_2^{(1)},\ldots, D_2^{(r_2)})$) and $D^{(k)}_1$ (resp.~$D^{(k)}_2$) is the pullback of the universal relative Mumford divisor on $\mathrm{KDiv}^{d^{(k)}_1}_{\mathcal{U}/U_1}\times_{U_1}\mathcal{U}$ (resp.~$\mathrm{KDiv}^{d^{(k)}_2}_{\mathcal{U}/U_1}\times_{U_1}\mathcal{U}$) on $\mathcal{U}\times_{U_1}U_3$ and $N\in\mathbb{Z}_{>0}$ that $Nw,Nw'\in\mathbb{Z}$, $N\vec{a}_1\in\mathbb{Z}^{r_1}$ and $N\vec{a}_2\in \mathbb{Z}^{r_2}$, 
    \item if $\mathcal{A}:=\mathcal{O}_{\mathcal{U}\times_{U_1}U_3}(M\vec{a}_2\cdot \vec{\mathcal{D}}_{2,T})\otimes \mathscr{F}_T^{\otimes Mw'}$, then $\mathcal{A}\sim_T\mathcal{O}(1)$ and 
    \item $\pi_{\mathcal{U},T*}\mathcal{A}\otimes\kappa(t)\to H^0(\mathcal{A}_{t})$ and the canonical morphism $\kappa(t)^{\oplus 1+d+l_1M(\sum \vec{a}_2\cdot \vec{d}_2+w'd)}\to H^0(\mathscr{F}_t(l_1))$ induced by the universal quotient $\mathcal{O}(-l_1)^{\oplus 1+d+l_1M(\sum \vec{a}_2\cdot \vec{d}_2+w'd)}\to \mathscr{F}$ are surjective for any $t\in T$.
\end{itemize} 
Next, we consider $\mathcal{H}:=\mathbf{Hom}_{U_{4}}(\mathcal{O}_{\mathcal{U}_{U_{4}}}^{N+1},\mathscr{F}_{U_4})$.
Let $g\in\mathfrak{Hom}(\mathcal{O}_{\mathcal{U}_{\mathcal{H}}}^{N+1},\mathscr{F}_{\mathcal{H}})(\mathcal{H})$ be the universal homomorphism.
We note that $g$ corresponds to a global section $(f_0,\ldots,f_N)\in H^0(\mathcal{H},\pi_{\mathcal{U}_{\mathcal{H}}*}\mathscr{F}_\mathcal{H})^{\oplus N+1}$.
Let $U_5\subset \mathcal{H}$ be an open subscheme such that for any $s\in \mathcal{H}$, $s$ is contained in $U_5$ if and only if $(\mathcal{U}_{\bar{s}},\vec{a}_1\cdot \vec{\mathcal{D}}_{1,\bar{s}}+\vec{a}_2\cdot \vec{\mathcal{D}}_{2,\bar{s}}+(w+w')\mathrm{div}(f_{\bar{s}}))$ is slc for some $(a_i)_{i=0}^N\in\overline{\kappa(s)}^{N+1}$, where $f_{\bar{s}}=\sum_{i=0}^N a_if_{i,\bar{s}}$.
Indeed, it is easy to check for any flat projective family of nodal curves, the semi-log canonicity is an open condition. 
Furthermore, as in the discussion of \cite[Section 3.2]{HH2}, there exists a closed subscheme $U_6\subset U_5$ such that for any morphism $\mu\colon T\to U_5$, $\mu$ factors through $U_6$ if and only if $g_{\mathcal{U}_T}$ defines a family of quasimaps that factors through $[\mathrm{Cone}(X)_T/\mathbb{G}_{m,T}]\subset [\mathbb{A}^{N+1}_{T}/\mathbb{G}_{m,T}]$.
Then it is easy to check that $$\mathcal{M}^{\mathrm{bpcyqmaps}}_{d,\vec{a_1},\vec{d_1},\vec{a_2},\vec{d_2},w,w',\iota}\cong [U_6/\PGL(P(1))\times \PGL(1+d+l_1M(\vec{a}_2\cdot \vec{d}_2+w'd))].$$
We complete the proof.
\end{proof}

\begin{thm}\label{thm:s-complete-bpcyqmaps}
    The moduli stack of boundary polarized CY quasimaps $\mathcal{M}^{\mathrm{bpcyqmaps}}_{d,\vec{a_1},\vec{d_1},\vec{a_2},\vec{d_2},w,w',\iota}$ is  $\mathsf{S}$-complete and $\Theta$-reductive.
    In particular, $\mathcal{M}^{\mathrm{bpcyqmaps}}_{d,\vec{a_1},\vec{d_1},\vec{a_2},\vec{d_2},w,w',\iota}$ admits a separated good moduli space $\pi\colon \mathcal{M}^{\mathrm{bpcyqmaps}}_{d,\vec{a_1},\vec{d_1},\vec{a_2},\vec{d_2},w,w',\iota}\to M^{\mathrm{bpcyqmaps}}_{d,\vec{a_1},\vec{d_1},\vec{a_2},\vec{d_2},w,w',\iota}$ of finite type over $\mathbb{K}$.
\end{thm}

\begin{proof}
First, we may assume that $w+w'\in\mathbb{Q}\cap (0,1]$, $w\in\mathbb{Q}$ and $\vec{a_1}$ and $\vec{a_2}$ are rational vectors by Theorem \ref{thm:K-moduli--weighted--log-Fano-qmaps} and Lemmas \ref{lem-wall-same:bdcy} and \ref{lem:bpcyw>11}. 
We may also exclude the case where $w'=0$ and $\vec{a}_2=\vec0$.
We first deal with the $\mathsf{S}$-completeness.   
This follows from a similar argument to \cite[Theorem 3.28]{HH2} and \cite[Theorem 5.4]{babwild} but we give a proof for the reader's convenience.
Note that the $w+w'=0$ case follows from the proof of Theorem \ref{thm:bpcY-one-dimensional-moduli}.

The uniqueness part of $\mathsf{S}$-completeness follows from the same argument of \cite[Theorem 3.28]{HH2}.
Therefore, it suffices to check the existence part.
Let $R$ be a DVR essentially of finite type over $\mathbb{K}$ with a uniformizing parameter $\pi$.
Let $S$ be $\mathrm{Spec(}R[s,t]/(st-\pi))$ and $S^\circ$ the complement of $\mathbf{0}$ the closed point defined by the ideal generated by $s$ and $t$.
Let $q^\circ\colon (\mathcal{C}^\circ,\Delta^\circ+\mathcal{D}^\circ)\to [\mathrm{Cone}(X)_{S^\circ}/\mathbb{G}_{m,S^\circ}]$ be a $\mathbb{G}_m$-equivariant family of boundary polarized CY quasimaps with weight $w,w'\in\mathbb{Q}_{\ge0}$ and the associated line bundle $\mathscr{L}^\circ$ such that $w'\mathscr{L}^\circ+\mathcal{D}^\circ$ is relatively ample over $S^\circ$, which corresponds to a morphism $\overline{\mathrm{ST}}^\circ(R)\to \mathcal{M}^{\mathrm{bpcyqmaps}}_{d,\vec{a_1},\vec{d_1},\vec{a_2},\vec{d_2},w,w',\iota}$.
Let $\pi^\circ\colon \mathcal{C}^\circ\to S^\circ$ be the canonical morphism and $j\colon S^\circ\hookrightarrow S$ the canonical open immersion.
Since $\pi^\circ_*\mathscr{L}^\circ$ is a locally free sheaf on $S^\circ$ by Lemma \ref{lem--vanish} and hence $j_*\pi^\circ_*\mathscr{L}^\circ$ is a coherent sheaf.
Let $f\in H^0(S,j_*\pi^\circ_*\mathscr{L}^\circ)$ be a general element of $\mathrm{Sec}_{q^\circ}(\mathscr{L}^\circ)$ corresponding to an element of $\mathbb{K}^{N+1}$.
Then $(\mathcal{C}^\circ,(w+w')\mathrm{div}(f)+\Delta^\circ+\mathcal{D}^\circ)$ is a boundary polarized CY pairs over $S^\circ$.
By \cite[Theorem 5.4]{babwild}, there exists a unique family of boundary polarized CY pairs $(\mathcal{C},\Delta+\mathcal{D}+\mathcal{D}_1)$, where $\Delta$, $\mathcal{D}$ and $\mathcal{D}_1$ are the Zariski closure of $\Delta^\circ$, $\mathcal{D}^\circ$ and $(w+w')\mathrm{div}(f)$.
Since $\mathcal{D}_1$ is a relative Mumford divisor in the smooth locus over $S$, $\mathcal{D}_1$ is a relative Cartier divisor (cf.~\cite[Corollary 4.36]{kollar-moduli}). Note that $\mathcal{O}_{\mathcal{C}}(\frac{1}{(w+w')}\mathcal{D}_1)$ is the extension of $\mathscr{L}^\circ$.
Using this, we can endow the structure of family of quasimaps $q$ to $\mathcal{C}$ as \cite[Theorem 3.28]{HH2}.
Indeed, if $g\in \mathrm{Sec}_{q^\circ}(\mathscr{L}^\circ)$ is another section corresponding to an element of $\mathbb{K}^{N+1}$, then we can regard $g\in H^0(S,j_*\pi^\circ_*\mathscr{L}^\circ)$ and $g$ can be uniquely extended to a section of $\mathcal{O}_{\mathcal{C}}(\frac{1}{w+w'}\mathcal{D}_1)$.
By construction, it is easy to check that $q$ is a $\mathbb{G}_m$-equivariant family of boundary polarized CY quasimaps.
Therefore, $\mathcal{M}^{\mathrm{bpcyqmaps}}_{d,\vec{a_1},\vec{d_1},\vec{a_2},\vec{d_2},w,w',\iota}$ is $\mathsf{S}$-complete.

The assertion for $\Theta$-reductivity follows from the same argument replacing $\overline{\mathrm{ST}}(R)$ with $\Theta_R$.
The last assertion follows from \cite[Theorem A]{AHLH}.
\end{proof}

\begin{defn}[Test configuration for boundary polarized CY quasimaps]\label{de--ws--test--config--forbpcy}
Let $q\colon(C,\Delta)\to[\mathrm{Cone}(X)/\mathbb{G}_m]$ be a boundary polarized CY quasimap. 
We say that $Q\colon (\mathcal{C},\mathcal{D})\to[\mathrm{Cone}(X)_{\mathbb{A}^1}/\mathbb{G}_{m,\mathbb{A}^1}]$ is a {\it weakly special (ws) test configuration} for $q$ if $Q$ is a $\mathbb{G}_m$-equivariant family of boundary polarized CY quasimaps such that $Q_1\cong q$, where $\mathbb{G}_m$ acts on $\mathbb{A}^1$ by multiplication.   
We define {\it $S$-equivalence} for boundary polarized CY quasimaps as the equivalence relation  generated by $Q_0$ and $q$ if there exists a weakly special test configuration $Q$ for $q$ as above.
In this setting, we say that
\begin{enumerate}
    \item $Q$ is {\it trivial} if $Q$ is $\mathbb{G}_{m}$-equivariantly isomorphic to $q_{\mathbb{A}^1}$ with the trivial $\mathbb{G}_m$-action; 
    \item $q$ is {\it bpcy stable} if every ws test configuration $Q$ for $q$ is trivial;
    \item $q$ is {\it bpcy polystable} if $Q_0\cong q$ for every ws test configuration $Q$ for $q$.
\end{enumerate}
Here, ``bpcy'' stands for boundary polarized CY.
If there is no fear of confusion, we will write ``bpcy stable'' and ``bpcy polystable'' just as ``stable'' and ``polystable''.
\end{defn}

\begin{lem}\label{lem:trivial}
In the above situation, $Q$ is trivial if and only if $(\mathcal{C},\mathcal{D})$ is $\mathbb{G}_m$-equivariantly isomorphic to $(C\times\mathbb{A}^1,\Delta\times\mathbb{A}^1)$ with the trivial $\mathbb{G}_m$-action on the first component.   
\end{lem}

\begin{proof}
If $Q$ is trivial, then $(\mathcal{C},\mathcal{D})$ is $\mathbb{G}_m$-equivariantly isomorphic to $(C\times\mathbb{A}^1,\Delta\times\mathbb{A}^1)$ with the trivial $\mathbb{G}_m$-action by definition.    
Suppose that $(\mathcal{C},\mathcal{D})$ is $\mathbb{G}_m$-equivariantly isomorphic to $(C\times\mathbb{A}^1,\Delta\times\mathbb{A}^1)$ with the trivial $\mathbb{G}_m$-action by definition.  
Then note that $\mathbb{G}_m$ acts on $H^0(\mathcal{C},\mathscr{L})$, where $\mathscr{L}$ is the line bundle associated to $Q$, and hence on $H^0(\mathcal{C}_0,\mathscr{L}_0)$.
Since the $\mathbb{G}_m$-action on $\mathcal{C}_0$ is trivial and there exists a section $f\in H^0(\mathcal{C}_0,\mathscr{L}_0)$ such that $\mathrm{div}(f)$ does not contain any singularity but is $\mathbb{G}_m$-invariant, we can normalize and trivialize the $\mathbb{G}_m$-action on $H^0(\mathcal{C}_0,\mathscr{L}_0)$.
Therefore, $Q$ is trivial.
\end{proof}

Suppose that $q\colon (C,\Delta+D)\to [\mathrm{Cone}(X)/\mathbb{G}_m]$ is a constant boundary polarized CY quasimap of weight $(w,w')$ with $w+w'\le 1$.
It is easy to check that $q$ is bpcy (semi/poly)stable if and only if $(C,(\Delta+w\mathrm{Bs}(q))+(D+w'\mathrm{Bs}(q)))$ is bpcy (semi/poly)stable as a log pair.
Taking it into consideration, we obtain the following.

\begin{lem}\label{lem:constant-qmaps}
 Let $q\colon (C,\Delta+D)\to [\mathrm{Cone}(X)/\mathbb{G}_m]$ be a boundary polarized CY quasimap of weight $(w,w')$ with $w+w'\le 1$.
 Suppose that $q$ is not klt.
 Then $q$ is S-equivalent to a boundary polarized CY quasimap $q'$ of weight $(w,w')$ such that $q'$ is constant and bpcy polystable.
\end{lem}

\begin{proof}
First, we deal with the case where $C$ is irreducible.
By Lemma \ref{lem--babwild--chain--two--comps}, $C\cong\mathbb{P}^1$.
Indeed, for any general section $f\in \mathrm{Sec}_q(\mathscr{L})$ that corresponds to a general element of $\mathbb{K}^{N+1}$, $(C,(w+w')\mathrm{div}(f)+\Delta+D)$ can be regarded as a boundary polarized CY pair.
Letting $\nu\colon \tilde{C}\to C$ be the normalization, $(\tilde{C},\mathfrak{cond}+\nu^{-1}_*((w+w')\mathrm{div}(f)+\Delta+D))$ is boundary polarized CY again, where $\mathfrak{cond}$ is the conductor divisor.
If $C$ is not normal, then it is easy to see that $\mathfrak{cond}$ has two components and they are lc centers.
However, it never occurs by Lemma \ref{lem--babwild--chain--two--comps}. 
Thus, $C\cong\P^1$.
For any general $f\in\mathrm{Sec}_q(\mathscr{L})$ corresponding to an element of $\mathbb{K}^{N+1}$, then $(\mathbb{P}^1,(w+w')\mathrm{div}(f)+\Delta+D)$ is a boundary polarized CY pair.
Let $p\in C$ be an lc center of $(\mathbb{P}^1,(w+w')\mathrm{div}(f)+\Delta+D)$.
Take a test configuration $\mathcal{C}:=\mathrm{Proj}_{\mathbb{A}^1}(\oplus_{m,\lambda\in\mathbb{Z}_{>0}}\mathscr{F}^\lambda_{\mathrm{ord}_p}H^0(\mathbb{P}^1,\mathcal{O}(m)))$.
Here, we set 
\[
\mathscr{F}^\lambda_{\mathrm{ord}_p}H^0(\mathbb{P}^1,\mathcal{O}(m)):=\{f\in H^0(\mathbb{P}^1,\mathcal{O}(m))|\mathrm{ord}_p(f)\ge\lambda\}.
\]
We say that the above semiample normal test configuration $(\mathcal{C},\mathcal{O}(1))$ for $(\P^1,\mathcal{O}_{\P^1}(1))$ as the {\it test configuration induced by $p$}.
Note that $\mathcal{C}_0\cong\mathbb{P}^1$ and $\mathbb{G}_m$ acts on $\mathcal{C}_0$ in a non-trivial way.
Then we see that for any general $f\in\mathrm{Sec}_q(\mathscr{L})$, if we set $\mathcal{D}$ as the Zariski closure of $((w+w')\mathrm{div}(f)+\Delta+D)\times\mathbb{G}_{m}$, then $(\mathcal{C},\mathcal{D})\to\mathbb{A}^1$ is a $\mathbb{G}_m$-equivariant family of boundary polarized CY pairs.
As \cite[Definition 3.18]{HH2}, we can put $Q\colon \mathcal{C}\to[\mathrm{Cone}(X)_{\mathbb{A}^1}/\mathbb{G}_{m,\mathbb{A}^1}]$ the canonical quasimap structure.
Therefore, $Q|_{\mathcal{C}_0}$ admits a natural structure of a boundary polarized CY quasimap.
Since $Q|_{\mathcal{C}_0}$ admits a non-trivial $\mathbb{G}_{m}$-action, we see that $Q|_{\mathcal{C}_0}$ is constant.
This shows that $q$ is S-equivalent to a constant quasimap from $\mathbb{P}^1$ by the $\mathsf{S}$-completeness.

Now, we may assume that $q$ is constant and $C\cong\P^1$.
For any general $f\in \mathrm{Sec}_q(\mathscr{L})$ corresponding to an element of $\mathbb{K}^{N+1}$, if $(\mathbb{P}^1,(w+w')\mathrm{div}(f)+\Delta+D)$ is polystable as a boundary polarized CY pair, then $q$ is also polystable.
Otherwise, $(\mathbb{P}^1,(w+w')\mathrm{div}(f)+\Delta+D)$ has two lc centers $p_1$ and $p_2$ and $\mathrm{ord}_{p_i}(\Delta+w'\operatorname{Bs}(q))>0$ for each $i=1,2$ by a simple observation.
In this case, we can yield a ws test configuration $Q$ for $q$ as one point blow up of $C\times\mathbb{A}^1$ at $(p_1,0)$ by putting the constant quasimap structure induced by $q_{\mathbb{A}^1}$.
It is easy to see that $Q_0$ is constant.
By \cite[Proposition 8.11]{babwild}, it is not hard to check that $Q_0$ is bpcy polystable.

Next, we deal with the case where $C$ is reducible.
Then $C$ consists of two irreducible components $C_1$ and $C_2$ intersecting at a point $p$ by Lemma \ref{lem--babwild--chain--two--comps}.
Let $q_i:=q|_{C_i}$ for $i=1,2$ and we can regard $q_i$ as a boundary polarized CY quasimap. 
By the irreducible case, we can take the test configurations $Q_1\colon \mathcal{C}_1\to[\mathrm{Cone}(X)_{\mathbb{A}^1}/\mathbb{G}_{m,\mathbb{A}^1}]$ and $Q_2\colon \mathcal{C}_2\to[\mathrm{Cone}(X)_{\mathbb{A}^1}/\mathbb{G}_{m,\mathbb{A}^1}]$ using the inverse images of $p$ for boundary polarized CY quasimaps $q_1$ and $q_2$. 
For any general $f\in\mathrm{Sec}_q(\mathscr{L})$, set $\mathcal{D}_i$ as the Zariski closure of $(p+((w+w')\mathrm{div}(f)+\Delta+D)|_{C_i})\times\mathbb{G}_{m}$ for $i=1,2$.
Since $(\mathcal{C}_1,\mathcal{D}_1)$ and $(\mathcal{C}_2,\mathcal{D}_2)$ are families of boundary polarized CY pairs, we can glue $\mathcal{C}_1$ and $\mathcal{C}_2$ along $\overline{\{p\}\times\mathbb{G}_{m}}$ and construct $(\mathcal{C},\mathcal{D})$ as a test configuration for the boundary polarized CY pair $(C,(w+w')\mathrm{div}(f)+\Delta+D)$ by \cite[Proposition 2.12]{babwild}.
Since $f_1$ and $f_2$ take the same value along $\overline{\{p\}\times\mathbb{G}_{m}}\subset\mathcal{C}$ for any functions $f_1\in\mathrm{Sec}_{Q_1}(\mathscr{L}_{Q_1})$ and $f_2\in\mathrm{Sec}_{Q_2}(\mathscr{L}_{Q_2})$ corresponding to the same coordinate function of $\mathbb{A}^{N+1}$, we can also patch functions $f_1$ and $f_2$ on $\mathcal{C}$ by the universal property of the categorical quotient \cite[Definition 9.4]{kollar2013singularities} and hence $\mathcal{C}$ admits a natural structure $Q$ of family of quasimaps.
It is easy to check that $Q$ is a family of boundary polarized CY quasimaps and note that $Q|_{\mathcal{C}_0}$ admits a non-trivial $\mathbb{G}_m^2$-action.
Therefore, $Q|_{\mathcal{C}_0}$ is constant.
By \cite[Proposition 8.11]{babwild}, we see that $(\mathcal{C}_0,\Delta_0+\mathcal{D}_0+w\mathrm{div}(f))$ is polystable for any $f\in \mathrm{Sec}_{Q|_{\mathcal{C}_0}}(\mathscr{L}')$ corresponding to a general element $e\in\mathbb{K}^{N+1}$, where $\mathscr{L}'$ is the associated line bundle and $\Delta_0$ and $\mathcal{D}_0$ are the degenerations of $\Delta$ and $D$ via the test configuration $\mathcal{C}$.
Then it is not hard to check that $Q|_{\mathcal{C}_0}$ is bpcy polystable.
We complete the proof.
\end{proof}

\begin{lem}\label{lem:stable-is-klt-bpcy}
Let $q\colon(C,\Delta)\to[\mathrm{Cone}(X)/\mathbb{G}_m]$ be a boundary polarized CY quasimap with weight $(w,w')$ such that $w+w'\le 1$. 
Then $q$ is bpcy stable if and only if $q$ is klt.
\end{lem}

\begin{proof}
Let $Q$ be an arbitrary ws test configuration for $q$ as Definition \ref{de--ws--test--config--forbpcy}.    
Since $w+w'\le1$, for any general element $F\in\mathrm{Sec}_Q(\mathscr{L}_Q)$ corresponding to an element of $\mathbb{K}^{N+1}$, we have that $(\mathcal{C},\mathcal{D}+(w+w')\mathrm{div}(F))$ is a ws test configuration for the boundary polarized CY pair $(C,\Delta+(w+w')\mathrm{div}(F)|_C)$, where $\mathscr{L}_Q$ is the associated line bundle with $Q$.
If $q$ is klt, then $(C,\Delta+(w+w')\mathrm{div}(F)|_C)$ and $(C,\Delta+(w+w')\mathrm{div}(F')|_C)$ have no common lc place for any general $F\in\mathrm{Sec}_Q(\mathscr{L}_Q)$ corresponding to an element of $\mathbb{K}^{N+1}$.
Thus, \cite[Lemma 4.7]{babwild} shows that $(\mathcal{C},\mathcal{D}+(w+w')\mathrm{div}(F))$ is trivial for any general $F$.
Therefore, $Q$ is trivial by Lemma \ref{lem:trivial}.
On the other hand, if $q$ is not klt, then for any general element $f\in \mathrm{Sec}_q(\mathscr{L})$, we can take a nontrivial ws test configuration for $q$ as the proof of Lemma \ref{lem:constant-qmaps}.
We complete the proof.
\end{proof}

\begin{prop}\label{prop--polyst-is-const}
    Bpcy polystable but not stable boundary polarized CY quasimaps with weight $(w,w')$ such that $w+w'\le 1$ are constant and the identitiy component of the automorphism group is a subtorus of $\mathbb{G}_m^2$.
\end{prop}

\begin{proof}
    Let $Q$ be a ws test configuration for a constant boundary polarized CY quasimap $q$.
    Then $Q_0$ is also constant.
    Combining this fact with the proof of Lemma \ref{lem:constant-qmaps}, we see that polystable but not stable boundary polarized CY quasimaps are constant.

    Next, we deal with such a bpcy polystable quasimap $q$ and its automorphism group.
    Note that $q$ is constant and hence $\mathrm{Aut}^0(q)$ acts on the boundary polarized CY pair $(C,\Delta+(w+w')\mathrm{div}(f))$ for any general $f\in\mathrm{Sec}_q(\mathscr{L})$ effectively.
    Since $\mathrm{dim}\,C=1$, \cite[Proposition 8.11]{babwild} shows that $\mathrm{Aut}^0(q)$ is a subgroup of $\mathbb{G}_m^2$.
\end{proof}

We will state the following for the reader's convenience, which is concerned with the structure of the moduli stack and is well-known by experts.

\begin{prop}\label{prop:bpcy-stack--points--characterizations}
Let $x$ and $x'$ are $\mathbb{K}$-valued points of $\mathcal{M}^{\mathrm{bpcyqmaps}}_{d,\vec{a_1},\vec{d_1},\vec{a_2},\vec{d_2},w,w',\iota}$ defined by boundary polarized CY quasimaps $q$ and $q'$ such that $w+w'\le1$. 
Then the following hold.
\begin{enumerate}
    \item $q$ and $q'$ are S-equivalent if and only if $\overline{\{x\}}\cap\overline{\{x'\}}\ne\emptyset$,
    \item if $x_1,x_2\in \overline{\{x\}}$ are $\mathbb{K}$-valued points, then $\overline{\{x_1\}}\cap \overline{\{x_2\}}\ne\emptyset$,
    \item $q$ is bpcy polystable if and only if $\{x\}$ is closed in $\mathcal{M}^{\mathrm{bpcyqmaps}}_{d,\vec{a_1},\vec{d_1},\vec{a_2},\vec{d_2},w,w',\iota}$,
    \item $\{x\}$ is closed in $\mathcal{M}^{\mathrm{bpcyqmaps}}_{d,\vec{a_1},\vec{d_1},\vec{a_2},\vec{d_2},w,w',\iota}$ and has a finite stabilizer if and only if $q$ is bpcy stable.
\end{enumerate}
\end{prop}

\begin{proof}
It follows from Lemma \ref{lem:stable-is-klt-bpcy} and Proposition \ref{prop--polyst-is-const} as the proof of \cite[Lemma 6.10 and Proposition 6.13]{babwild}.   
\end{proof}

\begin{prop}\label{prop:properness-bpcy}
$\mathcal{M}^{\mathrm{bpcyqmaps}}_{d,\vec{a_1},\vec{d_1},\vec{a_2},\vec{d_2},w,w',\iota}$ satisfies the existence part of the valuative criterion.    
In particular, $M^{\mathrm{bpcyqmaps}}_{d,\vec{a_1},\vec{d_1},\vec{a_2},\vec{d_2},w,w',\iota}$ is proper.
\end{prop}

\begin{proof}
Since $\mathcal{M}^{\mathrm{bpcyqmaps}}_{d,\vec{a_1},\vec{d_1},\vec{a_2},\vec{d_2},w,w',\iota}$ is of finite type over $\mathbb{K}$, it suffices to check that $\mathcal{M}^{\mathrm{bpcyqmaps}}_{d,\vec{a_1},\vec{d_1},\vec{a_2},\vec{d_2},w,w',\iota}$ satisfies the existence part of the valuative criterion for any discrete valuation ring essentially of finite type over $\mathbb{K}$ (cf.~\cite[Lemma A.11]{AHLH}).
By \cite[Theorem A]{AHLH}, Theorem \ref{thm:K-moduli--weighted--log-Fano-qmaps}, Lemmas \ref{lem-wall-same:bdcy} and \ref{lem:bpcyw>11}, we may assume that $w+w'\in\mathbb{Q}\cap [0,1]$, $w\in\Q$ and $\vec{a}_1$ and $\vec{a}_2$ are rational.
We may also exclude the case where $w'=0$ and $\vec{a}_2=\vec0$.
Note that the $w+w'=0$ case follows from the proof of Theorem \ref{thm:bpcY-one-dimensional-moduli} and we may assume that $w+w'>0$.
Let $R$ be a DVR essentially of finite type over $\mathbb{K}$ with the fractional field $K$.
Let $q_K\colon (\mathcal{C}_K,\vec{\Delta}_{1,K}+\vec{\Delta}_{2,K})\to [\mathrm{Cone}(X)_K/\mathbb{G}_{m,K}]$ be a family of boundary polarized CY quasimaps with weight $(w,w')$.
First, we deal with the case where $\mathcal{C}_K$ is geometrically irreducible.
Consider $q_K\colon (\mathcal{C}_K,(1-\epsilon)(\vec{\Delta}_{1,K}+\vec{\Delta}_{2,K}))\to [\mathrm{Cone}(X)_K/\mathbb{G}_{m,K}]$ with weight $(1-\epsilon)(w+w')$ for any sufficiently small $\epsilon>0$.
Then it is easy to check that $q_K\colon (\mathcal{C}_K,(1-\epsilon)(\vec{\Delta}_{1,K}+\vec{\Delta}_{2,K}))\to [\mathrm{Cone}(X)_K/\mathbb{G}_{m,K}]$ with weight $(1-\epsilon)(w+w')$ is K-semistable by using Proposition \ref{prop:chara--K-stability--qmaps} (3).
By the properness of K-moduli space of log Fano quasimaps (cf.~Theorem \ref{thm:K-moduli--weighted--log-Fano-qmaps} and \cite[Theorem 1.1]{bejleri2025proper}), there exist a finite extension $R\subset R'$ of Dedekind domains and a family of K-semistable quasimaps $q\colon(\mathcal{C},(1-\epsilon)(\vec{\Delta}_{1}+\vec{\Delta}_{2}))\to [\mathrm{Cone}(X)_{R'}/\mathbb{G}_{m,R'}]$ with weight $(1-\epsilon)(w+w')$ for a fixed $\epsilon$ such that $q_{K'}\cong (q_{K})_{K'}$, where $K'$ is the fractional field of $R'$.
Since the fibers of $\mathcal{C}$ over closed points of $\mathrm{Spec}\,R'$ are smooth, we see that $q\colon(\mathcal{C},\vec{\Delta}_{1}+\vec{\Delta}_{2})\to [\mathrm{Cone}(X)_{R'}/\mathbb{G}_{m,R'}]$ with weight $(w,w')$ is a family of boundary polarized CY quasimaps if you choose $\epsilon$ sufficiently small by Lemma \ref{lem-wall-same:bdcy}.
Here, note that we can always find a section $\sigma\colon \mathrm{Spec}\,R'\to \mathcal{C}$ such that $\mathrm{Supp}(\vec{\Delta}_{1}+\vec{\Delta}_{2})$ and the image of $\sigma$ are disjoint.

For the case where $\mathcal{C}_K$ is not normal, replacing $R$ with its finite extension, we may assume that $\mathcal{C}_K$ is reducible and take the normalization of $\mathcal{C}_K$.
By replacing $R$ with its finite cover, we may assume that the normalization has two connected components $\mathcal{C}^{(1)}_{K}$ and $\mathcal{C}^{(2)}_{K}$, where $q|_{\mathcal{C}^{(1)}_{K}}$ and $q|_{\mathcal{C}^{(2)}_{K}}$ are boundary polarized CY.
By the argument above, we see that there exist a finite extension of Dedekind domains $R\subset R'$, families of boundary polarized CY quasimaps $q^{(1)}\colon(\mathcal{C}^{(1)},D^{(1)}+\vec{\Delta}_{1}^{(1)}+\vec{\Delta}_{2}^{(1)})\to [\mathrm{Cone}(X)_{R'}/\mathbb{G}_{m,R'}]$, $q^{(2)}\colon(\mathcal{C}^{(2)},D^{(2)}+\vec{\Delta}_{1}^{(2)}+\vec{\Delta}_{2}^{(2)})\to [\mathrm{Cone}(X)_{R'}/\mathbb{G}_{m,R'}]$ and sections $\sigma^{(1)}\colon \mathrm{Spec}\,R'\to \mathcal{C}^{(1)}$ and $\sigma^{(2)}\colon \mathrm{Spec}\,R'\to \mathcal{C}^{(2)}$ such that $q^{(1)}$ and $q^{(2)}$ coincide with the pullback of $q|_{\mathcal{C}^{(1)}_K}$ and $q|_{\mathcal{C}^{(2)}_K}$, respectively along the generic fiber over $\mathrm{Spec}\,R'$, $D^{(i)}$ is the closure of the conductor divisor,  and the image of $\sigma^{(i)}$ is disjoint from $\mathrm{Supp}(D^{(i)}+\vec{\Delta}_{1}^{(i)}+\vec{\Delta}_{2}^{(i)})$ for $i=1,2$.
Then, we see that for any general $e\in\mathbb{K}^{N+1}$, let $f^{(1)}$ and $f^{(2)}$ be the corresponding sections over $\mathcal{C}^{(1)}$ and $\mathcal{C}^{(2)}$, respectively.
We see that $(\mathcal{C}^{(1)},D^{(1)}+\vec{\Delta}_{1}^{(1)}+\vec{\Delta}_{2}^{(1)}+(w+w')\mathrm{div}(f^{(1)})+\sigma^{(1)}(\mathrm{Spec}\,R'))$ and $(\mathcal{C}^{(2)},D^{(2)}+\vec{\Delta}_{1}^{(2)}+\vec{\Delta}_{2}^{(2)}+(w+w')\mathrm{div}(f^{(2)})+\sigma^{(2)}(\mathrm{Spec}\,R'))$ are stable families over $\mathrm{Spec}\,R'$.
By \cite[Corollary 11.41]{kollar-moduli}, there exists a stable family $(\mathcal{C},\vec{\Delta}_{1}+\vec{\Delta}_{2}+(w+w')D+\Sigma^{(1)}+\Sigma^{(2)})$ over $\mathrm{Spec}\,R'$ whose normalization coincides with the disjoint union of $(\mathcal{C}^{(1)},D^{(1)}+\vec{\Delta}_{1}^{(1)}+\vec{\Delta}_{2}^{(1)}+(w+w')\mathrm{div}(f^{(1)})+\sigma^{(1)}(\mathrm{Spec}\,R'))$ and $(\mathcal{C}^{(2)},D^{(2)}+\vec{\Delta}_{1}^{(2)}+\vec{\Delta}_{2}^{(2)}+(w+w')\mathrm{div}(f^{(2)})+\sigma^{(2)}(\mathrm{Spec}\,R'))$, where $\Sigma^{(i)}$ is the image of ${\sigma}^{(i)}(\mathrm{Spec}\,R')$ for $i=1,2$ and the inverse image of $D$ coincides with $\mathrm{div}(f^{(1)})+\mathrm{div}(f^{(2)})$.
We see that $D$ is a Cartier divisor.
$f^{(1)}$ and $f^{(2)}$ can be glued along $D^{(1)}$ and $D^{(2)}$ to a unique global section $f\in H^0(\mathcal{C},\mathcal{O}_{\mathcal{C}}(D))$ by the universal property of the categorical quotient \cite[Definition 9.4]{kollar2013singularities}.
For any $e'\in\mathbb{A}^{N+1}$, let $g^{(1)}$ and $g^{(2)}$ be the corresponding sections over $\mathcal{C}^{(1)}$ and $\mathcal{C}^{(2)}$.
The restrictions of $g^{(1)}$ and $g^{(2)}$ to $D^{(1)}$ and $D^{(2)}$, respectively, coincide with each other generically, and hence $g^{(1)}$ and $g^{(2)}$ descend to $g\in H^0(\mathcal{C},\mathcal{O}(D))$ by the universal property of the categorical quotient \cite[Definition 9.4]{kollar2013singularities} and the $S_2$-condition of $\mathcal{C}$.
Such data of $g$ define a structure of a family of quasimaps $q\colon (\mathcal{C},\vec{\Delta}_{1}+\vec{\Delta}_{2})\to[\mathrm{Cone}(X)_S/\mathbb{G}_{m,S}]$ and it is easy to see that $q$ is a family of boundary polarized CY quasimaps with weight $(w,w')$.
Then the last assertion follows from \cite[Theorem A]{AHLH}.
We complete the proof.
\end{proof}

\subsection{Positivity of the Hodge line bundle}

\begin{defn}
    Let $q\colon (\mathcal{C},\Delta_1+\Delta_2)\to [\mathrm{Cone}(X)_S/\mathbb{G}_{m,S}]$ be a family of boundary polarized CY quasimaps over $S$ of finite type over $\mathbb{K}$ with weight $(w,w')$.
    Let $\mathscr{L}$ be the associated line bundle of $q$.
    If $\Delta_1$ and $\Delta_2$ are relative Mumford $\Q$-divisors, $w'\in\Q$ and $w+w'\in\mathbb{Q}\cap [0,1]$, then we set the {\it Hodge line bundle} $\lambda_{\mathrm{Hodge},q}$ as the $\mathbb{Q}$-line bundle
    \[
    \frac{1}{\mu}(\pi_{\mathcal{C}})_*(\mathcal{O}_{\mathcal{C}}(\mu(K_{\mathcal{C}/S}+\Delta_1+\Delta_2+(w+w')\mathscr{L})))
    \]
    for any sufficiently divisible $\mu\in\mathbb{Z}_
    {>0}$.
    Note that $\lambda_{\mathrm{Hodge},q}$ is affine linear in the coefficients of $\Delta_1$ and $\Delta_2$, and $w$ and $w'$ in $\mathrm{Pic}_{\mathbb{Q}}(S)$.
    This $\mathbb{Q}$-line bundle is determined up to $\mathbb{Q}$-linear equivalence by \cite[Lemma 2.21]{CM26} and for any $g\colon T\to S$ morphism from a scheme, $g^*\lambda_{\mathrm{Hodge},q}\sim_{\mathbb{Q}}\lambda_{\mathrm{Hodge},q_T}$.
    By the proof of Proposition \ref{prop:bpcy--moduli--qmaps--const}, we see that $\mathcal{M}^{\mathrm{bpcyqmaps}}_{d,\vec{a_1},\vec{d_1},\vec{a_2},\vec{d_2},w,w',\iota}$ admits a unique $\mathbb{Q}$-line bundle $\lambda_{\mathrm{Hodge},\vec{a}_1,\vec{a}_2,w,w'}$ such that for any $f\colon S\to \mathcal{M}^{\mathrm{bpcyqmaps}}_{d,\vec{a_1},\vec{d_1},\vec{a_2},\vec{d_2},w,w',\iota}$, $f^*\lambda_{\mathrm{Hodge},\vec{a}_1,\vec{a}_2,w,w'}$ is $\Q$-linearly equivalent to $\lambda_{\mathrm{Hodge},q}$, where $q$ is the family of quasimaps corresponding to $f$, and $w'\ne0$ or $\vec{a}_2\ne\vec0$.
   If there is no fear of confusion, we will write it $\lambda_{\mathrm{Hodge}}$ for simplicity.
\end{defn}

\begin{lem}\label{lem:CM-descent}
Suppose that $w,w',\vec{a}_1$ and $\vec{a}_2$ are rational such that $w+w'\le1$. 
Suppose further that $w'\ne0$ or $\vec{a}_2\ne\vec0$.
    On $M^{\mathrm{bpcyqmaps}}_{d,\vec{a_1},\vec{d_1},\vec{a_2},\vec{d_2},w,w',\iota}$, there exists a $\Q$-line bundle $\Lambda_{\mathrm{Hodge},\vec{a}_1,\vec{a}_2,w,w'}$ such that $\pi^*\Lambda_{\mathrm{Hodge},\vec{a}_1,\vec{a}_2,w,w'}\sim_{\Q}\lambda_{\mathrm{Hodge},\vec{a}_1,\vec{a}_2,w,w'}$ unique up to $\Q$-linear equivalence. 
\end{lem}

\begin{proof}
    The assertion follows from the proof of Proposition \ref{prop:bpcy--moduli--qmaps--const} and the same argument as \cite[14.9 and 14.10.1]{babwild}.
\end{proof}

We also write $\Lambda_{\mathrm{Hodge},\vec{a}_1,\vec{a}_2,w,w'}$ as $\Lambda_{\mathrm{Hodge}}$ for simplicity if there is no fear of confusion.

\begin{thm}\label{thm:proj--CM--bpCY}
    In the same situation as in Lemma \ref{lem:CM-descent},  $\Lambda_{\mathrm{Hodge}}$ is ample. 
\end{thm}

\begin{proof} 
We mainly deal with the case where $w+w'>0$ but the same argument also works for the case where $w+w'=0$ as in Theorem \ref{thm:bpcY-one-dimensional-moduli}.
   Consider the largest open substack $\mathcal{M}^{\mathrm{nconst}}\subset\mathcal{M}^{\mathrm{bpcyqmaps}}_{d,\vec{a_1},\vec{d_1},\vec{a_2},\vec{d_2},w,w',\iota}$ that consists of non-constant quasimaps (cf.~\cite[Lemma 3.26]{HH2}).
   Let $\mathcal{M}^{\mathrm{const}}$ be the reduced structure of the complement of $\mathcal{M}^{\mathrm{nconst}}$ and $M^{\mathrm{const}}$ its good moduli space (cf.~\cite[Lemma 4.14]{Alper}).
   Let $q|_{\mathcal{M}^{\mathrm{const}}}$ be the restriction of the universal family.
   If $(\mathcal{C},\vec{a_1}\cdot\vec{\mathcal{D}_1}+\vec{a_2}\cdot\vec{\mathcal{D}_2})$ is the source of $q$, then we define $\mathcal{D}_3$ as the horizontal part of the base locus of $q$.
   Using the information of the target and the family of boundary polarized CY pairs $(\mathcal{C},\vec{a_1}\cdot\vec{\mathcal{D}_1}+\vec{a_2}\cdot\vec{\mathcal{D}_2}+(w+w')\mathcal{D}_3)$ over $\mathcal{M}^{\mathrm{const}}$, we obtain a morphism $\mathcal{M}^{\mathrm{const}}\to X\times \mathcal{M}^{\mathrm{bpCY}}$, where $\mathcal{M}^{\mathrm{bpCY}}:=\mathcal{M}^{\mathrm{bpCY}}_{r,s,v}$, $r,s$ are the lowest common multiples of the entries of $\vec{a}_1$ and $\vec{a}_2$, $w$ and $w'$, and $v:=dw'+\vec{a}_2\cdot\vec{d}_2$.
   Then, we obtain a morphism $h\colon {M}^{\mathrm{const}}\to X\times {M}^{\mathrm{bpCY}}$.
   By the same argument of the proof of \cite[Lemma 3.42]{HH2}, we see that $h$ is finite and  $\Lambda_{\mathrm{Hodge}}|_{{M}^{\mathrm{const}}}\equiv (w+w')h^*p^*\mathcal{O}_X(1)+h^*q^*\Lambda_{\mathrm{Hodge}}^{\mathrm{bpCY}}$, where $p\colon X\times {M}^{\mathrm{bpCY}}\to X$ and $q\colon X\times {M}^{\mathrm{bpCY}}\to {M}^{\mathrm{bpCY}}$ are the projections, and $\Lambda_{\mathrm{Hodge}}^{\mathrm{bpCY}}$ is the Hodge line bundle on ${M}^{\mathrm{bpCY}}$.
   Therefore, $\Lambda_{\mathrm{Hodge}}|_{{M}^{\mathrm{const}}}$ is ample since $\Lambda_{\mathrm{Hodge}}^{\mathrm{bpCY}}$ is ample by Theorem \ref{thm:bpcY-one-dimensional-moduli}.

   Next, take a morphism $g\colon S\to M^{\mathrm{bpcyqmaps}}_{d,\vec{a_1},\vec{d_1},\vec{a_2},\vec{d_2},w,w',\iota}$ generically finite to the image from a projective normal variety.
   We claim that $g^*\Lambda_{\mathrm{Hodge}}$ is big and nef.
   If $S$ is zero dimensional, this is trivial and we may assume that $\mathrm{dim}\,S>0$.
   If the image of $g$ is contained in ${M}^{\mathrm{const}}$, the assertion follows from the fact that $\Lambda_{\mathrm{Hodge}}|_{{M}^{\mathrm{const}}}$ is ample.
   Therefore, to show the claim, we may assume that the image of $g$ is not contained in ${M}^{\mathrm{const}}$. 
   By \cite[Theorem 1.1]{bejleri2025proper}, replacing $S$ with its generically finite cover, we may assume that $g$ factors through $\mathcal{M}^{\mathrm{bpcyqmaps}}_{d,\vec{a_1},\vec{d_1},\vec{a_2},\vec{d_2},w,w',\iota}$.
   Let $q\colon (\mathcal{C},\vec{a_1}\cdot\vec{\mathcal{D}_1}+\vec{a_2}\cdot\vec{\mathcal{D}_2})\to [\mathrm{Cone}(X)_S/{\mathbb{G}_{m,S}}]$ be the corresponding family over $S$.
   Let $\{e_i\}_{i=0}^{N}$ be a general basis of $\mathbb{K}^{N+1}$ and $f_{j}$ denote the section of $|\mathscr{L}|_q$ corresponding to $e_j$.
Take also a general element $e_{N+1}\in \mathbb{K}^{N+1}$ and set $f_{N+1}$ the corresponding section.   
   Then, the family $$f\colon \left(\mathcal{C},\vec{a_1}\cdot\vec{\mathcal{D}_1}+\vec{a_2}\cdot\vec{\mathcal{D}_2}+\frac{w+w'}{2(N+2)}\sum^{N+1}_{j=0}\mathrm{div}(f_{j})+\frac{w+w'}{2(N+1)}\sum^{N}_{j=0}(\mathrm{div}(f_{j}+f_{N+1}))\right)\to S$$ has maximal variation and the generic fiber is a boundary polarized CY pair by the same argument as \cite[Lemmas 3.35 and 3.43]{CM26}.
   Here, the generic point of $g$ is not mapped to $M^{\mathrm{const}}$ and hence the generic fiber is klt by Lemma \ref{lem:stable-is-klt-bpcy}.
   We note that $f$ is a klt-trivial fibration and by \cite[Theorem 14.4 and Proposition 14.7]{babwild}, we see that the line bundle
   \begin{align*}
   &f_*\mathcal{O}_{\mathcal{C}}\left(M\left(K_{\mathcal{C}/S}+\vec{a_1}\cdot\vec{\mathcal{D}_1}+\vec{a_2}\cdot\vec{\mathcal{D}_2}+\frac{w+w'}{2(N+2)}\sum^{N+1}_{j=0}\mathrm{div}(f_{j})+\frac{w+w'}{2(N+1)}\sum^{N}_{j=0}(\mathrm{div}(f_{j}+f_{N+1}))\right)\right)\\
   &\sim_{\mathbb{Q}}f_*\mathcal{O}_{\mathcal{C}}(M(K_{\mathcal{C}/S}+\vec{a_1}\cdot\vec{\mathcal{D}_1}+\vec{a_2}\cdot\vec{\mathcal{D}_2}+(w+w')\mathscr{L})))\sim_{\mathbb{Q}}M\lambda_{\mathrm{Hodge},q}\sim_{\Q}Mg^*\Lambda_{\mathrm{Hodge}}
   \end{align*}
   is big for any sufficiently divisible $M\in\mathbb{Z}_{>0}$.
   Applying the above argument to the case where $S$ is a smooth curve, we also see that $\Lambda_{\mathrm{Hodge}}$ is nef over $M^{\mathrm{bpcyqmaps}}_{d,\vec{a_1},\vec{d_1},\vec{a_2},\vec{d_2},w,w',\iota}$ by the fact that $\Lambda_{\mathrm{Hodge}}|_{{M}^{\mathrm{const}}}$ is ample. 
   By what we have shown, for any $g$ as above, $(\Lambda_{\mathrm{Hodge}}|_{g(S)})^{\mathrm{dim}\,S}=(\lambda_{\mathrm{Hodge},q})^{\mathrm{dim}\,S}>0$.

   Finally, we see that $\Lambda_{\mathrm{Hodge}}$ is ample by Proposition \ref{prop:properness-bpcy}, \cite[Theorem 3.11]{kollar-proj}, and Chow's lemma for algebraic spaces  (cf.~\cite[Theorem 7.4.1]{olsson2016algebraic}) since for any closed irreducible subspace $V$ of $M^{\mathrm{bpcyqmaps}}_{d,\vec{a_1},\vec{d_1},\vec{a_2},\vec{d_2},w,\iota}$, $\Lambda_{\mathrm{Hodge}}|_V$ is nef and $(\Lambda_{\mathrm{Hodge}}|_V)^{\mathrm{dim}\,V}>0$ (cf.~\cite[Proof of Corollary 4.6]{CM26}).
\end{proof}

\begin{cor}\label{cor:proj:bpcy}
    $M^{\mathrm{bpcyqmaps}}_{d,\vec{a_1},\vec{d_1},\vec{a_2},\vec{d_2},w,w',\iota}$ is projective.
\end{cor}

\begin{proof}
When $w+w'>1$, we see the assertion by Lemma \ref{lem:bpcyw>11} and Theorem \ref{thm:K-moduli--weighted--log-Fano-qmaps}.
    When $w+w'\in\mathbb{Q}\cap [0,1]$, $w'\in\Q$ and $\vec{a}_1$ and $\vec{a}_2$ are rational, then the assertion follows from Theorem \ref{thm:proj--CM--bpCY}.
    In the general case, it follows from Theorem \ref{thm:K-moduli--weighted--log-Fano-qmaps}, Lemmas \ref{lem-wall-same:bdcy} and \ref{lem:bpcyw>11}.
\end{proof}

We remark that the Hodge line bundle $\Lambda_{\mathrm{Hodge},\vec{a}_1,\vec{a}_2,w,w'}$ is affine linear on $w,w'$ and $\vec{a}_1,\vec{a}_2$.

\subsection{Wall crossing from the moduli of boundary polarized CY quasimaps}

\begin{thm}\label{thm--wall-crossing-bpcy}
Let $F$ be a face of $Q_k$ in Proposition \ref{prop:chmber--bdcy}. Let $(\vec{a}_1,\vec{a}_2,w_1,w_2)\in F^\circ$ and $(\vec{a}'_1,\vec{a}'_2,w'_1,w'_2)\in \overline{F}$ such that $\vec{d}_1\cdot \vec{a}_1+\vec{d}_2\cdot \vec{a}_2+d(w_1+w_2)=2$ and $\vec{d}_1\cdot \vec{a}'_1+\vec{d}_2\cdot \vec{a}'_2+d(w_1'+w_2')=2$ with $\vec{a}_2$ or $w_2$ not being zero and $\vec{a}'_2$ or $w'_2$ not being zero. 
Then there exists a canonical morphism $\mathcal{M}^{\mathrm{bpcyqmaps}}_{d,\vec{a_1},\vec{d_1},\vec{a_2},\vec{d_2},w_1,w_2,\iota}\to \mathcal{M}^{\mathrm{bpcyqmaps}}_{d,\vec{a_1}',\vec{d_1},\vec{a_2}',\vec{d_2},w'_1,w'_2,\iota}$.
Furthermore, if $\vec{a}^{(j)}_i\ne0$ if and only if $\vec{a}'^{(j)}_i\ne0$ for any $i=1,2$ and $1\le j\le r_i$, $w_1\ne0$ if and only if $w'_1\ne0$, and $w_2\ne0$ if and only if $w'_2\ne0$, then the above canonical morphism is an open immersion of Artin stacks.
\end{thm}

\begin{proof}
Suppose first that  $\vec{a}^{(j)}_i\ne0$ if and only if $\vec{a}'^{(j)}_i\ne0$ for any $i=1,2$ and $j$ and $w_2\ne0$ if and only if $w'_2\ne0$.
By Proposition \ref{prop:chmber--bdcy}, if $w'_1+w'_2>1$, then we may assume that $w_1+w_2>1$.
Then Lemma \ref{lem:bpcyw>11} reduces this case to Theorem \ref{thm:wallcrossingmorphKss}.
Indeed, we can check by the proofs of Propositions \ref{prop:chmber--Kss} and \ref{prop:chmber--bdcy} that there exists a face $F'$ of $P_l$ in Proposition \ref{prop:chmber--Kss} such that $((1-\epsilon)\vec{a_1},(1-\epsilon)\vec{a_2},(1-\epsilon) (w_1+w_2))\in F'^\circ$ and $((1-\epsilon)\vec{a'_1},(1-\epsilon)\vec{a'_2},(1-\epsilon) (w'_1+w'_2))\in \overline{F'}$ for any sufficiently small $\epsilon>0$.

Therefore, we may assume that $w'_1+w'_2\le1$.
Then the assertion follows from Proposition \ref{prop:chmber--bdcy}.
Indeed, we may first assume that $\vec{a}_i$ and $\vec{a}'_i$ are rational and $w_2,w'_2\in\mathbb{Q}$ by Lemma \ref{lem-wall-same:bdcy}.
It is easy to check that there exists a canonical morphism $\mathcal{M}^{\mathrm{bpcyqmaps}}_{d,\vec{a_1},\vec{d_1},\vec{a_2},\vec{d_2},w_1,w_2,\iota}\to \mathcal{M}^{\mathrm{bpcyqmaps}}_{d,\vec{a_1}',\vec{d_1},\vec{a_2}',\vec{d_2},w'_1,w'_2,\iota}$ by Proposition \ref{prop:chmber--bdcy}.
Then $\mathcal{M}^{\mathrm{bpcyqmaps}}_{d,\vec{a_1}',\vec{d_1},\vec{a_2}',\vec{d_2},w'_1,w'_2,\iota}\cong [U_6/G]$, where $U_6$ is as in the proof of Proposition \ref{prop:bpcy--moduli--qmaps--const} and $G=\PGL(P(1))\times \PGL(1+d+l_1M(\vec{a}'_2\cdot \vec{d}_2+w'_2d))$. 
If $w_1+w_2\le1$, then the argument we took $U_5$ from $\mathcal{H}$ in the proof of Proposition \ref{prop:bpcy--moduli--qmaps--const}, there exists an open subset $U_7\subset U_6$ such that a geometric point $\bar{s}\in U_7$ if and only if regarding the universal family of boundary polarized CY quasimaps $Q\colon (\mathcal{P},\vec{a}_1'\cdot \vec{\mathcal{D}}_1+\vec{a}_2'\cdot \vec{\mathcal{D}}_2)\to [\mathrm{Cone}(X)_{U_6}/\mathbb{G}_{m,U_6}]$ over $U_6$ with weight $(w'_1,w'_2)$ as a family of quasimaps $Q\colon (\mathcal{P},\vec{a}_1\cdot \vec{\mathcal{D}}_1+\vec{a}_2\cdot \vec{\mathcal{D}}_2)\to [\mathrm{Cone}(X)_{U_6}/\mathbb{G}_{m,U_6}]$ with weight $(w_1,w_2)$, $Q_{\bar{s}}$ is boundary polarized CY.
If $w_1+w_2>1$, we take an open subset $U_7\subset U_6$ such that a geometric point $\bar{s}\in U_7$ if and only if $Q_{\bar{s}}$ has no base locus.
In both cases, by Proposition \ref{prop:chmber--bdcy}, we see that $[U_7/G]\cong \mathcal{M}^{\mathrm{bpcyqmaps}}_{d,\vec{a_1},\vec{d_1},\vec{a_2},\vec{d_2},w_1,w_2,\iota}$.  
Therefore, the canonical morphism is an open immersion.

Next, we deal with the general case.
We may first assume that $\vec{a}_i$ and $\vec{a}'_i$ are rational and $w_1,w_2,w'_1,w'_2\in\mathbb{Q}$ by Lemma \ref{lem-wall-same:bdcy}.
For any object $q\colon (\mathcal{C},\vec{a}_1\cdot \vec{\mathcal{D}}_1+\vec{a}_2\cdot \vec{\mathcal{D}}_2)\to [\mathrm{Cone}(X)_{S}/\mathbb{G}_{m,S}] \in \mathcal{M}^{\mathrm{bpcyqmaps}}_{d,\vec{a_1},\vec{d_1},\vec{a_2},\vec{d_2},w_1,w_2,\iota}(S)$ for any scheme $S$ of finite type over $\mathbb{K}$, we aim to yield a certain object of $\mathcal{M}^{\mathrm{bpcyqmaps}}_{d,\vec{a_1}',\vec{d_1},\vec{a_2}',\vec{d_2},w'_1,w'_2,\iota}(S)$. Let $\mathscr{L}_q$ be the associated line bundle with $q$. Consider $w'_2\mathscr{L}_q+\vec{a_2}'\cdot\vec {\mathcal{D}_2}$.
$w'_2\mathscr{L}_q+\vec{a_2}'\cdot\vec {\mathcal{D}_2}\sim_{\mathbb{Q}}0$ never happens by assumption and $w'_2\mathscr{L}_q+\vec{a_2}'\cdot\vec {\mathcal{D}_2}$ is relatively nef.
By \cite[Proposition 3.3]{Hassett}, there exists $M\in\mathbb{Z}_{>0}$ such that for any $s\in S$, $H^1(\mathcal{C}_{s},\mathcal{O}_{\mathcal{C}_s}(M(w'_2\mathscr{L}_{q,s}+\vec{a_2}'\cdot \mathcal{D}_{2,s})))=0$ and $|M(w'_2\mathscr{L}_{q,s}+\vec{a_2}'\cdot \mathcal{D}_{2,s})|$ defines a contraction to nodal curves with the same arithmetic genus contracting irreducible components on which $M(w'_2\mathscr{L}_{q,s}+\vec{a_2}'\cdot \mathcal{D}_{2,s})$ has degree zero.
Therefore, let $h\colon \mathcal{C}\to \mathcal{C}'$ be the morphism, which is defined by the relative complete linear system of $M(w'_2\mathscr{L}_q+\vec{a_2}'\cdot\vec {\mathcal{D}_2})$ over $S$.
Since $H^1(\mathcal{C}_{s},\mathcal{O}_{\mathcal{C}_s}(M(w'_2\mathscr{L}_{q,s}+\vec{a_2}'\cdot \mathcal{D}_{2,s})))=0$, $h_s$ coincides with the morphism defined by $|M(w'_2\mathscr{L}_{q,s}+\vec{a_2}'\cdot \mathcal{D}_{2,s})|$.
If you choose $M$ sufficiently divisible, we see that then $h_*\mathcal{O}_{\mathcal{C}}\cong \mathcal{O}_{\mathcal{C}'}$ and $\mathcal{C}'$ is independent of the choice of $M$.
Let $\mathcal{A}$ be the relatively ample line bundle on $\mathcal{C}'$ corresponding to $M(w'_2\mathscr{L}_{q,s}+\vec{a_2}'\cdot \mathcal{D}_{2,s})$.
Then, $\pi'_*\mathcal{A}\cong\pi_*\mathcal{O}_{\mathcal{C}}(M(w'_2\mathscr{L}_{q,s}+\vec{a_2}'\cdot \mathcal{D}_{2,s}))$ is locally free, where $\pi\colon \mathcal{C}\to S$ and $\pi'\colon \mathcal{C}'\to S$ are the canonical morphisms, and hence $\pi'$ is flat (cf.~\cite[the proof of Theorem 3.6]{Hassett}).
By Lemma \ref{lem--babwild--chain--two--comps}, we see that this $h$ satisfies the condition of Proposition \ref{prop:pushfoward}.
Therefore, we can define $h_*(\vec{a}_1'\cdot\vec{\mathcal{D}_1})$ and $h_*(\vec{a}_2'\cdot\vec{\mathcal{D}_2})$.
Furthermore, we can put the structure of a locally stable family of weighted quasimaps $q'\colon (\mathcal{C}',h_*(\vec{a}_1'\cdot\vec{\mathcal{D}_1})+h_*(\vec{a}_2'\cdot\vec{\mathcal{D}_2}))\to [\mathbb{A}^{N+1}_S/\mathbb{G}_{m,S}]$ with weight $w'_1+w'_2$
defined by Proposition \ref{prop:contraction-thm}.
Since this construction is functorial by \cite[Proposition 3.3]{Hassett}, it suffices to check that $q'$ factors through $[\mathrm{Cone}(X)_S/\mathbb{G}_{m,S}]$ and if you put the weight $(w'_1,w'_2)$, then $q'$ becomes a family of boundary polarized CY quasimaps.
For any closed point $s\in S$, take a general basis $e_0,\ldots,e_N\in\mathbb{K}^{N+1}$ and let $f_0,\ldots,f_N\in \mathrm{Sec}_{q}(\mathscr{L}_q)$ the sections corresponding to $e_i$ such that $(f_i)_s$ does not take zero at any singular point.  
Take an open neighborhood $U\subset S$ of $s$ such that $\mathrm{div}(f_i)|_{\mathcal{C}_U}$ is a relative Mumford divisor for any $i=0,1,\ldots,N$.
Let $\mathscr{L}'$ be the line bundle associated with $q'$.
Then, $\mathscr{L}'_U\sim_U\mathcal{O}_{\mathcal{C}'_U}(\mathrm{div}(f_0)|_{\mathcal{C}'_U})$ by Proposition \ref{prop:contraction-thm} and hence $K_{\mathcal{C}'_U}+h_*(\vec{a}_1'\cdot\vec{\mathcal{D}_1})_U+h_*(\vec{a}_2'\cdot\vec{\mathcal{D}_2})+(w'_1+w'_2)\mathscr{L}'_U\sim_{\Q,U}0$.
It is easy to check that $h_*(\vec{a}_2'\cdot\vec{\mathcal{D}_2})+w'_2\mathscr{L}'_U$ is relatively ample over $U$.
Furthermore, if we set $\sigma\in \Gamma(\mathscr{L}'_U)$ as a section defining $h_*(\mathrm{div}(f_0)_U)$, then $f'_i:=(h_U)_\sharp(\frac{f_i}{f_0})\cdot\sigma\in \Gamma(\mathscr{L}'_U)$ for $i=0,\ldots,N$ define $q'_U$.
For any homogeneous polynomial $\varphi$ of the ideal defining $\mathrm{Cone}(X)\subset\mathbb{A}^{N+1}$, it is easy to see that $\varphi(f'_0,\ldots,f'_N)=0$.
This shows that $q'_U$ factors through $[\mathrm{Cone}(X)_U/\mathbb{G}_{m,U}]$.
Therefore, we can conclude that $q'$ is the desired family in $\mathcal{M}^{\mathrm{bpcyqmaps}}_{d,\vec{a_1}',\vec{d_1},\vec{a_2}',\vec{d_2},w'_1,w'_2,\iota}(S)$.

Note that $\mathcal{M}^{\mathrm{bpcyqmaps}}_{d,\vec{a_1},\vec{d_1},\vec{a_2},\vec{d_2},w_1,w_2,\iota}\cong[U_6/G]$, where $U_6$ and $G$ are as in Proposition \ref{prop:bpcy--moduli--qmaps--const} and $G$ acts on $U_6$ with the universal family of quasimaps over $U_6$.
By what we have shown in the previous paragraph, we see that there exists a canonical morphism $U_6\to \mathcal{M}^{\mathrm{bpcyqmaps}}_{d,\vec{a_1}',\vec{d_1},\vec{a_2}',\vec{d_2},w'_1,w'_2,\iota}$, which is $G$-invariant.
Therefore, we see that there exists a natural morphism $\mathcal{M}^{\mathrm{bpcyqmaps}}_{d,\vec{a_1},\vec{d_1},\vec{a_2},\vec{d_2},w_1,w_2,\iota}\to \mathcal{M}^{\mathrm{bpcyqmaps}}_{d,\vec{a_1}',\vec{d_1},\vec{a_2}',\vec{d_2},w'_1,w'_2,\iota}$. 
\end{proof}

Next, we consider the moduli wall crossing from the moduli of boundary polarized CY quasimaps to the K-moduli of log Fano quasimaps.
Fix $\mathcal{M}^{\mathrm{bpcyqmaps}}_{d,\vec{a_1},\vec{d_1},\vec{a_2},\vec{d_2},w,w',\iota}$ for $w+w'\le 1$.
Assume that $w'>0$ or $\vec{a}_2$ is nonzero.
Consider for any sufficiently small $\epsilon>0$.
Then, it is easy to see that there is a natural morphism $\mathcal{M}^{\mathrm{Kssqmaps}}_{d,(\vec{a_1},(1-\epsilon)\vec{a_2}),(\vec{d_1},\vec{d_2}),w+(1-\epsilon)w',\iota}\to \mathcal{M}^{\mathrm{bpcyqmaps}}_{d,\vec{a_1},\vec{d_1},\vec{a_2},\vec{d_2},w,w',\iota}$ by the proof of Propositions \ref{prop:chmber--Kss} and \ref{prop:chmber--bdcy}.
\begin{prop}\label{prop:standard-wall-crossing}
 The canonical morphism $\phi'\colon \mathcal{M}^{\mathrm{Kssqmaps}}_{d,(\vec{a_1},(1-\epsilon)\vec{a_2}),(\vec{d_1},\vec{d_2}),w+(1-\epsilon)w',\iota}\to \mathcal{M}^{\mathrm{bpcyqmaps}}_{d,\vec{a_1},\vec{d_1},\vec{a_2},\vec{d_2},w,w',\iota}$ is an open immersion.  
\end{prop}

\begin{proof}
Using the notations of the proof of Proposition \ref{prop:bpcy--moduli--qmaps--const}, $\mathcal{M}^{\mathrm{bpcyqmaps}}_{d,\vec{a_1},\vec{d_1},\vec{a_2},\vec{d_2},w,w',\iota}\cong[U_6/G]$, where $U_6$ is as in the proof of Proposition \ref{prop:bpcy--moduli--qmaps--const} and $G=\PGL(P(1))\times \PGL(1+d+l_1M(\vec{a}'_2\cdot \vec{d}_2+w'd))$.
Let $q\colon (\mathcal{C},\vec{\Delta}_1+\vec{\Delta}_2)\to [\mathrm{Cone}(X)_{U_6}/\mathbb{G}_{m,U_6}]\in (\mathcal{M}^{\mathrm{bpcyqmaps}}_{d,\vec{a_1},\vec{d_1},\vec{a_2},\vec{d_2},w,w',\iota})(U_6)$ be the universal family of quasimaps.
Regarding $q$ as a family of log quasimaps with coefficients $(\vec{a_1},(1-\epsilon)\vec{a_2})$ and the weight $w+(1-\epsilon)w'$, then there exists an open subset $U'_6$ such that for any geometric point $\bar{u}\in U_6$,  $q_{\bar{u}}$ is normal and $q_{\bar{u}}$ is K-semistable.
Indeed, let $\mathscr{L}$ be the associated line bundle. 
Suppose that a point $u\in U_6$ satisfies that $q_{\bar{u}}$ is K-semistable.
Replacing $U_6$ around $\bar{u}$ by its suitable \'etale cover, we may assume that $\mathcal{C}\cong \mathbb{P}^1_{U_6}$ and $\mathscr{L}\sim_{U_6}\mathcal{O}_{\mathbb{P}^1_{U_6}}(d)$.
Then, there exist a sufficiently large and divisible $l\in\mathbb{Z}_{>0}$ and general elements $e_1,\ldots,e_l\in \mathbb{K}^{N+1}$ such that the corresponding sections $f_i\in  \mathrm{Sec}_q(\mathscr{L})$ satisfy that $(\mathcal{C}_{\bar{u}},(\vec{\Delta}_{1}+(1-\epsilon)\vec{\Delta}_2+\frac{w+(1-\epsilon)w'}{l}\sum_{i=1}^l\mathrm{div}(f_i))_{\bar{u}})$ is a K-semistable log Fano pair by Lemma \ref{lem:HH2-modified}.
By \cite[Corollary 1.2]{blum2022openness}, we obtain that there exists an open neighborhood $V$ of $u$ such that $(\mathcal{C}_{\bar{v}},(\vec{\Delta}_{1}+(1-\epsilon)\vec{\Delta}_2+\frac{w+(1-\epsilon)w'}{l}\sum_{i=1}^l\mathrm{div}(f_i))_{\bar{v}})$ is K-semistable for any geometric point $\bar{v}\in V$.
This means that $V\subset U'_6$ and hence $U'_6$ is open by Lemma \ref{lem:HH2-modified}.
Thus, $\mathcal{M}^{\mathrm{Kssqmaps}}_{d,(\vec{a_1},(1-\epsilon)\vec{a_2}),(\vec{d_1},\vec{d_2}),w+(1-\epsilon)w',\iota}\cong [U'_6/G]$ and $\phi$ is an open immersion.
\end{proof}
As above, we suppose that $w+w'\le 1$
and that $w'>0$ or $\vec{a}_2$ is nonzero.
Furthermore, there is a natural morphism $\mathcal{M}^{\mathrm{Kssqmaps}}_{d,(1-\epsilon)(\vec{a_1},\vec{a_2}),(\vec{d_1},\vec{d_2}),(1-\epsilon)(w+w'),\iota}\to \mathcal{M}^{\mathrm{bpcyqmaps}}_{d,\vec{a_1},\vec{d_1},\vec{a_2},\vec{d_2},w,w',\iota}$ by the proof of Propositions \ref{prop:chmber--Kss} and \ref{prop:chmber--bdcy} by the same discussion.
We will write it simply $\mathcal{M}^{\mathrm{Kssqmaps}}_{\epsilon,\iota}\to \mathcal{M}^{\mathrm{bpcyqmaps}}_{\iota}$.

\begin{thm}\label{thm:wall-crossingbpcytoK-moduli}
The morphism $\phi\colon\mathcal{M}^{\mathrm{Kssqmaps}}_{\epsilon,\iota}\to \mathcal{M}^{\mathrm{bpcyqmaps}}_{\iota}$ is an open immersion and induces a homeomorphism $\varphi\colon M^{\mathrm{Kpsqmaps}}_{\epsilon,\iota}\to M^{\mathrm{bpcyqmaps}}_{\iota}$ of the good moduli spaces for any sufficiently small $\epsilon>0$. 
Furthermore, if $\iota=\mathrm{id}_{\mathbb{P}^N}$, then $\varphi$ is an isomorphism.
\end{thm}

\begin{proof}
First, note that the same argument as the proof of Proposition \ref{prop:standard-wall-crossing} shows that $\phi$ is an open immersion.
Note that $\varphi$ is induced by \cite[Theorem 6.6]{Alper}.

In the case where $\iota=\mathrm{id}_{\mathbb{P}^N}$, we note that the moduli stacks are smooth by Proposition \ref{prop:smoothness-of-stacks}.
Therefore, the good moduli spaces are normal.
Therefore, since the good moduli spaces are proper, it suffices to show that the morphism $\varphi$ is bijective by the Zariski main theorem.

Let $\mathcal{U}\subset \mathcal{M}^{\mathrm{bpcyqmaps}}_{\iota}$ be the open locus parameterizing all klt quasimaps and $\mathcal{M}^{\mathrm{ss}}$ the complement of $\mathcal{U}$ with reduced structure.
Let $M^{\mathrm{ss}}$ be the good moduli space and is a closed subscheme of $M^{\mathrm{bpcyqmaps}}_{\iota}$ by \cite[Lemma 4.14]{Alper}.
For any closed point $p\in M^{\mathrm{ss}}$, we claim that $\varphi^{-1}(p)$ consists of only one closed point. 
First, we show that $\varphi^{-1}(p)\ne\emptyset$.
Let $q_1\colon (C,\vec{\Delta}_1+\vec{\Delta}_2)\to [\mathrm{Cone}(X)/\mathbb{G}_{m}]$ be the bpcy polystable quasimap representing $p$ (cf.~Proposition \ref{prop:bpcy-stack--points--characterizations}).
If $q_1$ is normal, then it is easy to check that $\varphi^{-1}(p)\ne\emptyset$ by the discussion in the first paragraph of this proof.
Otherwise, note that $q_1$ is constant by Lemma \ref{lem:constant-qmaps}.
Let $C_1\subset C$ be an irreducible component. 
We can construct $h\colon C\to C_1$ the contraction morphism of $C\setminus C_1$ and we can put $q'_1\colon (C_1,h_*(\vec{\Delta}_1+\vec{\Delta}_2))$ the structure of slc quasimaps by Proposition \ref{prop:contraction-thm}. 
Let $x\in C_1$ be the unique singular point of $C$ on $C_1$.
As the proof of Lemma \ref{lem:constant-qmaps}, the test configuration induced by $x$ shows that $q'_1$ is S-equivalent to $q_1$ as boundary polarized CY quasimaps. 
Changing the weight and coefficients of $q'_1$ appropriately, $q_1'$ becomes a K-semistable weighted Fano quasimap in $\mathcal{M}^{\mathrm{Kssqmaps}}_{\epsilon,\iota}$ and we can check that $\varphi^{-1}(p)\ne\emptyset$.

Next, we show that $\varphi^{-1}(p)$ contains at most one point.
Let $q\colon (C,(1-\epsilon)(\vec{\Delta}_1+\vec{\Delta}_2))\to [\mathrm{Cone}(X)/\mathbb{G}_{m}]$ be the K-polystable log Fano quasimap corresponding to a closed point $p_1\in\varphi^{-1}(p)$.
Since $q\colon (C,\vec{\Delta}_1+\vec{\Delta}_2)\to [\mathrm{Cone}(X)/\mathbb{G}_{m}]$ is strictly lc, we see that $q\colon (C,(1-\epsilon)(\vec{\Delta}_1+\vec{\Delta}_2))\to [\mathrm{Cone}(X)/\mathbb{G}_{m}]$ with weight $(1-\epsilon)(w+w')$ is not strictly K-stable and admits a nontrivial $\mathbb{G}_m$-action (cf.~Proposition \ref{prop:chara--K-stability--qmaps} (3)).
Therefore, $\mathrm{Supp}(\vec{\Delta}_1+\vec{\Delta}_2)+(w+w')\mathrm{Bs}(q)$ consists of two closed points. 
If $\mathrm{Supp}(\vec{\Delta}_1)$ consists of one closed point, then $q\colon (C,\vec{\Delta}_1+\vec{\Delta}_2)\to [\mathrm{Cone}(X)/\mathbb{G}_{m}]$ is bpcy polystable and hence $\varphi^{-1}(p)$ consists of a unique closed point.
Otherwise, $q$ degenerates to a polystable boundary polarized CY quasimap $q'\colon (C',\vec{\Delta}'_1+\vec{\Delta}'_2)\to [\mathrm{Cone}(X)/\mathbb{G}_{m}]$ such that $\mathrm{Supp}(\vec{\Delta}'_1+\vec{\Delta}'_2)+(w+w')\mathrm{Bs}(q')$ consists of two closed points.
Note that $\vec{\Delta}'_1+\vec{\Delta}'_2$ is a degeneration of $\vec{\Delta}_1+\vec{\Delta}_2$ via a ws test configuration of $q$.
Therefore, $q'$ completely determines the isomorphism class of $q$ and hence a unique point of $\varphi^{-1}(p)$.

Let $\pi\colon \mathcal{M}^{\mathrm{bpcyqmaps}}_{\iota}\to {M}^{\mathrm{bpcyqmaps}}_{\iota}$ be the morphism of the good moduli space.
Note that $\pi^{-1}({M}^{\mathrm{bpcyqmaps}}_{\iota}\setminus M^{\mathrm{ss}})= \mathcal{U}$ by Lemmas \ref{lem:constant-qmaps} and \ref{lem:stable-is-klt-bpcy}.
If $\mathcal{U}=\emptyset$, then the assertion follows from the above discussion.
We may assume that  $\mathcal{U}$ is not empty.
We see by a straightforward argument that $\varphi|_{\varphi^{-1}(M^{\mathrm{bpcyqmaps}}_{\iota}\setminus M^{\mathrm{ss}})}$ is bijective to the image because the coarse moduli space of $\mathcal{U}$ coincides with $M^{\mathrm{bpcyqmaps}}_{\iota}\setminus M^{\mathrm{ss}}$ and $\mathcal{U}$ is contained in the image of $\phi$ by the discussion of the first paragraph of this proof. 
Indeed, we can check that $\phi^{-1}(\mathcal{U})$ coincides with the open locus, which parameterizes all K-stable log Fano quasimaps by Proposition \ref{prop:chara--K-stability--qmaps} (3).
Therefore, $\varphi$ is bijective.
We complete the proof.
\end{proof}

By Theorem \ref{thm:wall-crossingbpcytoK-moduli}, the homeomorphism class of ${M}^{\mathrm{bpcyqmaps}}_{d,\vec{a_1},\vec{d_1},\vec{a_2},\vec{d_2},w,w',\iota}$ is independent of the polarization in the following sense.
For the definition of the seminormalization, see \cite[Definition 4.3]{HH2}. For any Artin stack $\mathfrak{X}$, let $\mathfrak{X}^{\mathrm{sn}}$ denote the seminormalization of $\mathfrak{X}$.
Note that if $\mathfrak{X}$ admits a separated good moduli space $X$, then $\mathfrak{X}^{\mathrm{sn}}$ admits a good moduli space $X^{\mathrm{sn}}$, where $X^{\mathrm{sn}}$ is the seminormalization of $X$ by \cite[Lemma 13.7]{babwild}.  

\begin{cor}
Let $d,r\in\mathbb{Z}_{>0}$, $w\in[0,1]$, $\vec{d}\in \mathbb{Z}_{>0}^r$, and $\vec{a}=(a_1,\ldots,a_r)\in\mathbb{R}^r_{>0}$.    
Take $0\le k<k'\le r$ and $0\le\beta<\alpha\le1$.
Suppose that $\alpha=1$ or $w=0$ implies that $k'\ne r$.
Set $\vec{a}_1:=(a_1,\ldots,a_k)$, $\vec{a}_2:=(a_{k+1},\ldots,a_r)$, $\vec{a}'_1:=(a_1,\ldots,a_{k'})$, $\vec{a}'_2:=(a_{k'+1},\ldots,a_r)$, $\vec{d}_1:=(d_1,\ldots,d_k)$, $\vec{d}_2:=(d_{k+1},\ldots,d_r)$, $\vec{d}'_1:=(d_1,\ldots,d_{k'})$, and $\vec{d}'_2:=(d_{k'+1},\ldots,d_r)$.
Suppose further that $\vec{d}\cdot\vec{a}+dw=2$.

Then there exists a canonical isomorphism $({M}^{\mathrm{bpcyqmaps}}_{d,\vec{a_1},\vec{d_1},\vec{a_2},\vec{d_2},\alpha w,(1-\alpha)w,\iota})^{\mathrm{sn}}\cong ({M}^{\mathrm{bpcyqmaps}}_{d,\vec{a'_1},\vec{d'_1},\vec{a'_2},\vec{d'_2},\beta w,(1-\beta)w,\iota})^{\mathrm{sn}}$, where $({M}^{\mathrm{bpcyqmaps}}_{d,\vec{a_1},\vec{d_1},\vec{a_2},\vec{d_2},\alpha w,(1-\alpha)w,\iota})^{\mathrm{sn}}$ means the seminormalization of ${M}^{\mathrm{bpcyqmaps}}_{d,\vec{a_1},\vec{d_1},\vec{a_2},\vec{d_2},\alpha w,(1-\alpha)w,\iota}$.
If further $\iota=\mathrm{id}_{\mathbb{P}^N}$, then ${M}^{\mathrm{bpcyqmaps}}_{d,\vec{a_1},\vec{d_1},\vec{a_2},\vec{d_2},\alpha w,(1-\alpha)w,\iota}\cong {M}^{\mathrm{bpcyqmaps}}_{d,\vec{a'_1},\vec{d'_1},\vec{a'_2},\vec{d'_2},\beta w,(1-\beta)w,\iota}$.
\end{cor}

\begin{proof}
By Theorem \ref{thm:wall-crossingbpcytoK-moduli} and the definition of the seminormalization, $$({M}^{\mathrm{bpcyqmaps}}_{d,\vec{a_1},\vec{d_1},\vec{a_2},\vec{d_2},\alpha w,(1-\alpha)w,\iota})^{\mathrm{sn}}\cong ({M}^{\mathrm{Kpsqmaps}}_{d,(1-\epsilon)\vec{a},\vec{d},(1-\epsilon)w,\iota})^{\mathrm{sn}}\cong ({M}^{\mathrm{bpcyqmaps}}_{d,\vec{a'_1},\vec{d'_1},\vec{a'_2},\vec{d'_2},\beta w,(1-\beta)w,\iota})^{\mathrm{sn}}.$$  
If further $\iota=\mathrm{id}_{\mathbb{P}^N}$, since ${M}^{\mathrm{bpcyqmaps}}_{d,\vec{a_1},\vec{d_1},\vec{a_2},\vec{d_2},\alpha w,(1-\alpha)w,\iota}$ and $ {M}^{\mathrm{bpcyqmaps}}_{d,\vec{a'_1},\vec{d'_1},\vec{a'_2},\vec{d'_2},\beta w,(1-\beta)w,\iota}$ are normal and hence seminormal, then we see that ${M}^{\mathrm{bpcyqmaps}}_{d,\vec{a_1},\vec{d_1},\vec{a_2},\vec{d_2},\alpha w,(1-\alpha)w,\iota}\cong {M}^{\mathrm{bpcyqmaps}}_{d,\vec{a'_1},\vec{d'_1},\vec{a'_2},\vec{d'_2},\beta w,(1-\beta)w,\iota}$.
\end{proof}

\begin{rem}
    We cannot conclude that $(\mathcal{M}^{\mathrm{bpcyqmaps}}_{d,\vec{a_1},\vec{d_1},\vec{a_2},\vec{d_2},\alpha w,(1-\alpha)w,\iota})^{\mathrm{sn}}\cong (\mathcal{M}^{\mathrm{bpcyqmaps}}_{d,\vec{a'_1},\vec{d'_1},\vec{a'_2},\vec{d'_2},\beta w,(1-\beta)w,\iota})^{\mathrm{sn}}$ in general because in some case there possibly exist closed points $p_1\in (\mathcal{M}^{\mathrm{bpcyqmaps}}_{d,\vec{a_1},\vec{d_1},\vec{a_2},\vec{d_2},\alpha w,(1-\alpha)w,\iota})^{\mathrm{sn}}$ and $p_2\in (\mathcal{M}^{\mathrm{bpcyqmaps}}_{d,\vec{a'_1},\vec{d'_1},\vec{a'_2},\vec{d'_2},\beta w,(1-\beta)w,\iota})^{\mathrm{sn}}$ whose stabilizers are of different dimensions.
    Indeed, suppose that $w=0$, $r=2$, $k=0$, $k'=1$, $\vec{a}=(1,1)$, and $\vec{d}=(1,1)$.
    Then $\mathcal{M}^{\mathrm{bpcyqmaps}}_{d,\vec{a'_1},\vec{d'_1},\vec{a'_2},\vec{d'_2},\beta w,(1-\beta)w,\iota}$ consists of only one bpcy polystable log pair $(\mathbb{P}^1,\{0\}+\{\infty\})$ such that its automorphism group is of dimension one but $\mathcal{M}^{\mathrm{bpcyqmaps}}_{d,\vec{a_1},\vec{d_1},\vec{a_2},\vec{d_2},\alpha w,(1-\alpha)w,\iota}$ contains a log pair that is reducible and bpcy polystable whose automorphism group contains $\mathbb{G}_{m}^2$.
    Indeed, the log pair consists of two copies of $(\mathbb{P}^1,\{0\}+\{\infty\})$ glued at $\{0\}$.
\end{rem}

We finally deal with the interpolation of the Hodge line bundle and the CM line bundle.

\begin{prop}
    Let $q\colon (\mathcal{P},\Delta)\to [\mathrm{Cone}(X)_S/\mathbb{G}_{m,S}]$ be a family of connected K-semistable log Fano quasimaps over a normal variety $S$ with weight $w\ge0$ of degree $d$.
    Suppose that that $\alpha:=\mathrm{deg}_{\mathcal{P}_{s}}(\Delta_s)$ is independent of any choice of $s\in S$ and $\alpha\ne0$ or $w\ne0$.
    Note that $q\colon (\mathcal{P},\frac{2}{\alpha+wd}\Delta)\to [\mathrm{Cone}(X)_S/\mathbb{G}_{m,S}]$ is a family of boundary polarized CY quasimaps with weight $\frac{2}{\alpha+wd}w$.

    If $w\in\mathbb{Q}\cap[0,1]$ and $\Delta$ is a $\Q$-divisor, then
    \[
    \lambda_{\mathrm{CM},q}\sim_{\mathbb{Q}}(2-\alpha-wd)(\alpha+wd)\lambda_{\mathrm{Hodge},q}
    \]
\end{prop}

\begin{proof}
We note that it is easy to check that $q\colon (\mathcal{P},\frac{2}{\alpha+wd}\Delta)\to [\mathrm{Cone}(X)_S/\mathbb{G}_{m,S}]$ is a family of boundary polarized CY quasimaps with weight $\frac{2}{\alpha+wd}w$ by Proposition \ref{prop:chara--K-stability--qmaps} (3).
    We show this proposition as \cite[Proposition 2.25]{ascher2024wall}.
    Both of the $\mathbb{Q}$-line bundles are functorial and we may replace $S$ with its resolution of singularities.
    Therefore, we may assume that $S$ is quasi-projective and smooth.
    Let $\pi\colon \mathcal{P}\to S$ be the canonical morphism.
    We can apply \cite[Lemma A.2]{codogni2021positivity} to $\pi$ (we note that this lemma is applicable to the case for the projective morphism of quasi-projective normal varieties, cf.~the proof of \cite[Proposition 2.25]{ascher2024wall}).
    In this case, we see that 
\begin{align*}
&\lambda_{\mathrm{CM},q}\sim_{\mathbb{Q}}-\pi_*(K_{\mathcal{P}/S}+\Delta+w\mathscr{L})^2\\
&\sim_{\mathbb{Q}}-\frac{\alpha+wd}{2}\pi_*\Bigg(\left(K_{\mathcal{P}/S}+\frac{2}{\alpha+wd}(\Delta+w\mathscr{L})\right)\\
&\cdot \left(\frac{4-\alpha-wd}{2}K_{\mathcal{P}/S}+\Delta+w\mathscr{L}\right)\Bigg)\\
&+\frac{(4-\alpha-wd)(\alpha+wd)-4}{4}\pi_*((-K_{\mathcal{P}/S})^2)
\end{align*}
Since the K-moduli space of $\mathbb{P}^1$ consists of one point, $\pi_*((-K_{\mathcal{P}/S})^2)\sim_{\mathbb{Q}}0$ (cf.~\cite[Definition 3.3]{codogni2021positivity}). 
On the other hand, $$\pi_*\left(\left(K_{\mathcal{P}/S}+\frac{2}{\alpha+wd}(\Delta+w\mathscr{L})\right)\cdot \left(\frac{4-\alpha-wd}{2}K_{\mathcal{P}/S}+\Delta+w\mathscr{L}\right)\right)\sim_{\Q}2(\alpha+wd-2)\lambda_{\mathrm{Hodge},q}.$$
Therefore, we have the assertion.
\end{proof}

By this proposition, we see that the CM line bundle and the Hodge line bundle are $\mathbb{Q}$-linearly equivalent via an isomorphism of the normalizations $(M^{\mathrm{Kpsqmaps}}_{\epsilon,\iota})^\nu\to (M^{\mathrm{bpcyqmaps}}_{\iota})^\nu$, which is obtained in Theorem \ref{thm:wall-crossingbpcytoK-moduli} up to multiple of positive constant.

\section{Moduli of weighted stable quasimaps and wall crossing}\label{sec:KSBAqmapswallcrossing}

We fix $\iota\colon X\hookrightarrow\mathbb{P}^N$ as Section \ref{sec:Kssqmapswallcrossing} throughout this section.

\begin{defn}
Let $d\in\mathbb{Z}_{>0}$ and $w>0$.
Let $\vec{a}\in (0,1]^{r}$ and $\vec{d}\in \mathbb{Z}_{>0}^{r}$.
    Let $\mathcal{M}^{\mathrm{stqmaps}}_{p_a,d,\vec{a},\vec{d},w,\iota}$ be a category fibered in groupoids such that the objects of  $\mathcal{M}^{\mathrm{stqmaps}}_{p_a,d,\vec{a},\vec{d},w,\iota}(S)$ for any scheme $S$ are families of weighted stable quasimaps $q\colon (\mathcal{C},\vec{a}\cdot\vec{\mathcal{D}})\to[\mathrm{Cone}(X)_S/\mathbb{G}_{m,S}]$ over $S$ with weight $w$ of degree $d\in\mathbb{Z}_{>0}$ such that \begin{itemize}
    \item $\mathcal C_{\bar{s}}$ is connected for every
geometric point $\bar{s}\to S$.
        \item $\mathrm{deg}_{\mathcal{C}_{s}}((\mathscr{L}_q)_{s})=d$ and $p_a(\mathcal{C}_{s})=p_a$ for any $s\in S$, where $\mathscr{L}_q$ is the associated line bundle with $q$,
        \item $\vec{\mathcal{D}}$ is a vector of $(D^{(1)},\ldots,D^{(r)})$ and $D^{(j)}$ is a relative Mumford divisor of degree $d^{(j)}$, where $\vec{d}=(d^{(1)},\ldots,d^{(r)})$.
        \end{itemize}
        Isomorphisms of $\mathcal{M}^{\mathrm{stqmaps}}_{p_a,d,\vec{a},\vec{d},w,\iota}(S)$ are defined as isomorphisms of families of weighted quasimaps.
        When $w=0$, we set $\mathcal{M}^{\mathrm{stqmaps}}_{p_a,d,\vec{a},\vec{d},w,\iota}$ as the corresponding moduli stack of weighted stable pairs with relative Mumford divisors of degree $d^{(1)}\ldots,d^{(r)}$ with coefficients $\vec{a}$.
        If $\vec{a}\in [0,1]^r$ and there exists $J\subset \{1,\ldots,r\}$ such that $i\in J$ if and only if $a_i=0$, then let $\vec{b},\vec{d}'\in \mathrm{Map}(\{1,\ldots,r\}\setminus J,(0,1])$ such that $b_i=a_i$ and $d_i=d'_i$ for $i\not\in J$ and regard $\mathcal{M}^{\mathrm{stqmaps}}_{p_a,d,\vec{a},\vec{d},w,\iota}$ as $\mathcal{M}^{\mathrm{stqmaps}}_{p_a,d,\vec{b},\vec{d}',w,\iota}$.  

        It is easy to check that $\mathcal{M}^{\mathrm{stqmaps}}_{p_a,d,\vec{a},\vec{d},w,\iota}$ is a stack. We regard this stack as empty if any boundary coefficient
exceeds one.
\end{defn}

\begin{rem}
Marian-Oprea-Pandharipande \cite{marian2011moduli} considered the moduli space of quasimaps not only to $[\mathbb{A}^{N+1}/\mathbb{G}_m]$ but also for quasimaps to the Grassmannian stack for any sufficiently small weight. 
    Toda \cite{Toda} generalized their framework to an arbitrary weight in the case where $\vec{d}=(1,1,\ldots,1)$. Moreover, he also considered the moduli wall crossing for the case where the target is $[\mathbb{A}^{N+1}/\mathbb{G}_m]$.

    For the weighted stable maps, Bayer-Manin \cite{bayer2009stability} and Alexeev-Guy \cite{Al-Guy} discussed the moduli wall crossing, which will be shown by Proposition \ref{prop:when:quasimap-is-map} later that this is contained in our situation as the $w>1$ case.
\end{rem}

\begin{prop}\label{prop:chmber--st}
 Fix 
 $d\in\mathbb{Z}_{>0}$ and $\vec{d}\in \mathbb{Z}_{>0}^{r}$. 
Then there are finitely many locally closed disjoint rational polyhedra $R_k\subset [0,1]^{r}\times\R_{\ge0}$ for $k\in\{1,\ldots,m\}$ that satisfy the following: if there exists a stable quasimap $q\colon (C,\vec{a}\cdot\vec{D})\to[\mathrm{Cone}(X)/\mathbb{G}_{m}]$ with weight $w$ of degree $d$, where $\vec{D}=(D_i)$ such that $D_i$ is an effective Cartier divisor of degree $d_i$, then $(\vec{a},w)\in R_{k}^\circ$ for some $k$ and for any $(\vec{c},w')\in R_k^\circ$ (resp.~$(\vec{c},w')\in \overline{R_k}$), $q\colon (C,\vec{c}\cdot\vec{D})\to[\mathrm{Cone}(X)/\mathbb{G}_{m}]$ with weight $w'$ is stable (resp.~$\omega_C+w'\mathscr L+\vec c\cdot\vec D$ is nef and $q\colon (C,\vec{c}\cdot\vec{D})\to[\mathrm{Cone}(X)/\mathbb{G}_{m}]$ with weight $w'$ is slc).
\end{prop}

\begin{proof}
    Let $q\colon (C,\vec{a}\cdot\vec{D})\to[\mathrm{Cone}(X)/\mathbb{G}_{m}]$ be a stable quasimap with weight $w$ of degree $d$.
Let $T$ be an irreducible component of $C$ such that $\omega_C$ is not ample along $T$.    
We note that $\mathrm{deg}_T\omega_C|_T=-2,=-1$ or $=0$.
    To achieve the stable condition for $q$, $\mathrm{deg}_T(w\mathscr{L}+\vec{a}\cdot \vec{D})|_T>2$, $>1$ or $>0$, respectively. 
    Furthermore, for the slc condition for $q$, we need $\mathrm{mult}_p(w\mathrm{Bs}(q)+\vec{a}\cdot \vec{D})\le 1$ for any smooth point $p\in C$. Therefore, 
    \begin{align}
    we+\sum_{i}a_ie_i&>s \label{eq:st-domain}, \text{ and}\\
    \sum_ia_ie_i+we&\le1, \label{eq:st-domainlc} 
    \end{align}
    where $s\in\{0,1,2\}$, $e\in \{0,1,\ldots, d\}$ and $e_i\in\{0,\ldots,d_i\}$.
    Therefore, by using \eqref{eq:st-domain} and \eqref{eq:st-domainlc}, there is the decomposition of $[0,1]^r\times\R_{>0}$ into finitely many rational polyhedra $R_k\subset [0,1]^{r}\times\R_{>0}$ and we see that such $R_k$ satisfy the assertion as the argument of \cite[Section 2]{Al-Guy}.
\end{proof}

\begin{prop}\label{prop:when:quasimap-is-map}
If $w>1$, then a family of weighted stable quasimaps $q\colon (\mathcal{C},\vec{a}\cdot\vec{\mathcal{D}})\to[\mathrm{Cone}(X)_S/\mathbb{G}_{m,S}]$ over $S$ with weight $w$ of degree $d\in\mathbb{Z}_{>0}$ is a family of weighted stable maps. 
In particular, if $2p_a-2+wd+\vec{a}\cdot\vec{d}>0$, then for any $w'>w$, $\mathcal{M}^{\mathrm{stqmaps}}_{p_a,d,\vec{a},\vec{d},w,\iota}\cong \mathcal{M}^{\mathrm{stqmaps}}_{p_a,d,\vec{a},\vec{d},w',\iota}$.  
\end{prop}

\begin{proof}
Note that for any $w>1$, by the slc condition, $q_{\bar{s}}$ has no fixed part for any geometric point $\bar{s}\in S$.    
By Nakayama's lemma, we see that the associated morphism $\mathcal{O}_{\mathcal{C}}^{\oplus N+1}\to\mathscr{L}_q$ is surjective.
Therefore, $q$ factors through $X_S\subset [\mathrm{Cone}(X)_S/\mathbb{G}_{m,S}]$ and $q$ is a weighted stable map by the slc condition.
Moreover, we see that  there exists a canonical morphism $\mathcal{M}^{\mathrm{stqmaps}}_{p_a,d,\vec{a},\vec{d},w,\iota}\to \mathcal{M}^{\mathrm{stqmaps}}_{p_a,d,\vec{a},\vec{d},w',\iota}$ for any $w'>w$.  

Conversely, for any $q\in \mathcal{M}^{\mathrm{stqmaps}}_{p_a,d,\vec{a},\vec{d},w',\iota}(S)$, if you replace the weight $w'$ with $w$, then $q$ remains to be stable.
Indeed, for any geometric point $\bar{s}\in S$, let $T\subset \mathcal{C}_{\bar{s}}$ be an irreducible component.
If $T\ne \mathcal{C}_{\bar{s}}$ or $p_a(\mathcal{C}_{\bar{s}})>0$ for any $\bar{s}\in S$, then $\mathrm{deg}_T(\omega_{\mathcal{C}_{\bar{s}}}|_T)\ge-1$.
If $\mathrm{deg}_{T}(\mathscr{L})\ge1$, then $\mathrm{deg}_T(\omega_{\mathcal{C}_{\bar{s}}}+w''\mathscr{L}|_T)>0$ for any $w''>1$.
Otherwise, $\mathrm{deg}_T(\omega_{\mathcal{C}_{\bar{s}}}+\vec{a}\cdot\vec{\mathcal{D}}_{\bar{s}}+w'\mathscr{L}|_T)=\mathrm{deg}_T(\omega_{\mathcal{C}_{\bar{s}}}+\vec{a}\cdot\vec{\mathcal{D}}_{\bar{s}}|_T)>0$.
Next, we deal with the case where $T=\mathcal{C}_{\bar{s}}$ and $p_a(\mathcal{C}_{\bar{s}})=0$.
By assumption, we see that $\vec{a}\cdot\vec{d}+wd>2$.
This shows that $\mathrm{deg}_{\mathcal{C}_{\bar{s}}}(\omega_{\mathcal{C}_{\bar{s}}}+\vec{a}\cdot\vec{\mathcal{D}}_{\bar{s}}+w'\mathscr{L})>0$.
Therefore, if  $2p_a-2+wd+\vec{a}\cdot\vec{d}>0$ and you replace the weight $w'$ with $w$, then $q$ remains to be stable since the slc condition is unchanged.
As a result, if $2p_a-2+wd+\vec{a}\cdot\vec{d}>0$, then we see that $\mathcal{M}^{\mathrm{stqmaps}}_{p_a,d,\vec{a},\vec{d},w,\iota}\cong \mathcal{M}^{\mathrm{stqmaps}}_{p_a,d,\vec{a},\vec{d},w',\iota}$.
\end{proof}

    \begin{lem}\label{lem-wall-same:st}
Fix a face $F\subset R_k$ and let $F^\circ$ be the interior of $F$ in Proposition \ref{prop:chmber--st}.
    For every two points $(\vec{a}_1,w_1),(\vec{a}_2,w_2)\in F^\circ$, $\mathcal{M}^{\mathrm{stqmaps}}_{p_a,d,\vec{a}_1,\vec{d},w_1,\iota}\cong \mathcal{M}^{\mathrm{stqmaps}}_{p_a,d,\vec{a_2},\vec{d},w_2,\iota}$ as stacks.
\end{lem}

\begin{proof}
    It is straightforward by Proposition \ref{prop:chmber--st} and can be shown as Lemma \ref{lem-wall-same}.
\end{proof}
\begin{thm}
\label{thm:constksbaqmaps}
$\mathcal{M}^{\mathrm{stqmaps}}_{p_a,d,\vec{a},\vec{d},w,\iota}$ is a proper Deligne-Mumford stack that has a proper coarse moduli space $M^{\mathrm{stqmaps}}_{p_a,d,\vec{a},\vec{d},w,\iota}$.
\end{thm}

\begin{proof}
  By Lemma \ref{lem-wall-same:st}, we may assume that $\vec{a}$ is rational and $w\in\mathbb{Q}$.
  Note that if $2p_a-2+wd+\vec{a}\cdot\vec{d}\le0$, then $\mathcal{M}^{\mathrm{stqmaps}}_{p_a,d,\vec{a},\vec{d},w,\iota}=\emptyset$. Thus, we may assume that $2p_a-2+wd+\vec{a}\cdot\vec{d}>0$.
By Proposition \ref{prop:when:quasimap-is-map}, if $w>1$, we may assume that $w=3$ because if there exists $w>1$ such that $\mathcal{M}^{\mathrm{stqmaps}}_{p_a,d,\vec{a},\vec{d},w,\iota}\ne\emptyset$ then $\mathcal{M}^{\mathrm{stqmaps}}_{p_a,d,\vec{a},\vec{d},3,\iota}\cong \mathcal{M}^{\mathrm{stqmaps}}_{p_a,d,\vec{a},\vec{d},w,\iota}$. Note that if $\mathcal{M}^{\mathrm{stqmaps}}_{p_a,d,\vec{a},\vec{d},3,\iota}=\emptyset$, there is nothing to prove.
For any object $q\colon (C,\vec{a}\cdot \vec{D})\to  [\mathrm{Cone}(X)/\mathbb{G}_m]$ with the associated line bundle $\mathscr{L}$, if $w=3$, then there exists $M\in \mathbb{Z}_{>0}$ depending only on $p_a,d,\vec{a},\vec{d},w$ such that  $M(K_C+\vec{a}\cdot \vec{D}+w\mathscr{L})$ is a very ample line bundle with $H^1(C,\mathcal{O}_C(mM(K_C+\vec{a}\cdot \vec{D}+w\mathscr{L})))=0$ for any $m\in\mathbb{Z}_{>0}$ by the boundedness of the stable maps (cf.~\cite[Theorem 1.9]{Al-Guy}).
When $w\le1$, the log pair $(C,\vec{a}\cdot \vec{D}+w\mathrm{div}(f))$ for any general $f\in \mathrm{Sec}_q(\mathscr{L})$ is bounded by \cite{hacon2018boundedness}.
Therefore, there also exists $M\in \mathbb{Z}_{>0}$ depending only on $p_a,d,\vec{a},\vec{d},w$ such that  $M(K_C+\vec{a}\cdot \vec{D}+w\mathscr{L})$ is a very ample line bundle with $H^1(C,\mathcal{O}_C(mM(K_C+\vec{a}\cdot \vec{D}+w\mathscr{L})))=0$ for any $m\in\mathbb{Z}_{>0}$ in this case.
Furthermore, in both cases, there exists $l_1\in\mathbb{Z}_{>0}$ depending only $M$ such that $\mathscr{L}$ is $l_1$-regular with respect to $M(K_C+\vec{a}\cdot \vec{D}+w\mathscr{L})$.
Note that $P(m):=\chi(mM(K_C+\vec{a}\cdot \vec{D}+w\mathscr{L}))=1-p_a+mM(2p_a-2)+mM\vec{a}\cdot\vec{d}+mMwd$.

Consider the Hilbert scheme $H$ parameterizing closed subschemes of $\mathbb{P}^{P(1)-1}$ with Hilbert polynomial $P$.
By the standard argument, there exists the largest open subset $U_1\subset H$ that parameterizes only connected nodal curves with the arithmetic genus $p_a$.
Let $\pi\colon\mathcal{U}\to U_1$ be the universal family and $\mathcal{O}(k)$ denote the pullback of $\mathcal{O}_{\mathbb{P}^{P(1)-1}_{U_1}}(k)$ for any $k\in\mathbb{Z}$.
Consider $U_2:=\bigtimes_{U_1}^{1\le j\le r}\mathrm{KDiv}^{d_j}_{\mathcal{U}/U_1}$.
Moreover, consider $\mathrm{Quot}^{\mathcal{O}(-l_1)^{\oplus 1-p_a+d+l_1M(2p_a-2+\vec{a}\cdot\vec{d}+wd)}}_{\mathcal{U}\times_{U_1}U_2/U_2}$ and let $\mathscr{F}$ be the universal quotient sheaf of $\mathcal{O}(-l_1)^{\oplus 1-p_a+d+l_1M(2p_a-2+\vec{a}\cdot\vec{d}+wd)}$.
Take an open subset $U_3\subset \mathrm{Quot}_{\mathcal{U}\times_{U_1}U_2/U_2}$ that is the locus over which $\mathscr{F}_s$ is a line bundle of degree $d$ that is $l_1$-regular with respect to $\mathcal{O}(1)$ if and only if $s\in U_3$. Note that $U_3$ is of finite type.
By the same argument in \cite[Section 3.2]{HH2}, we take a locally closed subscheme $U_4\subset U_3$ such that for any morphism $g\colon T\to U_3$ from a scheme, $g$ factors through $U_4$ if and only if $g$ satisfies the following:
\begin{itemize}
\item if $\mathcal{A}:=\mathcal{O}_{\mathcal{U}_{U_3}}(M(K_{\mathcal{U}_{U_3}/U_3}+\vec{a}\cdot \vec{\mathcal{D}}+w\mathscr{F}))$, then $\mathcal{A}_T\sim_T\mathcal{O}_{\mathcal{U}_T}(1)$, where $\vec{\mathcal{D}}=(D_{1},\ldots, D_{r})$ and $D_k$ is the pullback of the universal relative Mumford divisor on $\mathrm{KDiv}^{d_k}_{\mathcal{U}/U_1}\times_{U_1}\mathcal{U}$, and 
    \item $\pi_{\mathcal{U},T*}\mathcal{A}\otimes\kappa(t)\to H^0(\mathcal{A}_{t})$ and the canonical morphism $\kappa(t)^{\oplus 1-p_a+d+l_1M(2p_a-2+\vec a\cdot\vec d+wd)}\to H^0(\mathscr{F}_t(l_1))$ induced by $\mathcal{O}(-l_1)^{\oplus 1-p_a+d+l_1M(2p_a-2+ \vec{a}\cdot \vec{d}+wd)}\to \mathscr{F}$ are surjective for any $t\in T$.
\end{itemize} 
Next, we consider $\mathcal{H}:=\mathbf{Hom}_{U_{4}}(\mathcal{O}_{\mathcal{U}_{U_{4}}}^{N+1},\mathscr{F}_{U_4})$.
Let $g\in\mathfrak{Hom}(\mathcal{O}_{\mathcal{U}_{\mathcal{H}}}^{N+1},\mathscr{F}_{\mathcal{H}})(\mathcal{H})$ be the universal homomorphism.
We note that $g$ corresponds to a global section $(f_0,\ldots,f_N)\in H^0(\mathcal{H},\pi_{\mathcal{U}_{\mathcal{H}}*}\mathscr{F}_\mathcal{H})^{\oplus N+1}$.
Let $U_5\subset \mathcal{H}$ be an open subscheme such that for any $s\in \mathcal{H}$, $s$ is contained in $U_5$ if and only if $(\mathcal{U}_{\bar{s}},\vec{a}\cdot \vec{\mathcal{D}}_{\bar{s}}+w\mathrm{div}(f_{\bar{s}}))$ is slc for some $(b_i)_{i=0}^N\in\overline{\kappa(s)}^{N+1}$, where $f_{\bar{s}}=\sum_{i=0}^N b_if_{i,\bar{s}}$, or $(\mathcal{U}_{\bar{s}},\vec{a}\cdot \vec{\mathcal{D}}_{\bar{s}})$ is slc and $g_{\bar{s}}$ is surjective.
Furthermore, as in the discussion of \cite[Section 3.2]{HH2}, there exists a closed subscheme $U_6\subset U_5$ such that for any morphism $\mu\colon T\to U_5$, $\mu$ factors through $U_6$ if and only if $g_{\mathcal{U}_T}$ factors through $\mathrm{Cone}(X)\times T\subset \mathbb{A}^{N+1}_{T}$.
Then it is easy to check that $$\mathcal{M}^{\mathrm{stqmaps}}_{p_a,d,\vec{a},\vec{d},w,\iota}\cong [U_6/\PGL(P(1))\times \PGL(1-p_a+d+l_1M(2p_a-2+\vec{a}\cdot\vec{d}+wd))].$$
Therefore, we see that $\mathcal{M}^{\mathrm{stqmaps}}_{p_a,d,\vec{a},\vec{d},w,\iota}$ is an Artin stack of finite type with affine diagonal.

Next, we show that $\mathcal{M}^{\mathrm{stqmaps}}_{p_a,d,\vec{a},\vec{d},w,\iota}$ is a proper Deligne-Mumford stack. Since $\mathcal{M}^{\mathrm{stqmaps}}_{p_a,d,\vec{a},\vec{d},w,\iota}$ is a closed substack of $\mathcal{M}^{\mathrm{stqmaps}}_{p_a,d,\vec{a},\vec{d},w,\mathrm{id}_{\P^N}}$ as the above argument, we may assume that $\iota=\mathrm{id}_{\P^N}$. To see that $\mathcal{M}^{\mathrm{stqmaps}}_{p_a,d,\vec{a},\vec{d},w,\iota}$ is Deligne-Mumford, it suffices to show that $\mathcal{M}^{\mathrm{stqmaps}}_{p_a,d,\vec{a},\vec{d},w,\iota}$ has the proper diagonal because $\mathcal{M}^{\mathrm{stqmaps}}_{p_a,d,\vec{a},\vec{d},w,\iota}$ has the affine diagonal.
First, we deal with the case where $w=3$.
In this case, the same argument as \cite[Theorem 1.12]{Al-Guy} shows that $\mathcal{M}^{\mathrm{stqmaps}}_{p_a,d,\vec{a},\vec{d},w,\iota}$ is proper and has the proper diagonal.
Next, we deal with the case where $w\le1$.
In this case, we show the following claim.
\begin{claim}
When $w\le1$ and $X=\mathbb{P}^N$, let $R$ be a DVR essentially of finite type over $\mathbb{K}$ with the fractional field $K$.   
Let $q^\circ\colon (\mathcal{C}^\circ,\Delta^\circ)\to [\mathrm{Cone}(X)_K/\mathbb{G}_{m,K}]$ be an object of $\mathcal{M}^{\mathrm{stqmaps}}_{p_a,d,\vec{a},\vec{d},w,\iota}(K)$.
Then there exists a finite extension $R\subset R'$, where $R'$ is a Dedekind domain with the fractional field $K'$ such that there exists a unique extension $q_{R'}\in \mathcal{M}^{\mathrm{stqmaps}}_{p_a,d,\vec{a},\vec{d},w,\iota}(R')$ of $q^\circ_{K'}$.
\end{claim}

\begin{proof}[Proof of Claim]
Suppose that there are two extensions $q_{R'}\colon (\mathcal{C},\Delta)\to [\mathrm{Cone}(X)_{R'}/\mathbb{G}_{m,R'}]$ and $q'_{R'}\colon (\mathcal{C}',\Delta')\to [\mathrm{Cone}(X)_{R'}/\mathbb{G}_{m,R'}]$ as the assertion.
Then for any general element of $e\in\mathbb{A}^{N+1}_{\mathbb{K}}$, consider the corresponding sections $f\in\mathrm{Sec}_{q_{R'}}(\mathscr{L})$ and $f'\in\mathrm{Sec}_{q'_{R'}}(\mathscr{L}')$, where $\mathscr{L}$ and $\mathscr{L}'$ are the line bundles associated with $q_{R'}$ and $q'_{R'}$, respectively.
By \cite[Proposition 2.50]{kollar-moduli}, we see that there exists an isomorphism $\varphi\colon(\mathcal{C},\Delta+w\mathrm{div}(f))\to (\mathcal{C}',\Delta'+w\mathrm{div}(f'))$ over $R'$.
For any $e'\in\mathbb{A}^{N+1}_{\mathbb{K}}$, consider the corresponding sections $g\in\mathrm{Sec}_{q_{R'}}(\mathscr{L})$ and $g'\in\mathrm{Sec}_{q'_{R'}}(\mathscr{L}')$.
Via $\varphi$, we see that $\frac{g}{f}$ the rational function can be identified with $\frac{g'}{f'}$ because they are determined by the restrictions to $\mathcal{C}^\circ_{K'}$.
Therefore, $q_{R'}\cong q'_{R'}$ by $\varphi$.

We next deal with the existence part.
When $\vec{a}=(1,1,\ldots,1)$, this is shown in \cite[Theorem 2.12]{Toda} (this only deals with the case where $\mathbb{K}=\mathbb{C}$ but the same argument works for any algebraically closed field of characteristic zero).
For the general case, considering the irreducible decomposition $\sum b_i D'^\circ_i$ of $w\mathrm{Bs}(q^\circ)+\Delta^\circ$. By the slc condition, $0\le b_i\le 1$.
Let $\varphi^\circ\colon \mathcal{C}^\circ\to X$ be the morphism defined by $q^\circ$ subtracting $\mathrm{Bs}(q^\circ)$.
Note that $\varphi^\circ\colon (\mathcal{C}^\circ,\sum D'^\circ_i)\to [\mathrm{Cone}(X)_K/\mathbb{G}_{m,K}]$ is a stable map.
There exist a finite extension $R\subset R'$ of Dedekind domains and a family of stable quasimaps $\overline{\varphi}\colon (\overline{\mathcal{C}},\sum \overline{D}'_i)\to [\mathbb{A}^{N+1}_{R'}/\mathbb{G}_{m,R'}]$, which is the extension of $\varphi^\circ_{K'}$ with the boundary divisor $\sum (D'^\circ_i)_{K'}$ by \cite[Theorem 2.12]{Toda}.
For any general $f\in \mathrm{Sec}_{\overline{\varphi}}(\overline{\mathscr{L}})$ that corresponds to $e\in\mathbb{A}^{N+1}$, where $\overline{\mathscr{L}}$ is the associated line bundle with $\overline{\varphi}$, $(\overline{\mathcal{C}},w\mathrm{div}(f)+\sum b_i\overline{(D'^\circ_i)_{K'}})$ is a stable family over $R'$, where $\overline{(D'^\circ_i)_{K'}}$ is the Zariski closure.
Replacing $R'$ with its finite cover, we may assume that $\overline{D}'_i$'s are sections over $\mathrm{Spec}\,R'$.
By \cite[Theorem 4.1]{Hassett}, there exist a stable family $(\mathcal{C}',wD_f+\sum b_i D'_i)$ over $R'$ that is an extension of $(\mathcal{C}_{K'},\sum b_i (D'^\circ_i)_{K'}+w\mathrm{div}(f)_{K'})$ and a projective dominant morphism $\psi\colon \overline{\mathcal{C}}\to \mathcal{C}'$ over $R'$ such that $\psi_*\mathrm{div}(f)=D_f$ and $D'_i=\psi_*\overline{(D'^\circ_i)_{K'}}$.
Set $\mathscr{L}':=\mathcal{O}_{\mathcal{C}'}(D_f+\psi_*\overline{(\mathrm{Bs}(q^\circ)_{K'})})$ and take $f'\in \Gamma(\mathcal{C}',\mathscr{L}')$ as a section corresponding to $D_f+\psi_*\overline{(\mathrm{Bs}(q^\circ)_{K'})}$. 
For any $g\in \mathrm{Sec}_{\overline{\varphi}}(\overline{\mathscr{L}})$ that corresponds to $e'\in\mathbb{A}^{N+1}$, we set $g':=\frac{g}{f}f'$.
Note that $\frac{g}{f}$ is a rational function on $\overline{\mathcal{C}}$ but we can regard this as a rational function on $\mathcal{C}'$ and thus we can regard $g$ as a rational section of $\mathscr{L}'$.
Since $g'$ has no pole in codimension one in $\mathcal{C}'$ and $\mathcal{C}'$ is $S_2$, we see that $g'\in \Gamma(\mathcal{C}',\mathscr{L}')$.
By using such $g'$, we can yield a structure of a family of quasimaps $q'\colon (\mathcal{C}',\psi_*(\overline{\Delta^\circ_{K'}}))\to[\mathbb{A}^{N+1}_{R'}/\mathbb{G}_{m,R'}]$ and it is easy to see that $q'$ is stable and an extension of $q^\circ_{K'}$ by the stability of $(\mathcal{C}',wD_f+\sum b_i D'_i)$ over $R'$.
Moreover, since $q'_{K'}$ factors through $[\mathrm{Cone}(X)_{K'}/\mathbb{G}_{m,K'}]$, we see that $q'$ factors through $[\mathrm{Cone}(X)_{R'}/\mathbb{G}_{m,R'}]$.
\end{proof}

Therefore, $\mathcal{M}^{\mathrm{stqmaps}}_{p_a,d,\vec{a},\vec{d},w,\iota}$ is a proper Deligne-Mumford stack (cf.~\cite[Lemma A.11]{AHLH}) and has the proper coarse moduli space ${M}^{\mathrm{stqmaps}}_{p_a,d,\vec{a},\vec{d},w,\iota}$ by \cite{keel1997quotients}.
\end{proof}

\begin{defn}
Let $q\colon(\mathcal{C},\Delta)\to[\mathrm{Cone}(X)_S/\mathbb{G}_{m,S}] \in \mathcal{M}^{\mathrm{stqmaps}}_{p_a,d,\vec{a},\vec{d},w,\iota}(S)$ and suppose that $\vec{a}$ is rational and $w\in\mathbb{Q}$.   
Let $\pi\colon\mathcal{C}\to S$ be the canonical morphism and $\mathscr{L}$ the associated line bundle with $q$.
By the Knudsen-Mumford expansion, recall that we set $\lambda_{\mathrm{CM},q}$ the CM line bundle of the family of stable quasimaps $q$.
As \cite[Definition 6.9]{CM26}, we see that the CM line bundle can be defined as $\lambda_{\mathrm{CM}}^{\mathrm{KSBA}}$ over the moduli stack $\mathcal{M}^{\mathrm{stqmaps}}_{p_a,d,\vec{a},\vec{d},w,\iota}$ in a functorial way.
Furthermore, $\lambda_{\mathrm{CM}}^{\mathrm{KSBA}}$ can be descended to  ${M}^{\mathrm{stqmaps}}_{p_a,d,\vec{a},\vec{d},w,\iota}$ as a $\mathbb{Q}$-line bundle $\Lambda_{\mathrm{CM}}^{\mathrm{KSBA}}$ by \cite[Theorem 10.3]{Alper}.
\end{defn}

\begin{thm}
\label{thm:projectivity-stqmaps}
    In the above situation, $\Lambda_{\mathrm{CM}}^{\mathrm{KSBA}}$ is ample. In particular, ${M}^{\mathrm{stqmaps}}_{p_a,d,\vec{a},\vec{d},w,\iota}$ is projective in any case.
\end{thm}

\begin{proof}
As the proof of Theorem \ref{thm:proj--CM--bpCY}, for any irreducible closed subspace $V\subset {M}^{\mathrm{stqmaps}}_{p_a,d,\vec{a},\vec{d},w,\iota}$, it suffices to show that $\Lambda_{\mathrm{CM}}^{\mathrm{KSBA}}|_V$ is nef and $(\Lambda_{\mathrm{CM}}^{\mathrm{KSBA}}|_V)^{\mathrm{dim}\,V}>0$ by \cite[Theorem 3.11]{kollar-proj} and Theorem \ref{thm:constksbaqmaps}.   
By \cite[Theorem 7.4.1]{olsson2016algebraic}, for any generically finite morphism $g\colon S\to V$ from a projective smooth variety, it suffices to show that $g^*\Lambda_{\mathrm{CM}}^{\mathrm{KSBA}}$ is big and nef.
By \cite{bejleri2025proper}, we may assume that $g$ factors through $\mathcal{M}^{\mathrm{stqmaps}}_{p_a,d,\vec{a},\vec{d},w,\iota}$.
Let $q\colon (\mathcal{C},\Delta)\to [\mathrm{Cone}(X)_S/\mathbb{G}_{m,S}]$ be the corresponding object.

If $w\le 1$, the same argument as \cite[Proposition 6.15]{CM26} shows the claim (in the argument, we assumed that $\Delta=0$ in \cite{CM26} but the argument also works for this case).
Therefore, it suffices to check the case where $w>1$.
Since the canonical morphism ${M}^{\mathrm{stqmaps}}_{p_a,d,\vec{a},\vec{d},w,\iota}\to {M}^{\mathrm{stqmaps}}_{p_a,d,\vec{a},\vec{d},w,\mathrm{id}_{\P^N}}$ is a closed immersion, it suffices to check the case where $X=\P^N$.
Since $w>1$, we see that $q$ factors through $\P_S^N$. 
Let $\nu\colon \mathbb{P}^N\hookrightarrow \P^W$ be the $\lceil w\rceil$-th Veronese embedding and note that $\nu_S\circ q$ corresponds to a morphism $S\to \mathcal{M}^{\mathrm{stqmaps}}_{p_a,\lceil w\rceil d,\vec{a},\vec{d},\frac{w}{\lceil w\rceil},\nu}$.
Note also that the corresponding morphism $g'\colon S\to {M}^{\mathrm{stqmaps}}_{p_a,\lceil w\rceil d,\vec{a},\vec{d},\frac{w}{\lceil w\rceil},\nu}$ is generically finite to the image.
Let $\Lambda_{\mathrm{CM}}^{\mathrm{KSBA}'}$ be the corresponding CM line bundle on ${M}^{\mathrm{stqmaps}}_{p_a,\lceil w\rceil d,\vec{a},\vec{d},\frac{w}{\lceil w\rceil},\nu}$ and by what we have shown in the case where $w\le 1$, $g'^*\Lambda_{\mathrm{CM}}^{\mathrm{KSBA}'}$ is big and nef.
It is easy to see that $g'^*\Lambda_{\mathrm{CM}}^{\mathrm{KSBA}'}\equiv g^*\Lambda_{\mathrm{CM}}^{\mathrm{KSBA}}$ by definition and we obtain that $g^*\Lambda_{\mathrm{CM}}^{\mathrm{KSBA}}$ is big and nef.

Finally, the last assertion follows from Lemma \ref{lem-wall-same:st} and Propositions \ref{prop:when:quasimap-is-map} and \ref{prop:chmber--st}. 
\end{proof}

Next, we state the following wall crossing phenomena for stable quasimaps, which is a generalization of \cite[Theorems 3.1 and 3.2]{Al-Guy}.

\begin{thm}\label{thm:st-qmaps-wall-crossing}
For any $w\le w'$ and $a_i\le b_i$, there exists a morphism $\mathcal{M}^{\mathrm{stqmaps}}_{p_a,d,\vec{b},\vec{d},w',\iota}\to\mathcal{M}^{\mathrm{stqmaps}}_{p_a,d,\vec{a},\vec{d},w,\iota}$ if $2p_a-2+dw+\vec{a}\cdot\vec{d}>0$.    
\end{thm}

\begin{proof}
    As the proof of \cite[Theorems 3.1 and 3.2]{Al-Guy}, we may assume that $(\vec{b},w')$ and $(\vec{a},w)$ belong to the closure of the same face $F$ of some $R_l$ considering the segment connecting $(\vec{b},w')$ and $(\vec{a},w)$.
    Moreover, we may assume that $(\vec{b},w')$ and $(\vec{a},w)$ are both rational (cf.~Propositions \ref{prop:when:quasimap-is-map} and \ref{prop:chmber--st}).
    For any object of $q\colon (\mathcal{C},\vec{b}\cdot\vec{\mathcal{D}})\to [\mathrm{Cone}(X)_S/\mathbb{G}_{m,S}]$ of $\mathcal{M}^{\mathrm{stqmaps}}_{p_a,d,\vec{b},\vec{d},w',\iota}(S)$ and scheme $S$ of finite type over $\mathbb{K}$, we see that $q\colon (\mathcal{C},\vec{a}\cdot\vec{\mathcal{D}})\to [\mathrm{Cone}(X)_S/\mathbb{G}_{m,S}]$ with weight $w$ satisfies that $\omega_{\mathcal{C}/S}(\vec{a}\cdot\vec{\mathcal{D}}+w\mathscr{L})$ is relatively nef over $S$ but for any $s\in S$, there exists an irreducible component $F\subset \mathcal{C}_s$ such that $\omega_{\mathcal{C}/S}(\vec{a}\cdot\vec{\mathcal{D}}+w\mathscr{L})|_F$ has a positive degree.
    By \cite[Proposition 3.3]{Hassett}, we see that there exists $M\in\mathbb{Z}_{>0}$ such that $\omega^{\otimes M}_{\mathcal{C}/S}(M(\vec{a}\cdot\vec{\mathcal{D}}+w\mathscr{L}))$ is a relatively globally generated line bundle with $H^1(\mathcal{C}_s,\omega^{\otimes M}_{\mathcal{C}_s}(M(\vec{a}\cdot\vec{\mathcal{D}}_s+w\mathscr{L}_s)))=0$ such that $|\omega^{\otimes M}_{\mathcal{C}_s}(M(\vec{a}\cdot\vec{\mathcal{D}}_s+w\mathscr{L}_s))|$ defines a contraction of nodal curves for any $s\in S$.
    Set $\mathcal{C}'$ as $\mathbf{Proj}_S(\oplus_{m\ge0}\pi_*(\omega^{\otimes mM}_{\mathcal{C}/S}(mM(\vec{a}\cdot\vec{\mathcal{D}}+w\mathscr{L}))))$, where $\pi\colon \mathcal{C}\to S$ is the canonical morphism.
    Let $h\colon \mathcal{C}\to \mathcal{C}'$ be the canonical morphism and note that $\mathcal{C}'$ is flat over $S$.
    For any geometric point $\bar{s}\in S$, we see that $h_{\bar{s}}$ is a contraction.
    By Proposition \ref{prop:pushfoward}, we set $\vec{\mathcal{D}}':=h_*\vec{\mathcal{D}}$.

    If $w=0$, we obtain a canonical morphism $\mathcal{M}^{\mathrm{stqmaps}}_{p_a,d,\vec{b},\vec{d},w',\iota}\to\mathcal{M}^{\mathrm{stqmaps}}_{p_a,d,\vec{a},\vec{d},w,\iota}$, forgetting the quasimap structure unless $w'=0$.
In this case, we may only deal with the case where $w'<1$ and the open segment whose boundaries are $(\vec{b},w')$ and $(\vec{a},0)$ lies in the same face because we will deal with the general case later.
For any $q\in \mathcal{M}^{\mathrm{stqmaps}}_{p_a,d,\vec{b},\vec{d},w',\iota}(S)$, shrinking $S$ if necessary, we can take a general section $f\in \mathrm{Sec}_q(\mathscr{L})$ corresponding to an element of $\mathbb{K}^{N+1}$ such that $(\mathcal{C},(t\vec{b}+(1-t)\vec{a})\cdot\vec{\mathcal{D}}+tw'\mathrm{div}(f))$ is a stable family for any $t\in (0,1]$.
Then, we can check that $(\mathcal{C}',h_*(\vec{a}\cdot\vec{\mathcal{D}}))$ is stable over $S$ and belongs to $\mathcal{M}^{\mathrm{stqmaps}}_{p_a,d,\vec{a},\vec{d},0,\iota}(S)$. This shows the existence of the canonical morphism of the moduli stacks in this case.

    We will show that $(\mathcal{C}',\vec{a}\cdot\vec{\mathcal{D}}')$ has a canonical structure of a stable family of weighted log quasimaps induced by $q$ under the assumption that $w\ne0$.
 To show this, we first put the structure of a locally stable family of weighted quasimaps $q'\colon(\mathcal{C}',\vec{a}\cdot\vec{\mathcal{D}}')\to [\mathbb{A}^{N+1}_S/\mathbb{G}_{m,S}]$ by Proposition \ref{prop:contraction-thm}.
     As the proof of Theorem \ref{thm--wall-crossing-bpcy}, it is not hard to check that $q'$ factors through $[\mathrm{Cone}(X)_S/\mathbb{G}_{m,S}]$.
     Moreover, we know that $h_*\mathscr{L}$ can be regarded as the line bundle associated with $q'$ and $K_{\mathcal{C}'/S}+h_*(\vec{a}\cdot\vec{\mathcal{D}'})+wh_*\mathscr{L}$ is relatively ample over $S$. 
     Therefore, $q'\in \mathcal{M}^{\mathrm{stqmaps}}_{p_a,d,\vec{a},\vec{d},w,\iota}(S)$.
     Since this construction is functorial, we can check that this induces a morphism $\mathcal{M}^{\mathrm{stqmaps}}_{p_a,d,\vec{b},\vec{d},w',\iota}\to\mathcal{M}^{\mathrm{stqmaps}}_{p_a,d,\vec{a},\vec{d},w,\iota}$.
\end{proof}

Finally, we state the wall crossing phenomena when $p_a=0$ to the moduli of boundary polarized CY quasimaps.

\begin{thm}\label{thm:final-wall-crossing-st-kst}
 Let $d,d^{(j)}_i\in\mathbb{Z}_{>0}$, $w,w'\in\mathbb{R}_{\ge0}$ and $a^{(j)}_i\ge0$ for $i=1,2$ and $1\le j\le r_i$ with $(w+w')d+\vec{a}_1\cdot\vec{d}_1+\vec{a}_2\cdot\vec{d}_2=2$, where we set $\vec{d}_i:=(d^{(j)}_i)_{j=1}^{r_i}$, $\vec{a}_i:=(a^{(j)}_i)_{j=1}^{r_i}$ and  $r_1,r_2\in\mathbb{Z}_{\ge0}$.
 Suppose that $w'\ne0$ or $\vec{a}_2\ne\vec0$.
 Let $\vec{d}=(\vec{d}_1,\vec{d}_2)$, $\vec{a}=(\vec{a}_1,\vec{a}_2)$ and $\vec{b}_{t}=(\vec{a}_1,(1+t)\vec{a}_2)$ for any $t\in\mathbb{R}$.

 Then, there exists a canonical morphism $\mathcal{M}^{\mathrm{stqmaps}}_{0,d,\vec{b}_{\epsilon},\vec{d},w+(1+\epsilon)w',\iota}\to\mathcal M^{\mathrm{bpcyqmaps}}_{
  d,\vec a_1,\vec d_1,\vec a_2,\vec d_2,w,w',\iota}$ for any sufficiently small $\epsilon>0$ and this is an open immersion.  
 Furthermore, there exists a canonical morphism $({M}^{\mathrm{stqmaps}}_{0,d,\vec{b}_{\epsilon},\vec{d},w+(1+\epsilon)w',\iota})^{\mathrm{sn}}\to ({M}^{\mathrm{Kpsqmaps}}_{d,(1-\epsilon)(\vec{a}_1,\vec{a}_2),\vec{d},(1-\epsilon)(w+w'),\iota})^{\mathrm{sn}}$. 
\end{thm}

\begin{proof}
     We may assume that $\vec{b}_{\epsilon},\vec{a},w,w'$ are all rational and  $\mathcal{M}^{\mathrm{stqmaps}}_{0,d,\vec{b}_{\epsilon},\vec{d},w+(1+\epsilon)w',\iota}$ is independent of the choice of any sufficiently small $\epsilon>0$ by Lemmas \ref{lem-wall-same:bdcy} and \ref{lem-wall-same:st}.
    For any object of $q\colon (\mathcal{C},\vec{a}_1\cdot\vec{\mathcal{D}}_1+(1+\epsilon)\vec{a}_2\cdot\vec{\mathcal{D}_2})\to [\mathrm{Cone}(X)_S/\mathbb{G}_{m,S}]$ of $\mathcal{M}^{\mathrm{stqmaps}}_{0,d,\vec{b}_{\epsilon},\vec{d},w+(1+\epsilon)w',\iota}(S)$ and scheme $S$ of finite type over $\mathbb{K}$, we see that $q\colon (\mathcal{C},\vec{a}_1\cdot\vec{\mathcal{D}}_1+\vec{a}_2\cdot\vec{\mathcal{D}_2})\to [\mathrm{Cone}(X)_S/\mathbb{G}_{m,S}]$ with weight $w+w'$ satisfies that  for any geometric point $\bar{s}\in S$ and irreducible component $F\subset \mathcal{C}_{\bar{s}}$, $\mathrm{deg}_F(\omega_{\mathcal{C}/S}(\vec{a}_1\cdot\vec{\mathcal{D}}_1+\vec{a}_2\cdot\vec{\mathcal{D}_2}+(w+w')\mathscr{L})|_F)=0$ by nefness.
    Therefore, $\vec{a}_2\cdot\vec{\mathcal{D}_2}+w'\mathscr{L}$ is relatively ample over $S$.
    This means that we can regard $q$ as a boundary polarized CY quasimaps with $\vec{a}_2\cdot\vec{\mathcal{D}_2}+w'\mathscr{L}$ the polarization part.
By the proof of Theorem \ref{thm:constksbaqmaps}, we see that the slc condition for $q_{s}\colon (\mathcal{C}_s,(\vec{a}_1\cdot\vec{\mathcal{D}}_1)_s+(1+\epsilon)(\vec{a}_2\cdot\vec{\mathcal{D}_2})_s)\to [\mathrm{Cone}(X)_s/\mathbb{G}_{m,s}]$ with the weight $w+(1+\epsilon)w'$ is open. Therefore, we see that the morphism $\mathcal{M}^{\mathrm{stqmaps}}_{0,d,\vec{b}_{\epsilon},\vec{d},w+(1+\epsilon)w',\iota}\to\mathcal M^{\mathrm{bpcyqmaps}}_{
  d,\vec a_1,\vec d_1,\vec a_2,\vec d_2,w,w',\iota}$ is an open immersion.
    
    Finally, we deal with the last assertion.
    This follows from Theorems \ref{thm:st-qmaps-wall-crossing} and \ref{thm:wall-crossingbpcytoK-moduli} and the property of the seminormalization. 
\end{proof}

\begin{rem}\label{rem:final}
In Theorem \ref{thm:final-wall-crossing-st-kst}, we note that if $\iota=\mathrm{id}_{\P^N}$, then $\mathcal{M}^{\mathrm{stqmaps}}_{0,d,\vec{b}_{\epsilon},\vec{d},w+(1+\epsilon)w',\iota}$ is smooth by Proposition \ref{prop:smoothness-of-stacks}.    
This can be regarded as an analogue of \cite[Theorem 2]{FP}.
Therefore, without taking the seminormalization, there exists a canonical morphism ${M}^{\mathrm{stqmaps}}_{0,d,\vec{b}_{\epsilon},\vec{d},w+(1+\epsilon)w',\mathrm{id}_{\P^N}}\to {M}^{\mathrm{Kpsqmaps}}_{d,(1-\epsilon)(\vec{a}_1,\vec{a}_2),\vec{d},(1-\epsilon)(w+w'),\mathrm{id}_{\P^N}}$.
\end{rem}

\bibliographystyle{alpha}
\bibliography{main}

\end{document}